\documentclass[12pt]{amsart}
\usepackage{currvita, amssymb, amsmath, amsthm, 
graphicx, subfigure, pifont, wrapfig, pinlabel, amscd,
latexsym ,exscale, enumerate, amsfonts, times, color, hyphenat, pinlabel, 
mathtools, a4wide, fullpage, array}
\usepackage{float}

\usepackage[table]{xcolor}

\definecolor{myorange}{RGB}{225,127,0}
\definecolor{mygreen}{RGB}{0,225,0}
\definecolor{mypurple}{RGB}{128,0,128}
\definecolor{myred}{RGB}{225,0,0}
\definecolor{myblue}{RGB}{0,0,225}
\definecolor{myyellow}{RGB}{210,210,0}
\definecolor{mycream}{RGB}{245,243,198}

\definecolor{dummy}{RGB}{10,10,10}
\definecolor{mygray}{gray}{0.7}

\definecolor{orchid}{RGB}{143,40,194}
\definecolor{lava}{RGB}{207,16,32}

\definecolor{mydarkblue}{RGB}{10,10,170}

\usepackage[hypertexnames=false]{hyperref}
\hypersetup{
    pdftoolbar=true,        
    pdfmenubar=true,        
    pdffitwindow=false,     
    pdfstartview={FitH},    
    pdftitle={The $\mathfrak{sl}_3$-web algebra},    
    pdfauthor={Marco Mackaay, Weiwei Pan and Daniel Tubbenhauer},     
    pdfsubject={},   
    pdfcreator={Marco Mackaay, Weiwei Pan and Daniel Tubbenhauer},   
    pdfproducer={Marco Mackaay, Weiwei Pan and Daniel Tubbenhauer}, 
    pdfkeywords={}, 
    pdfnewwindow=true,      
    colorlinks=true,       
    linkcolor=mydarkblue,          
    citecolor=teal,        
    filecolor=magenta,      
    urlcolor=orchid,          
    linkbordercolor=lava,
    citebordercolor=teal,
    urlbordercolor=orchid,  
    linktocpage=true
}
\usepackage[all]{xy}
\usepackage{epstopdf}

\newcommand{\foamt}{\mathbf{Foam}_{3}}

\newcommand{\llambda}{\overline{\lambda}}

\newcommand{\ii}{\underline{\textbf{\textit{i}}}}
\newcommand{\jj}{\underline{\textbf{\textit{j}}}}

\newcommand{\glcat}{\mathcal{U}(\mathfrak{gl}_n)}

\newcommand{\Ugl}{\dot{\mathbf U}_q(\mathfrak{gl}_n)}

\DeclareMathOperator{\End}{End}
\DeclareMathOperator{\HomGL}{Hom_{\glcat}}
\DeclareMathOperator{\HOMGL}{HOM_{\glcat}}

\newcommand{\bN}{\mathbb{N}}
\newcommand{\bZ}{\mathbb{Z}}
\newcommand{\bQ}{\mathbb{Q}}

\newcommand{\bC}{\mathbb{C}}

\newcommand{\sseq}{{\rm SiSeq}}

\newcommand{\onel}{{\mathbf 1}_{\lambda}}
\newcommand{\onelp}{{\mathbf 1}_{\lambda'}}

\newcommand{\refequal}[1]{\xy {\ar@{=}^{#1}
(-1,0)*{};(1,0)*{}};
\endxy}

\newcommand{\U}{\dot{{\bf U}}_q(\mathfrak{sl}_n)}

\newcommand{\UZ}{\dot{\bf U}_q}

\newcommand{\Ucat}{\mathcal{U}(\mathfrak{sl}_n)}
\newcommand{\Scat}{\mathcal{S}}

\newcommand{\UcatD}{\dot{\mathcal U}(\mathfrak{sl}_n)}
\newcommand{\ScatD}{\dot{\Scat}}

\newcommand{\qbin}[2]{
\left[
 \begin{array}{c}
 #1 \\
 #2 \\
 \end{array}
 \right]
}

\newtheorem{prop}{Proposition}[section]
\newtheorem{thm}[prop]{Theorem}
\newtheorem{lem}[prop]{Lemma}
\newtheorem{cor}[prop]{Corollary}

\theoremstyle{remark}
\newtheorem{rem}[prop]{Remark}
\theoremstyle{remark}
\newtheorem{ex}[prop]{\textbf{Example}}

\theoremstyle{remark}

\theoremstyle{definition}
\newtheorem{defn}[prop]{Definition}

\newtheorem{question}[prop]{Question}

\numberwithin{equation}{section}

\newcommand{\figins}[3] 
{\raisebox{#1pt}{\includegraphics[height=#2 in]{section22/#3}}}

\newcommand{\figwhins}[4] 
{\raisebox{#1pt}{\includegraphics[height=#2 in, width=#3 in]{section22/#4}}}

\newcommand{\Ver}{\mathrm{Vert}}
\newcommand{\F}{\mathcal{F}}

\title{The $\mathfrak{sl}_{3}$-web algebra}
\author{M.~Mackaay, W.~Pan and D.~Tubbenhauer}
\thanks{The first author was supported by the FCT - Funda\c c\~{a}o para a 
Ci\^{e}ncia e a Tecnologia, through project number PTDC/MAT/101503/2008, 
New Geometry and Topology.}
\thanks{The second and the third author were supported by the German Research 
Foundation (Deutsche Forschungsgemeinschaft (DFG)) 
through the Institutional Strategy of the University of G\"{o}ttingen.}

\begin{document}
\begin{abstract}
In this paper we use Kuperberg's $\mathfrak{sl}_3$-webs and Khovanov's 
$\mathfrak{sl}_3$-foams to define a new algebra $K^S$, which we call the 
$\mathfrak{sl}_3$-web algebra. It is the $\mathfrak{sl}_3$ analogue of 
Khovanov's arc algebra. 

We prove that $K^S$ is a graded symmetric Frobenius algebra. 
Furthermore, we categorify an instance of $q$-skew Howe duality, 
which allows us to prove that $K^S$ is Morita equivalent to 
a certain cyclotomic KLR-algebra of level 3. This allows us to determine 
the split Grothendieck group $K^{\oplus}_0(\mathcal{W}^S)_{\bQ(q)}$, to show that its center is 
isomorphic to the cohomology ring of a certain Spaltenstein variety, 
and to prove that $K^S$ is a graded cellular algebra. 
\end{abstract}

\maketitle
 
\tableofcontents


\section{Introduction}
\setcounter{subsection}{1}

In this paper, we define the $\mathfrak{sl}_3$ analogue of Khovanov's arc algebras $H^n$, introduced in~\cite{kh}. We call 
them \textit{web algebras} and denote them by $K^S$, where $S$ is a \textit{sign string} (string of $+$ and $-$ signs). 
Instead of arc diagrams, which give a diagrammatic presentation of 
the representation theory of 
$U_q(\mathfrak{sl}_2)$, we use $\mathfrak{sl}_3$-webs, introduced by Kuperberg~\cite{ku}. These webs give a diagrammatic presentation of the 
representation theory of $U_q(\mathfrak{sl}_3)$. 
Instead of $\mathfrak{sl}_2$-cobordisms, which Bar-Natan 
used~\cite{bn} to give his formulation of Khovanov's link homology, we use Khovanov's~\cite{kv} $\mathfrak{sl}_3$-foams. 
\vspace*{0.25cm}

We prove the following main results regarding $K^S$.
\begin{enumerate}
\item $K^S$ is a graded symmetric Frobenius algebra (Theorem~\ref{thm:frob}).
\item We give an explicit degree preserving algebra isomorphism 
between the cohomology ring of the Spaltenstein variety 
$X^{\lambda}_{\mu}$ and the center $Z(K^S)$ of $K^S$, where $\lambda$ and $\mu$ are 
two weights determined by $S$ (Theorem~\ref{thm:center}). 
\item Let $V^S=V^{s_1}\otimes\cdots\otimes V^{s_n}$, where $V^+$ is the 
basic $U_q(\mathfrak{sl}_3)$-representation and $V^-$ its dual. 
Kuperberg~\cite{ku} proved that $W^S$, the space of $\mathfrak{sl}_3$-webs 
whose boundary is determined by $S$, is isomorphic to 
$\mathrm{Inv}_{U_q(\mathfrak{sl}_3)}(V^S)$, the space of invariant tensors in 
$V^S$. 

Choose an arbitrary 
$k\in\mathbb{N}$ and let $n=3k$. Let us denote by $V_{(3^k)}$ the 
irreducible $U_q(\mathfrak{sl}_n)$-module 
with highest weight $3\omega_k$, where $\omega_k$ is the $k$-th fundamental 
$\mathfrak{sl}_n$-weight. As the reader will have noticed, we actually 
use the corresponding $\mathfrak{gl}_n$-weight $(3^k)$. This is natural from 
the point of view of skew Howe duality, as we will explain in the paper. 

The $\mathfrak{gl}_n$-weights of $V_{(3^k)}$ belong to 
$\Lambda(n,n)_3$, which is the set of $n$-part compositions of $n$ whose 
parts are integers between $0$ and $3$. These weights, denoted $\mu_S$, 
correspond bijectively to \textit{enhanced sign sequences} of length $n$, 
denoted by $S$ as before, and are in bijective correspondence to the 
semi-standard Young tableaux with $k$ rows and 3 columns. 

Define the \textit{web module}
\[
W_{(3^k)}=\bigoplus_{\mu_S\in\Lambda(n,n)_3} W^S,
\]
where $W^S$ is defined as before after deleting the entries of $\mu_S$ 
which are equal to $0$ or $3$. 

By $q$-skew Howe duality, which 
we will explain at the beginning of Section~\ref{sec:grothendieck}, 
there is an $\U$-action on $W_{(3^k)}$ such that 
\[
V_{(3^k)}\cong \bigoplus_{\mu_S\in\Lambda(n,n)_3} W^S
\]
as $\U$-modules. 
\vskip0.5cm
In Section~\ref{sec:grothendieck} we categorify 
this result. Let $R_{(3^k)}$ be the cyclotomic level-three 
Khovanov-Lauda Rouquier algebra (cyclotomic 
KLR algebra for short) with highest $\mathfrak{gl}_n$-weight $(3^k)$, and 
let  
\[
\mathcal{V}_{(3^k)}=R_{(3^k)}\text{-}\mathrm{\textbf{Mod}}_{\mathrm{gr}}\quad\text{and}\quad {}_p\mathcal{V}_{(3^k)}=
R_{(3^k)}\text{-}\mathrm{p\textbf{Mod}}_{\mathrm{gr}}
\]
be its categories of finite dimensional, graded modules and 
finite dimensional, graded, projective modules respectively. 
We define grading shifts by 
\[
M\{t\}_i=M_{i-t}
\]
for any $M\in \mathcal{V}_{(3^k)}$ and $t\in\mathbb{Z}$. 
We denote the split Grothendieck group of ${}_p\mathcal{V}_{(3^k)}$ by 
\[
K^{\oplus}_0({}_p\mathcal{V}_{(3^k)}),
\]
which becomes a $\mathbb{Z}[q,q^{-1}]$-module by defining  
\[
q^t[M]=[M\{t\}]
\] 
for any $M\in {}_p\mathcal{V}_{(3^k)}$ and $t\in\mathbb{Z}$. 
For the rest of this paper, we will always work with  
\[ 
K^{\oplus}_0({}_p\mathcal{V}_{(3^k)})_{\bQ(q)}=K^{\oplus}_0({}_p\mathcal{V}_{(3^k)})
\otimes_{\mathbb{Z}[q,q^{-1}]}{\bQ(q)}.
\]

Brundan and 
Kleshchev~\cite{bk} (see also~\cite{kakash},
~\cite{lv},~\cite{vv} and~\cite{we1}) proved that there is 
a strong $\mathfrak{sl}_n$-$2$-representation on 
$\mathcal{V}_{(3^k)}$, which can be restricted to ${}_p\mathcal{V}_{(3^k)}$  
such that  
\[
K^{\oplus}_0({}_p\mathcal{V}_{(3^k)})_{\bQ(q)}
\cong V_{(3^k)}
\] 
as $\U$-modules. 

We prove (Proposition~\ref{prop:cataction}) that there exists a 
strong $\mathfrak{sl}_n$-$2$-representation on 
\[
\mathcal{W}_{(3^k)}=\bigoplus_{\mu_S\in\Lambda(n,n)_3} \mathcal{W}^{S},
\]
where 
\[
\mathcal{W}^{S}=K^S\text{-}\mathrm{\textbf{Mod}}_{\mathrm{gr}}~\footnote{The idea for this 
$2$-representation was suggested by Mikhail Khovanov to M.~M. in 2008 and its basic 
ideas were worked out modulo 2 in the unpublished preprint~\cite{mack}.}.
\]
This $2$-representation can be restricted to 
\[
{}_p\mathcal{W}_{(3^k)}=\bigoplus_{\mu_S\in\Lambda(n,n)_3} {}_p\mathcal{W}^{S},
\]
where 
\[
{}_p\mathcal{W}^{S}=K^S\text{-}\mathrm{p\textbf{Mod}}_{\mathrm{gr}}.
\] 

By a general result due to Rouquier~\cite{rou}, which 
we recall in Proposition~\ref{prop:rouquier}, we get
\begin{equation}
\label{eq:rouquierequiv}
{}_p\mathcal{V}_{(3^k)}\cong {}_p\mathcal{W}_{(3^k)}.
\end{equation}
\item In particular, 
this proves that the split Grothendieck groups of both categories are 
isomorphic (Corollary~\ref{cor:equivalence}). It follows that we have 
\[
K^{\oplus}_0\left({}_p\mathcal{W}^S\right)_{\bQ(q)}\cong W^S,
\]
for any $S$ such that $\mu_S\in\Lambda(n,n)_3$. 
\item As proved in Corollary~\ref{cor:equivalence}, the equivalence 
in~\eqref{eq:rouquierequiv} implies that 
$R_{(3^k)}$ and 
\[
K_{(3^k)}=\bigoplus_{\mu_S\in\Lambda(n,n)_3} K^S
\] 
are Morita equivalent (Proposition~\ref{prop:morita}), i.e. we have 
\begin{equation}
\label{eq:rouquiermorita}
\mathcal{V}_{(3^k)}\cong \mathcal{W}_{(3^k)}
\end{equation}
as strong $\mathfrak{sl}_n$-$2$-representations. 
\item In Corollary \ref{cor:moritacellular}, we show 
that \eqref{eq:rouquiermorita} and the cellularity of $R_{(3^k)}$, due to Hu and 
Mathas \cite{hm}, imply that $K^S$ is a graded cellular algebra for any $S$. 
\item We also show that~\eqref{eq:rouquierequiv} and Brundan and 
Kleshchev's results in~\cite{bk} imply that the indecomposable 
objects in ${}_p\mathcal{W}^S$, with 
a suitable normalization of their gradings, correspond bijectively 
to the dual canonical basis elements in 
$\mathrm{Inv}(V^S)$ (Theorem~\ref{thm:dualcan}).  
\end{enumerate}
The first result is easy to prove and similar to the case for $H^n$. 
Some of the other results are much harder to prove for $K^S$ 
than their analogues are for $H^n$ 
(e.g. see Remark~\ref{rem:counter2}). In order to prove the second and 
the last result, we introduce a 
``new trick'', i.e. we use a deformation of $K^S$, called $G^S$. This deformation 
is induced by Gornik's~\cite{g} deformation of Khovanov's original 
$\mathfrak{sl}_3$-foam relations. One big difference 
between $G^S$ and $K^S$ is that the former algebra is \textit{filtered} whereas 
the latter is \textit{graded}. As a matter of fact, 
$K^S$ is the associated graded algebra of $G^S$. The usefulness of 
$G^S$ relies on the fact that $G^S$ is semisimple as an algebra, 
i.e. forgetting the filtration (see Proposition~\ref{prop:Gsemisimple}). 
\vskip0.5cm
Let us explain the connection to some of the existing work in the literature. 

We first comment on the relation of our results with known results in the $\mathfrak{sl}_2$ case. 
Khovanov~\cite{kh} introduced the arc algebras $H^n$ in his work on the generalization of 
his celebrated categorification of the Jones link polynomial to tangles. As he showed, 
\[
K_0^{\oplus}(H^n\text{-}\mathrm{p\textbf{Mod}}_{\mathrm{gr}})\cong \mathrm{Inv}_{\dot{\mathbf U}_q(\mathfrak{sl}_2)}(V^{\otimes 2n}),
\]
where $V$ is the fundamental $\dot{\mathbf U}_q(\mathfrak{sl}_2)$-module. 
The Grothendieck classes of the indecomposable graded $H^n$-modules, with a suitable 
normalization of their grading, correspond bijectively to the dual canonical basis elements 
of the invariant tensor space. The proof of these facts is completely elementary and does not 
require any categorified skew Howe duality.  

Huerfano and Khovanov categorified the irreducible level-two 
$\U$-representations with highest weight $2\omega_k$ in~\cite{hkh}, using the arc algebras and 
categorified skew Howe duality (without calling it that explicitly), but without explaining the relation with the level-two 
cyclotomic KLR algebras which had not yet been invented at that time. 

That relation only appeared in the work by Brundan and Stroppel~\cite{bs3}, 
who studied the representation theory of the arc algebras in 
great detail in~\cite{bs},~\cite{bs2},~\cite{bs3},~\cite{bs4} and~\cite{bs5}.  

Khovanov showed that the center of the arc algebra $H^n$ is isomorphic to the 
cohomology ring of the $(n,n)$-Springer variety $X^n$ and Stroppel and Webster 
showed that $H^n$ can be realized using the intersection cohomology of $X^n$.  

The results in this paper are the $\mathfrak{sl}_3$ analogues 
of some of the results in the papers cited above. 

There are several results for $\mathfrak{sl}_2$ for which we have not yet found the $\mathfrak{sl}_3$ 
analogues, e.g. we have not defined the quasi-hereditary cover of $K^S$ in this paper. The quasi-hereditary cover of 
$H^n$ is due to Chen and Khovanov~\cite{ck} and Stroppel~\cite{s} and was studied by Brundan 
and Stroppel in~\cite{bs},\cite{bs2} and~\cite{bs3}. Furthermore, in~\cite{bs4} Brundan and Stroppel found a remarkable 
representation theoretic relation between general linear super groups and certain generalized arc algebras. 
This is another result for which we do not have an $\mathfrak{sl}_3$-analogue.  
\vskip0.5cm
Let us now comment on the connection with other work on categorified $\mathfrak{sl}_n$-representations and 
link homologies, for $n\geq 3$. There are essentially three diagrammatic or 
combinatorial approaches which give 
$\mathfrak{sl}_n$-link homologies (there are other approaches using 
representation theory or algebraic geometry for example, but we will not 
consider those in this introduction).  
\begin{enumerate}
\item There is the approach using matrix factorizations due 
to Khovanov and Rozansky~\cite{kv}, which was proved to be equivalent 
to an approach using foams~\cite{kh},~\cite{kr},~\cite{msv} and ~\cite{mv2}. 
\item There is Webster's approach using cyclotomic tensor algebras, which 
generalize 
the cyclotomic KLR-algebras, see~\cite{we1} and~\cite{we2}. 
\item There is an approach using 
Chuang-Rouquier complexes over cyclotomic quotients~\cite{cr} (its details have 
only been worked out and written up completely for $n=2,3$ in~\cite{lqr}, 
but see our remarks below for the general case. Note that Chuang and Rouquier 
did not prove invariance under the third Reidemeister move nor did they 
discuss braid closures, i.e. knots and links in~\cite{cr})
\end{enumerate} 

Only for $n=2$ and $n=3$ it is known that all three approaches 
give isomorphic link homologies. For $n\geq 4$ they are conjectured 
to be isomorphic, but only the first and the third approaches are known 
to give isomorphic link homologies. Let us explain this in a bit 
more detail.   

The paper by Lauda, Queffelec and Rose~\cite{lqr} appeared online a little 
after our paper became available and is completely independent. 
They used a slightly different type of $\mathfrak{sl}_3$-foams in order to define their version of 
categorified level-three skew Howe duality and used it to relate the first 
and third aforementioned approach to $\mathfrak{sl}_3$-link homologies. 
\vspace*{0.25cm}

Although link homologies have attracted a lot of attention, 
it seems important to us to understand the 
bigger picture of the categorified representation theory behind 
the link homologies. 

In this paper we do not work out the application to link homology, but 
concentrate on the indecomposable projective modules 
and the center of $K^S$. In Proposition~\ref{prop:unitriang} we actually 
prove a conjecture due to Morrison and Nieh~\cite{mn} about the relation 
between the basis webs in $B^S$ and the indecomposables in 
${}_p\mathcal{W}^S$ (see Remark~\ref{rem:mn}). For related results 
in this direction, see the work by Robert in~\cite{rob} and~\cite{rob2}. 
\vspace*{0.25cm}

It is not so hard to generalize our results in this paper 
to the case for $\mathfrak{sl}_n$, with $n\geq 2$, using 
matrix factorizations instead of foams. As a matter of fact, 
this has been done in the meanwhile by Mackaay and 
Yonezawa in~\cite{mack1} and~\cite{my}. 
General $\mathfrak{sl}_n$-foams have not been defined yet (for a partial 
case, see~\cite{msv}), so they cannot be used. But the results in~\cite{my} 
will probably be helpful to define a finite set of relations on 
$\mathfrak{sl}_n$-foams and to prove that these relations are consistent and 
sufficient. 

Having such a definition of $\mathfrak{sl}_n$-foams is important, 
just as it is important to have generators and 
relations for any interesting algebra. In our opinion, the 
categorified skew Howe duality using foams would be a proper 
categorification of Cautis, Kamnitzer and Morrison's results on 
quantum skew Howe duality in~\cite{ckm}. Furthermore, $\mathfrak{sl}_n$-foams 
might be very helpful in computing the 
Khovanov-Rozansky link homologies effectively (i.e. using computers). 

Although Mackaay and Yonezawa did not work out the details, it is 
fairly straightforward to show that their definitions and results in~\cite{my} 
imply that the equivalence between 
level-$n$ cyclotomic KLR algebras and $\mathfrak{sl}_n$-web algebras 
maps the colored Chuang-Rouquier complex of a braid over a cyclotomic KLR 
algebra to the corresponding colored Khovanov-Rozansky complex 
(due to Wu~\cite{wu} and Yonezawa~\cite{yo} in the colored case). This 
would imply the above claim that the third aforementioned construction of  
$\mathfrak{sl}_n$-link homologies works for all $n\geq 2$ and 
that the first and the third construction give isomorphic 
$\mathfrak{sl}_n$-link homologies.  

In Proposition 4.4 in~\cite{we2}, Webster 
proved that his and Khovanov's $\mathfrak{sl}_3$-link homologies are 
isomorphic, but the proof is quite sophisticated and relies on 
Mazorchuk and Stroppel's approach to link homology using functors and 
natural transformations on certain 
blocks of category $\mathcal{O}$~\cite{ms}. 
Our results in this paper might help to give an elementary and direct 
isomorphism between Webster's and Khovanov's $\mathfrak{sl}_3$-link homologies. In 
order to do that, one would have to use yet to be defined bimodules 
over the cyclotomic tensor algebras 
which categorify the $\mathfrak{sl}_3$-webs. 

For $n\geq 4$ Webster conjectured his $\mathfrak{sl}_n$-link homology to 
be isomorphic to Khovanov and Rozansky's, but did not prove it. 
Following the same reasoning as for $n=3$, such a proof should now be within 
reach. 
\vskip0.5cm
Finally, let us mention one interesting open question w.r.t. 
$\mathfrak{sl}_3$-web algebras. 
In~\cite{fkk}, Fontaine, Kamnitzer and Kuperberg studied spiders 
using an algebro-geometric approach. For $\mathfrak{sl}_3$ these 
spiders are exactly the webs in our paper. Given a sign string $S$, 
the \emph{Satake fiber} $F(S)$, denoted 
$F(\overrightarrow{\lambda})$ in~\cite{fkk}, is isomorphic to 
the Spaltenstein variety $X^{\lambda}_{\mu}$ mentioned above. 
Let us point out the difference in these notations 
that otherwise might confuse the reader: the $\lambda$ in~\cite{fkk} is 
equal to $\mu$ in our paper, which is also equivalent to $S$. 
Given a web $w$ with boundary corresponding to $S$, Fontaine, Kamnitzer and 
Kuperberg also defined a 
variety $Q(D(w))$, called the \textit{web variety}. One interesting question 
is the following (asked to us by Kamnitzer).
\begin{question}
For any two basis webs $u,v\in B^S$, does there exist 
a degree preserving algebra isomorphism  
\[
\bigoplus_{u,v\in B^S} H^*(Q(D(u)))\otimes_{F(S)}H^*(Q(D(v)))\cong 
\bigoplus_{u,v\in B^S} {}_uK_v?
\]  
Here 
\[
K^S=\bigoplus_{u,v\in B^S}{}_uK_v
\]
is the decomposition of $K^S$ in Section~\ref{sec:webalgebra}. The product 
on 
\[
\bigoplus_{u,v\in B^S} H^*(Q(D(u)))\otimes_{F(S)}H^*(Q(D(v)))
\] 
is given by convolution.   
\end{question}

If the answer to this question is affirmative, then that would be 
the $\mathfrak{sl}_3$ analogue of the aforementioned result for 
$\mathfrak{sl}_2$ due to Stroppel and Webster~\cite{sw}. 
Our Theorem~\ref{thm:center} could be a first step towards answering 
Kamnitzer's question. 
\vskip0.5cm
This paper is organized as follows.
\begin{enumerate}
\item In Section~\ref{sec:basic}, we recall the definitions 
and some fundamental properties of webs, foams and categorified quantum 
algebras and their categorical representations. The reader who already 
knows all this material well enough can just leaf through it, 
in order to understand our notations and conventions. Other readers 
might perhaps find it helpful as a brief introduction to the rapidly growing 
literature on categorification, although it is far from self-contained.  
\item In Section~\ref{sec:webalgebra}, we define $K^S$ and prove the 
first of our aforementioned main results.
\item In Section~\ref{sec:center}, we first study the relation between 
column strict tableaux and webs with 
flows. Using this relation, we prove our second main result. 
\item In Section~\ref{sec:grothendieck}, we explain skew Howe duality in our 
context and categorify the case relevant to this paper. This leads to the 
other main results. 
\item Sections~\ref{sec:center} and~\ref{sec:grothendieck} 
are largely independent of each other. However, the proof of 
Theorem~\ref{thm:center} requires Proposition~\ref{prop:morita}. The proof of 
Proposition~\ref{prop:unitriang}, which is a key ingredient 
for the proof of Theorem~\ref{thm:dualcan}, requires Lemma~\ref{lem:dimZG}. 
\item In ``Appendix 1'', we collect some technical facts 
from the literature on filtered algebras, filtered modules and their 
associated graded counterparts. These are needed at various places in the 
paper. 
\end{enumerate}

\paragraph*{Acknowledgements}
We thank Jonathan Brundan, Joel Kamnitzer, Mikhail Khovanov and 
Ben Webster for helpful exchanges of emails, some of which will 
hopefully bear fruit in future publications on this topic. In particular, we thank Mikhail Khovanov for spotting a crucial mistake in 
a previous version of this paper and Ben Webster for 
suggesting to us to use $q$-skew Howe duality in order to relate 
$K^S$ to a cyclotomic KLR-algebra.   

M.M. thanks the Courant Research Center ``Higher Order Structures'' and 
the Graduiertenkolleg 1493 in G\"{o}ttingen for sponsoring 
two research visits during this project. 

W.P. and D.T. thank the University of the Algarve and the Instituto Superior 
T\'{e}cnico for sponsoring three research visits during this project. 


\section{Basic definitions and background}
\label{sec:basic}

\subsection{Webs}
\label{sec:webs}
In~\cite{ku}, Kuperberg describes the representation theory of 
$U_{q}(\mathfrak{sl}_3)$ using oriented trivalent graphs, possibly with 
boundary, called \textit{webs}. Boundaries of webs
consist of univalent vertices (the ends of oriented edges), which we will 
usually put on a horizontal line (or various horizontal lines), e.g. such a web is shown below.
\begin{align}
\xy
   (0,0)*{\includegraphics[width=140px]{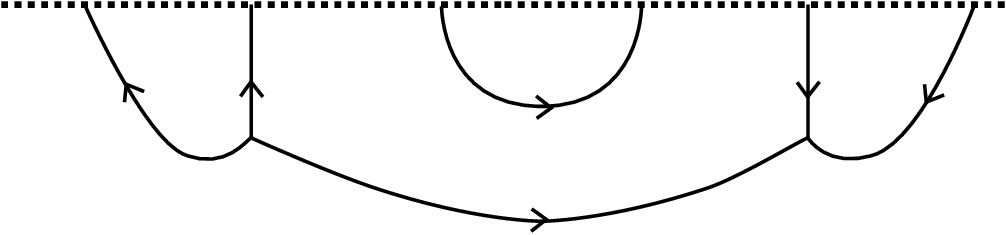}};
\endxy
\end{align}
We say that a web has $n$ free strands if the number of non-trivalent vertices 
is exactly $n$. In this way, the boundary of a web can be 
identified with a \textit{sign string} $S=(s_1,\ldots,s_n)$, with $s_i=\pm$, 
such that upward oriented boundary edges get a ``$+$'' and downward oriented 
boundary edges a ``$-$'' sign. Webs without boundary are called 
\textit{closed} webs. 

Any web can be obtained from the following 
elementary webs by glueing and disjoint union.
\begin{align}
\xy
   (0,0)*{\includegraphics[width=280px]{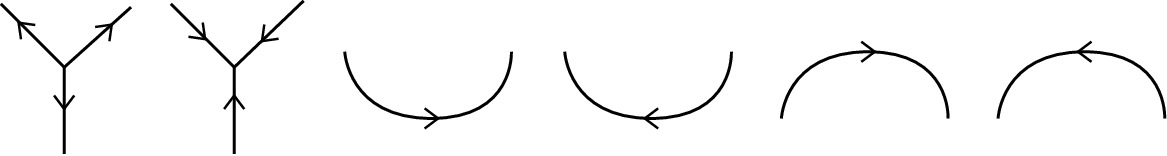}};
\endxy
\end{align}
Fixing a boundary $S$, we can form the 
$\mathbb{Q}(q)$-vector space $W^S$, spanned by all webs with boundary $S$, 
modulo the following set of local relations (due to Kuperberg~\cite{ku}). 
\begin{align}
\label{eq:circle}
\xy(0,0)*{\includegraphics[width=20px]{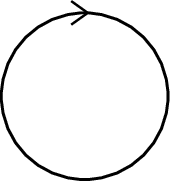}}\endxy\;\; &=\;\; [3]\\
\label{eq:digon}
\xy(0,0)*{\includegraphics[width=70px]{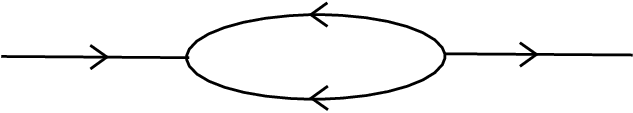}}\endxy\;\; &=\;\; [2]\;\; \xy(0,0)*{\includegraphics[width=50px]{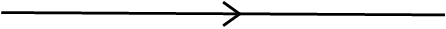}}\endxy\\
\label{eq:square}
\xy(0,0)*{\includegraphics[width=50px]{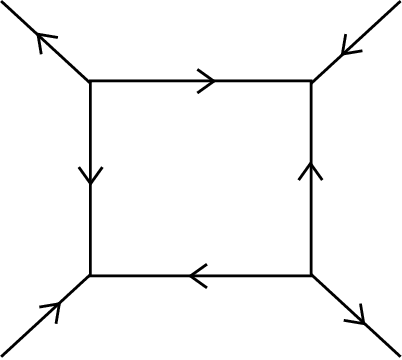}}\endxy\;\; &=\;\;\xy(0,0)*{\includegraphics[height=45px]{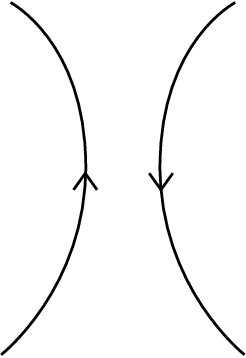}}\endxy + \;\;\xy(0,0)*{\includegraphics[width=45px]{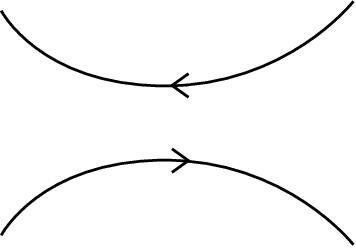}}\endxy
\end{align}
Recall that 
\[
[a]=\frac{q^a-q^{-a}}{q-q^{-1}}=q^{a-1}+q^{a-3}+\cdots+q^{-(a-1)}\in\mathbb{N}[q,q^{-1}]
\] 
denotes the \textit{quantum integer}.
\begin{rem}
Throughout this paper we will use $q$ as the quantum parameter, which we take to be equal to the 
parameter $v$ in~\cite{kk}. Note that in that paper Khovanov and Kuperberg use the 
different convention $v=-q^{-1}$. 
\end{rem}

By abuse of notation, we will call all elements of $W^S$ webs. 
From relations~\eqref{eq:circle},~\eqref{eq:digon} and~\eqref{eq:square} it 
follows that any element in $W^S$ is a linear 
combination of webs with the same boundary and without circles, digons or 
squares. These are called \textit{non-elliptic webs}. 
As a matter of fact, the non-elliptic webs form a basis of $W^S$, which 
we call $B^S$. 
Therefore, we will simply call them \textit{basis webs}.
\vskip0.5cm
Following Brundan and Stroppel's~\cite{bs} notation for arc diagrams, 
we will write $w^*$ to denote the web obtained by reflecting a given 
web $w$ horizontally and reversing all 
orientations.
\begin{align}
\xy
 (0,0)*{\includegraphics[width=140px]{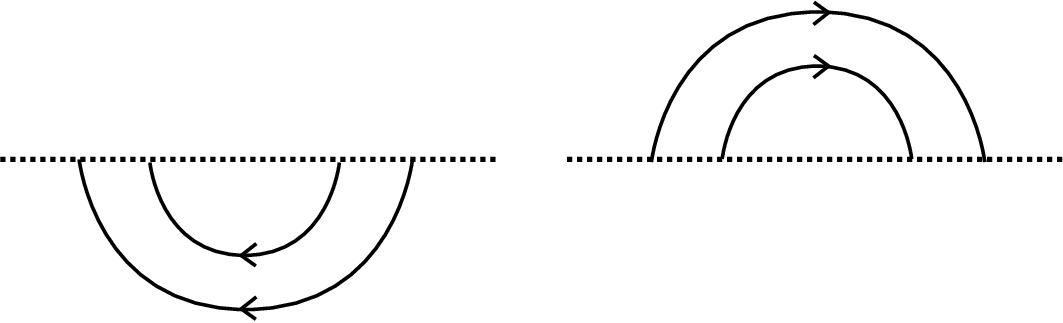}};
 (-13.5,2)*{w};
 (14,-2)*{w^*};
\endxy
\end{align}
By $uv^*$, we mean the planar diagram containing the disjoint union of $u$ and 
$v^*$, where $u$ lies vertically above $v^*$.
\begin{align}
\xy
 (0,0)*{\includegraphics[width=70px]{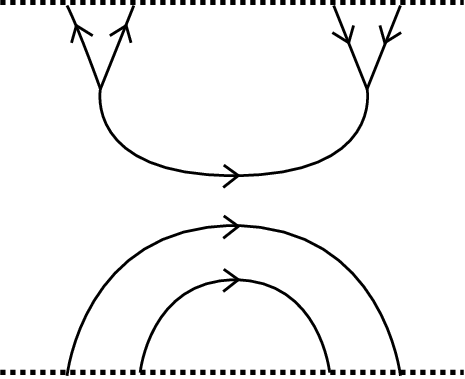}};
 (12,5)*{u};
 (12.5,-5)*{v^*};
\endxy
\end{align}
By $v^*u$, we shall mean the closed web obtained by glueing 
$v^*$ on top of $u$, when such a construction is possible (i.e. the number of free strands and orientations on the strands match).
\begin{align}\label{closed}
\xy
 (0,0)*{\includegraphics[width=70px]{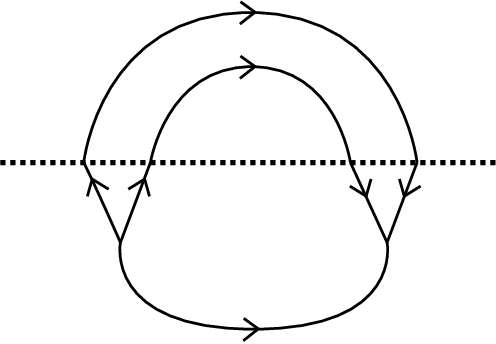}};
 (12,-4)*{u};
 (12.5,5)*{v^*};
\endxy
\end{align}
In the same vein, by $v_1^*u_1v_2^*u_2$ we denote the following web.
\begin{align}
\xy
 (0,0)*{\includegraphics[width=70px]{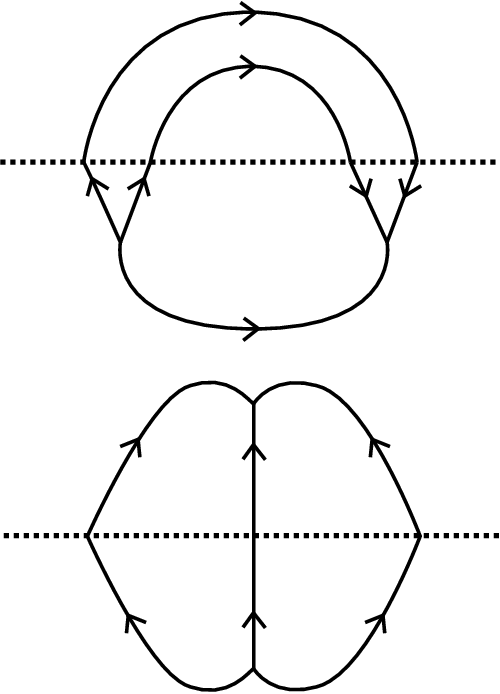}};
 (12.2,5)*{u_1};
 (12.4,-5.5)*{v_2^*};
 (12.2,-13)*{u_2};
 (12.4,13)*{v_1^*};
\endxy
\end{align}
\vskip0.5cm
To make the connection with the representation theory of 
$U_{q}(\mathfrak{sl}_3)$, we recall that 
a sign string $S=(s_1,\ldots,s_n)$ corresponds to 
\[
V^S=V^{s_1}\otimes \cdots\otimes V^{s_n},
\]
where $V^{+}$ is the fundamental representation and $V^{-}$ its dual. The 
latter is also isomorphic to $V^+\wedge V^+$, a fact which we will need 
later on.

Both $V^+$ and $V^-$ have dimension three. In this interpretation, 
webs correspond to intertwiners and we have
\[
W^S\cong \mathrm{Inv}_{U_q(\mathfrak{sl}_3)}(V^S)\cong\mathrm{Hom}_{U_q(\mathfrak{sl}_3)}(\bC_q,V^S).
\]
Here $\bC_q$ is the trivial representation.

Thus, the elements of 
$B^S$ give a basis of $\mathrm{Inv}_{U_q(\mathfrak{sl}_3)}(V^S)$. 
Using the embedding $\mathrm{Inv}_{U_q(\mathfrak{sl}_3)}(V^S)\subset V^S$, one can expand a non-elliptic 
web in terms of elementary tensors. In Theorem 2 of~\cite{kk}, Kuperberg and Khovanov prove an 
important result about this expansion, which we will reproduce in Theorem~\ref{thm:upptriang}.
\vskip0.5cm
Kuperberg showed in~\cite{ku} (see also~\cite{kk}) that basis webs are indexed 
by closed weight lattice paths in the dominant Weyl chamber of 
$\mathfrak{sl}_3$. It is well-known that any path in the 
$\mathfrak{sl}_3$-weight lattice can be presented by a pair consisting of a 
sign string $S=(s_1,\ldots,s_n)$ and a 
\textit{state string} $J=(j_1,\ldots,j_n)$, with $j_i\in \{-1,0,1\}$ for all 
$1\leq i\leq n$. Given a pair $(S,J)$ representing a closed dominant path, 
a unique basis web (up to isotopy) is determined by a set of 
inductive rules called 
the \textit{growth algorithm}. We briefly recall the algorithm as 
described in~\cite{kk}. In fact, the algorithm can be applied to any path, but 
we will only use it for closed dominant paths. 

\begin{defn} \label{growth}
\textbf{(The growth algorithm)} 
Given $(S,J)$, a web $w^S_J$ is recursively generated by the following rules.
\begin{enumerate}
\item Initially, the web consists of $n$ parallel strands whose orientations are given by the sign string. If $s_{i} = +$, then the $i$-th strand is oriented 
upwards; if $s_{i} = -$, it is oriented downwards.
\item The algorithm builds the web downwards. Suppose we have already applied 
the algorithm $k-1$ times. For the $k$-th step, do the following. 
If the bottom boundary string contains a neighboring pair of edges 
matching the top of one of 
the following webs (called H, arc and Y respectively), then glue 
the corresponding H, arc or Y to the relevant bottom boundary edges.
\begin{figure}[H]
  \centering
    \includegraphics[width=260px]{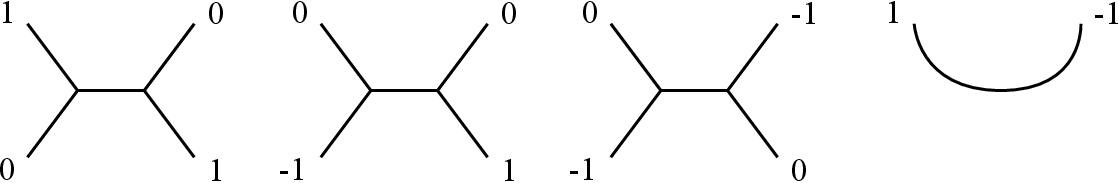}
    \caption{Top strands have different signs.}
\end{figure}
\begin{figure}[H]
  \centering
    \includegraphics[width=180px]{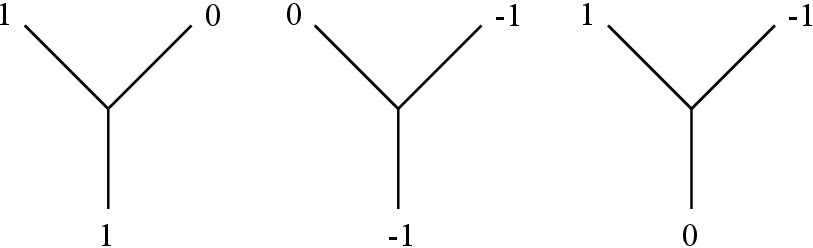}
    \caption{Top strands have same sign.}
\end{figure}
\end{enumerate}
These rules apply for any compatible orientation of the edges in the webs. 
Therefore, we have drawn them without any specific orientations. Below, 
whenever we write down an equation involving webs without orientations, 
we mean that the equation holds for all possible orientations. 
For future use, we will call the rules above the \textit{H, arc and Y-rule}. 
The growth algorithm stops if no further rules can be applied. 
\end{defn}
If $(S,J)$ represents a closed dominant path, then the growth algorithm 
produces a basis web. 

For example, the growth algorithm converts $S=(+-+-+++)$ and 
$J=(1,1,0,0,-1,0,-1)$ into the following basis web.
\begin{align}
\xy
 (0,0)*{\includegraphics[width=140px]{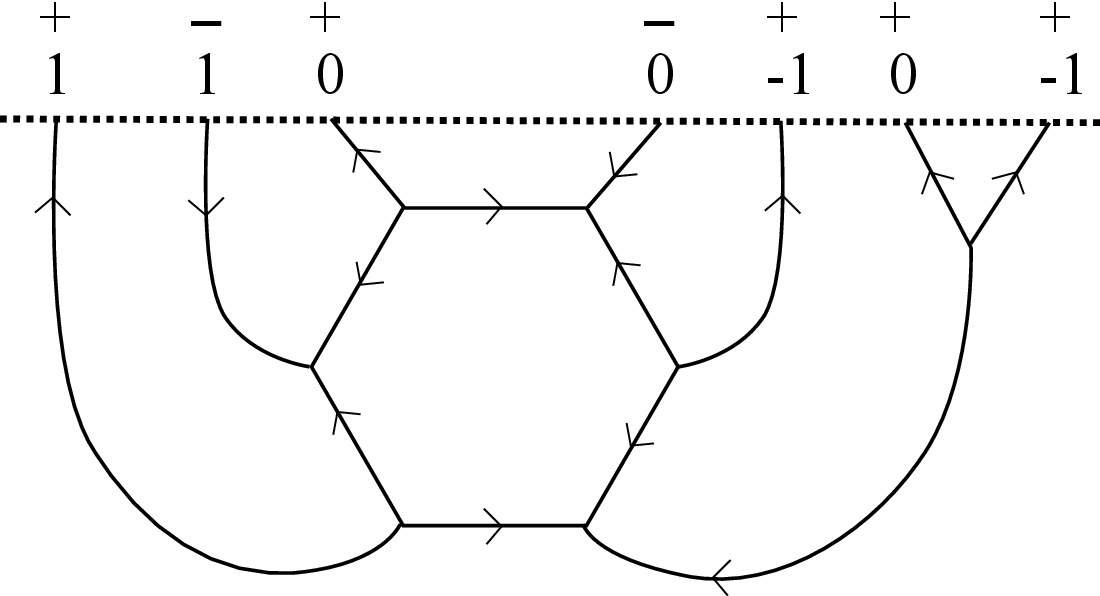}};
\endxy
\end{align}
In addition, the growth algorithm has an inverse, 
called the \textit{minimal cut path algorithm}~\cite{kk}, which we will 
not use in this paper. 
\vskip0.5cm
Following Khovanov and Kuperberg in~\cite{kk}, we define a \textit{flow} $f$ on 
a web $w$ to be an oriented subgraph that contains exactly two of the three edges 
incident to each trivalent vertex. The connected components of the 
flow are called the \textit{flow lines}. The orientation of the flow lines 
need not agree with the orientation of $w$. Note that if $w$ is closed, then 
each flow line is a closed cycle. At the boundary, the flow lines 
can be represented by a state string $J$. By convention, at the $i$-th 
boundary edge, we set $j_i=+1$ if the flow line is oriented upward, $j_i=-1$ if 
the flow line is oriented downward and $j_i=0$ there is no flow line. The same 
convention determines a state for each edge of $w$. 
\begin{rem}
\label{rem:3color}
Every flow determines a unique 3-coloring of $w$, with colours $-1,0,1$, 
satisfying the property that, for any trivalent vertex of $w$, 
the colors of the three incident edges are all distinct. 
These colorings are called \textit{admissible} in~\cite{g}. 

Conversely, any such 3-coloring determines a unique flow on $w$. This 
correspondence determines a bijection between flows and admissible 
3-colorings on $w$.   

This remark will be important in Section~\ref{sec-webhowea} and in Section~\ref{sec-webhoweb}.
\end{rem}
We will also say that any flow $f$ that is compatible with a given state string $J$ 
on the boundary of $w$ \textit{extends} $J$. 

Given a web with a flow, denoted $w_f$, Khovanov and Kuperberg~\cite{kk} 
attribute a \textit{weight} to each trivalent vertex and each arc in $w_f$, 
as in Figures~\ref{weights} and~\ref{weights2}. 
The total weight of $w_f$ is by definition the sum of the 
weights at all trivalent vertices and arcs.
\begin{align}\label{weights}
  & \xy
 (0,0)*{\includegraphics[width=310px]{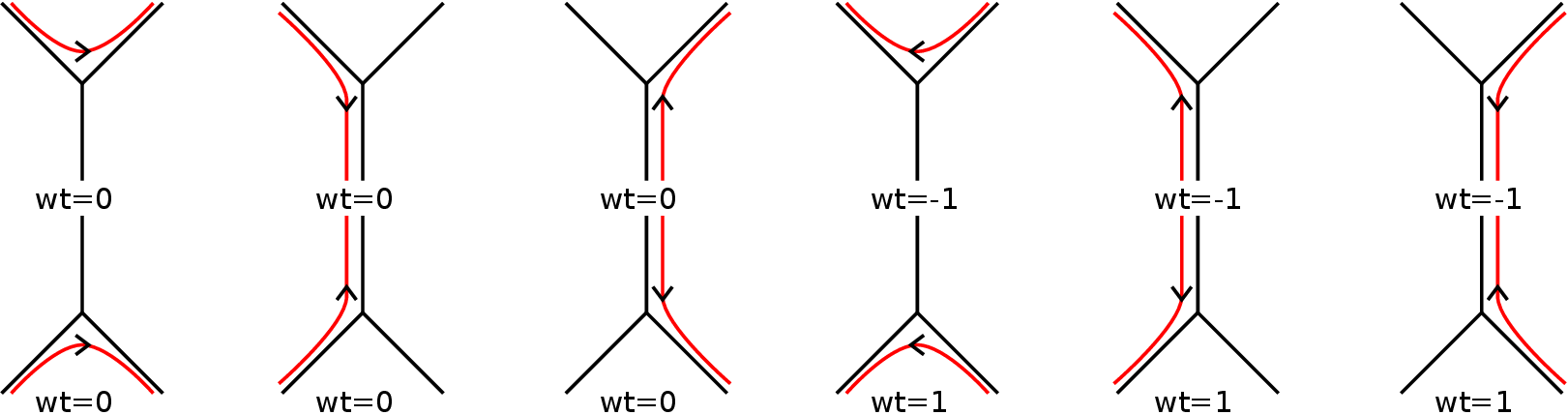}};
\endxy\\
   \label{weights2}
  & \xy
 (0,0)*{\includegraphics[width=310px]{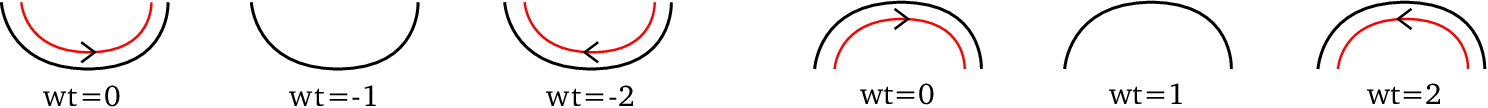}};
\endxy
\end{align}
For example, the following web has weight $-4$.
\begin{align}
\xy
 (0,0)*{\includegraphics[width=150px]{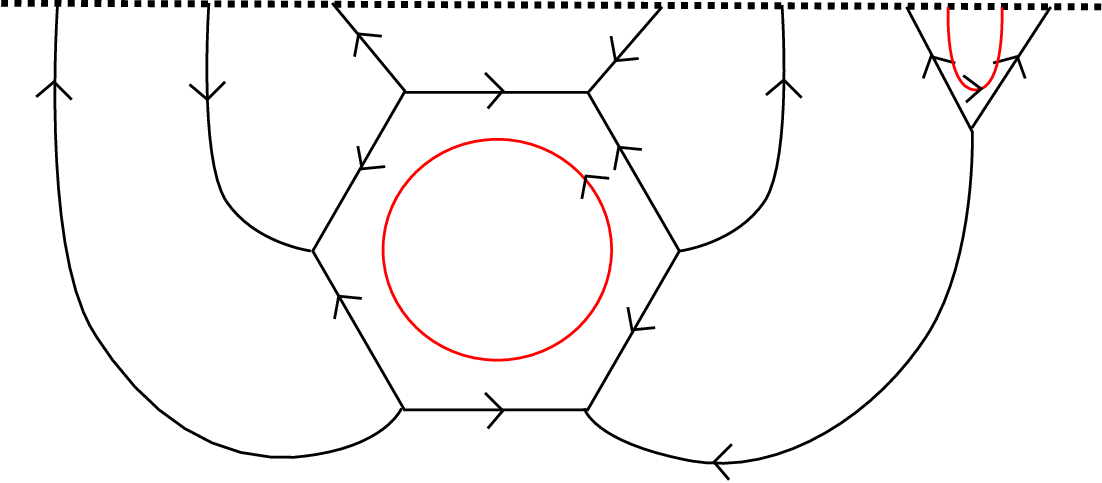}};
\endxy
\end{align}
We can extend the table in~\eqref{weights} and~\eqref{weights2} to calculate 
weights determined by flows on $H$'s, so that it becomes easier to 
compute the weight of $w_f$ when $w$ is expressed using the growth algorithm 
(Definition~\ref{growth}).

\begin{defn}\cite{kk}\label{defn-cano}
\textbf{(Canonical flows on basis webs)} Given a basis web $w$ expressed 
using the growth algorithm. 
We define the \textit{canonical flow} on $w$ by the following rules.
\begin{align} \label{canrule}
\xy
 (0,0)*{\includegraphics[width=250px]{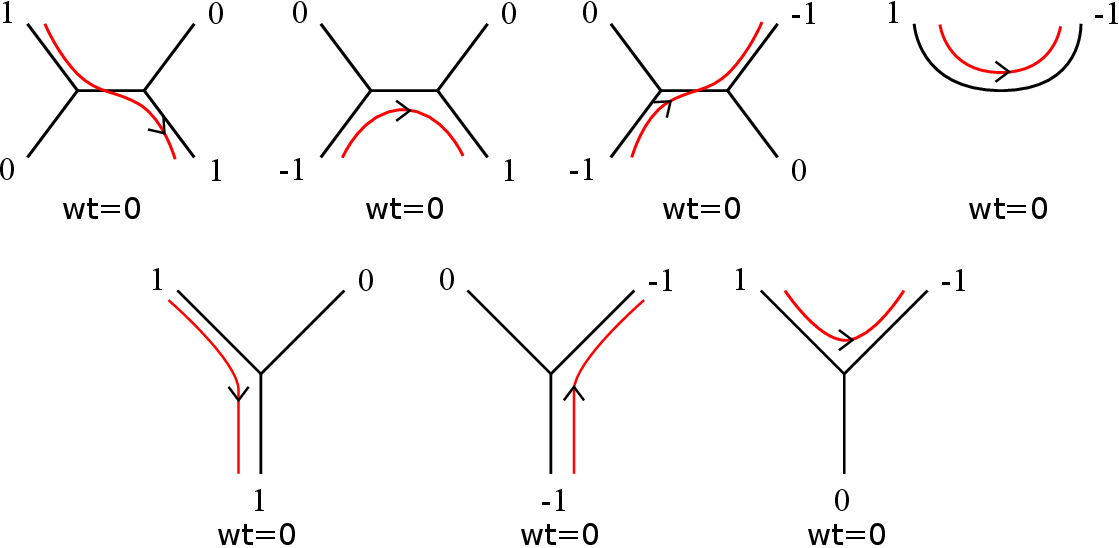}};
\endxy
\end{align}
The canonical flow does not depend on the 
particular instance of the growth algorithm that we have chosen to obtain 
$w$. 
\end{defn}
Observe that the definition of the canonical flows implies the following lemma.  

\begin{lem}
\label{lem:canflowzero}
A basis web with the canonical flow has weight zero.
\end{lem}

Khovanov and Kuperberg~\cite{kk} use a particular basis for $V^+$, 
denoted $\{e^+_1,e^+_2,e^+_3\}$, and also one for $V^-$, 
denoted $\{e^-_1,e^-_2,e^-_3\}$. Given $(S,J)$, let 
\[
e^S_J=e^{s_1}_{j_1}\otimes\cdots\otimes e^{s_n}_{j_n}
\] 
be the {\em elementary tensor} corresponding to $(S,J)$. 
Interpreting the webs in $W^S$ as invariant tensors in $V^S$, we can write 
any web as a linear combination of elementary tensors. 
For each state string $J^{\prime}$, 
one can consider all flows on a given basis web $w=w^S_J$ which are compatible 
with $J^{\prime}$ on the boundary. The coefficient $c(S,J,J^{\prime})$ of 
$e^S_{J^{\prime}}$ in the linear combination corresponding to $w$ satisfies  
\[
c(S,J,J^{\prime})=\sum_{f}q^{\mathrm{wt}(w_f)},
\]
where the sum is taken over all flows on $w$ which are compatible with 
$J^{\prime}$ on the boundary. 

Khovanov and Kuperberg prove the following  
theorem (Theorem 2 in~\cite{kk}), which will be important for us 
in Section~\ref{sec:grothendieck}.
\begin{thm}\label{thm:upptriang}(\textbf{Khovanov-Kuperberg})
Given $(S,J)$, we have 
\[ 
w^S_J=e^S_J+\sum_{J^{\prime}<J}c(S,J,J^{\prime})e^S_{J^{\prime}}
\]
for some coefficients $c(S,J,J^{\prime})\in\mathbb{N}[q,q^{-1}]$, where the 
state strings $J$ and $J^{\prime}$ are ordered lexicographically.
\end{thm}

\begin{rem} 
\label{rem:counter}
Khovanov and Kuperberg~\cite{kk} show that $B^S$ is not 
equal to the dual canonical basis of $W^S$. This follows from the 
fact that $c(S,J,J^{\prime})\not\in q^{-1}\mathbb{N}[q^{-1}]$, for general $J^{\prime}< J$. 
In their Section 8, they give explicit counter-examples of elements $w\in B^S$ which 
admit non-canonical weight zero flows. 
\end{rem}
 

\subsection{Foams}
\label{subsec:foams}
In this subsection we review the category $\foamt$ of $\mathfrak{sl}_3$-foams 
introduced by Khovanov in~\cite{kv}. As a matter of fact, we will also need 
a deformation of Khovanov's original category, due to Gornik~\cite{g} 
in the context of matrix factorizations, and studied in~\cite{mv} in the 
context of foams. Therefore, we introduce a parameter $c\in \mathbb{C}$ in 
$\foamt^c$, just 
as in~\cite{mv}, such that we get Khovanov's original category for $c=0$ and, 
for any $c\ne 0$, the category $\foamt^c$ is isomorphic to 
Gornik's deformation (his original deformation was for $c=1$). 
A big difference between these two 
specializations is that $\foamt^c$ is graded for $c=0$ and filtered for 
any $c\ne 0$. In fact, for any $c\ne 0$, the associated graded category of 
$\foamt^c$ is isomorphic to $\foamt^0$.  

We recall the following definitions as they appear in~\cite{mv}. We note that the diagrams accompanying these definitions are taken, also, from~\cite{mv}.

A \textit{pre-foam} is a cobordism with singular arcs between two webs. 
A singular arc in a pre-foam $U$ is the set of points of $U$ which have a neighborhood 
homeomorphic to the letter Y 
times an interval. Note that singular arcs are disjoint. 
Interpreted as morphisms, we read pre-foams from bottom to top by convention. Thus, pre-foam composition consists of placing one 
pre-foam on top of the other. The orientation of the singular arcs is, by convention, as in 
the diagrams below (called the \textit{zip} and the 
\textit{unzip} respectively).
\begin{equation*}
\figins{-20}{0.5}{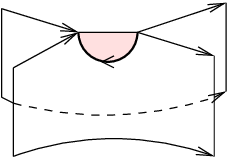}
\mspace{35mu}\mspace{35mu}
\figins{-20}{0.5}{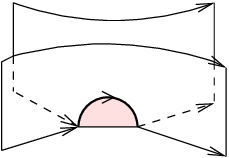}\ 
\end{equation*}
We allow pre-foams to have dots that can move freely about the facet on which they 
belong, but we do not allow a dot to cross singular arcs. 

By a \textit{foam}, we mean a formal $\mathbb{C}$-linear combination of 
isotopy classes of pre-foams modulo 
the ideal generated by the set of relations $\ell=(3D,NC,S,\Theta)$ and 
the \textit{closure relation}, 
as described below. 
\begin{gather*}
\figins{-7}{0.25}{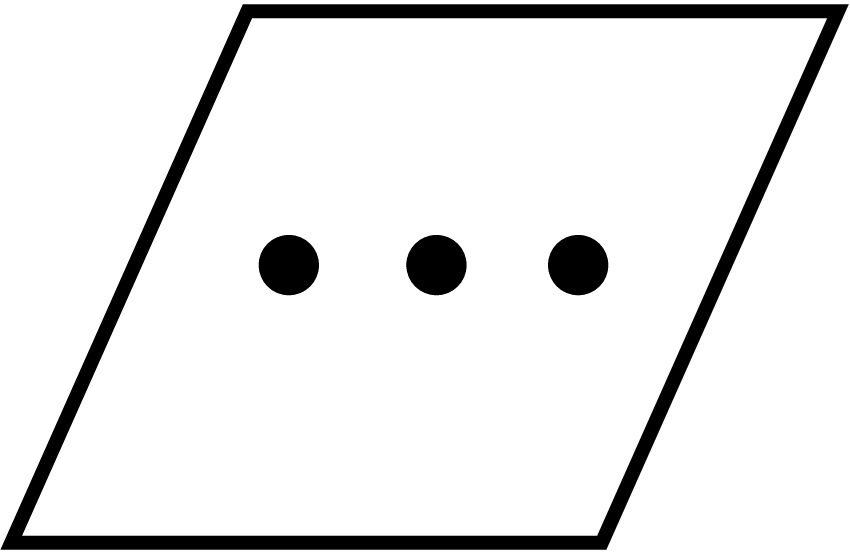}
= c\figins{-7}{0.25}{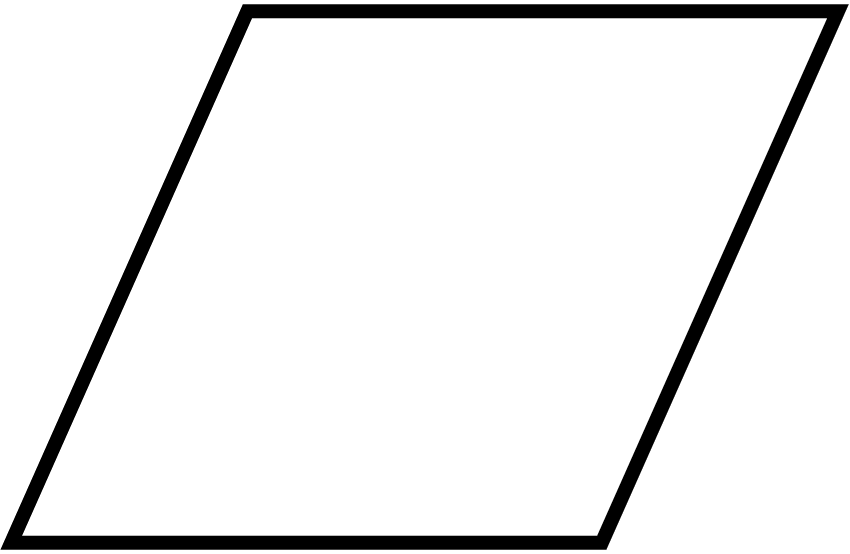}
\tag{3D}
\\[1.5ex]\displaybreak[0]
 \figwhins{-22}{0.65}{0.26}{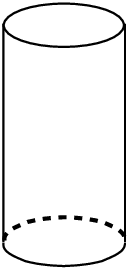}=
-\figwhins{-22}{0.65}{0.26}{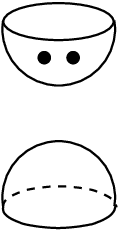}
-\figwhins{-22}{0.65}{0.26}{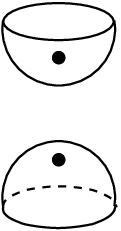}
-\figwhins{-22}{0.65}{0.26}{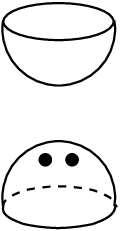}
\tag{NC}\label{eq:cn}
\\[1.5ex]\displaybreak[0]
\figins{-8}{0.3}{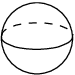}=
\figins{-8}{0.3}{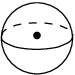}=0,\quad
\figins{-8}{0.3}{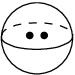}=-1
\tag{S}
\\[1.5ex]\displaybreak[0]
\labellist
\small\hair 2pt
\pinlabel $\alpha$ at 3 33
\pinlabel $\beta$ at -3 17
\pinlabel $\delta$ at 3 5
\endlabellist
\centering
\figins{-10}{0.4}{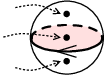} 
=\begin{cases}
\ \ 1, & (\alpha,\beta,\delta)=(1,2,0)\text{ or a cyclic permutation}, \\ 
-1, & (\alpha,\beta,\delta)=(2,1,0)\text{ or a cyclic permutation}, \\ 
\ \ 0, & \text{else}.
\end{cases}
\tag{$\Theta$}\label{eq:theta}
\end{gather*}
\begin{quote}
The \textit{closure relation}, i.e. any $\mathbb{C}$-linear combination of foams with the same boundary,
is equal to zero if and only if any way of capping off these foams with a 
common foam yields a $\mathbb{C}$-linear combination of closed foams 
whose evaluation is zero. 
\end{quote}

The relations in $\ell$ imply the following identities (for detailed proofs see~\cite{kv}).
\begin{align}
\figwhins{-22}{0.65}{0.30}{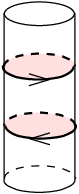}
& = -\ 
\figwhins{-22}{0.65}{0.30}{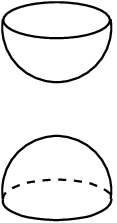}
\tag{Bamboo}\label{eq:bamboo}
\\[1.2ex]\displaybreak[0]
\figwhins{-22}{0.65}{0.30}{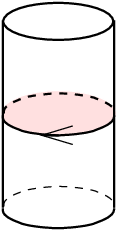}
& =\
\figwhins{-22}{0.65}{0.30}{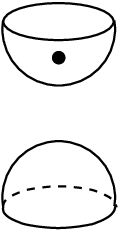}-
\figwhins{-22}{0.65}{0.30}{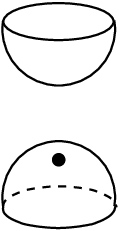}
 \tag{RD}\label{eq:rd}
\\[1.2ex]\displaybreak[0]
\figins{-12}{0.4}{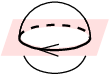} 
& =\ 0
\tag{Bubble}\label{eq:bubble}
\\[1.2ex]\displaybreak[0]
\figins{-20}{0.6}{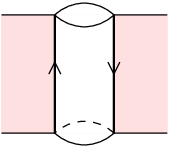}
& = 
\figins{-26}{0.75}{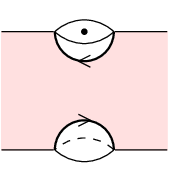}-
\figins{-26}{0.75}{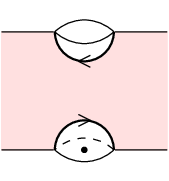}
\tag{DR}\label{eq:dr}
\\[1.2ex]\displaybreak[0]
\figins{-28}{0.8}{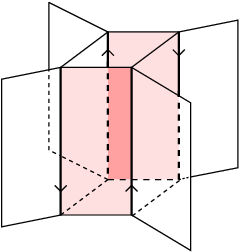}
&=
-\ \figins{-28}{0.8}{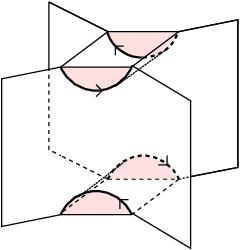}
-\figins{-28}{0.8}{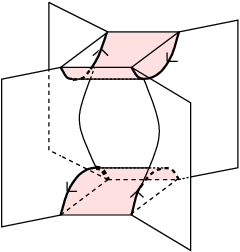}
\tag{SqR}\label{eq:sqr}
\end{align}

\begin{equation}\tag{Dot Migration}\label{eq:dotm}
\begin{split}
\figins{-22}{0.6}{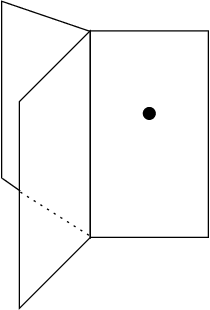}
\,+\,
\figins{-22}{0.6}{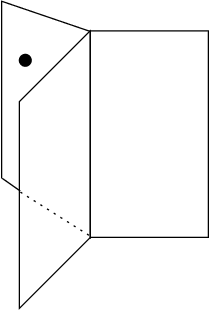}
\,+\,
\figins{-22}{0.6}{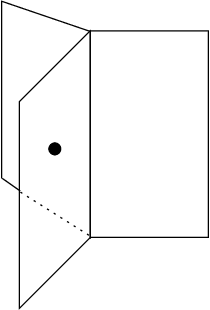}
\, &= 0
\\[1.2ex]
\figins{-22}{0.6}{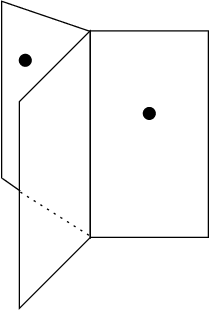}
\,+\,
\figins{-22}{0.6}{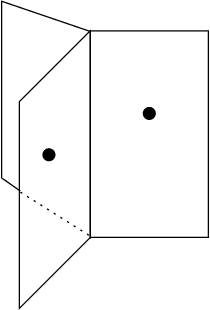}
\,+\,
\figins{-22}{0.6}{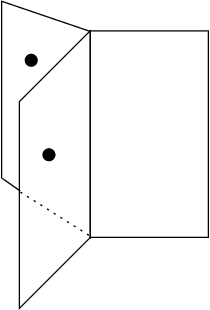}
\, &=\ 0
\\[1.2ex]
\figins{-22}{0.6}{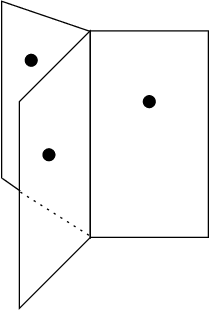}
\, &=c\;\figins{-22}{0.6}{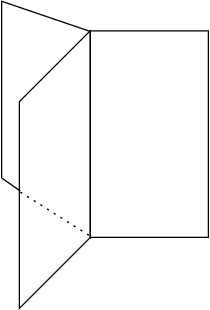}
\end{split}
\end{equation}
\begin{defn}
For any $c\in \mathbb{C}$, let $\foamt^c$ be the category whose 
objects are webs $\Gamma$ lying inside a horizontal strip in 
$\mathbb{R}^2$, which is bounded by the lines $y=0,1$ containing the boundary points of $\Gamma$. 
The morphisms of $\foamt^c$ are $\mathbb{C}$-linear combinations of foams lying inside the horizontal strip bounded by $y=0,1$ times the unit interval. We require that the vertical boundary of each foam is a set (possibly empty) of vertical lines.
\end{defn}

The \textit{$q$-grading} of a foam $U$ is defined as 
\[
q(U) = \chi(\partial U) - 2\chi(U) + 2d+b,
\] 
where $\chi$ denotes the Euler characteristic, $d$ is 
the number of dots on $U$ and $b$ is the number of vertical boundary 
components. This makes $\foamt^0$ into a graded category. For any 
$c\ne 0$, this makes $\foamt^c$ into a filtered category, whose 
associated graded category is isomorphic to $\foamt^0$. 

\begin{defn}\cite{kv} \label{foamhom} \textbf{(Foam Homology)} Given a web 
$w$ the \textit{foam homology} of $w$ is the complex vector space, 
$\F^c(w)$, spanned by all foams  
\[
U\colon\emptyset \to w
\]
in $\foamt^c$.
\end{defn}

The complex vector space $\F^c(w)$ is filtered/graded by 
the $q$-grading on foams and has rank 
$\langle w\rangle_{\mathrm{Kup}}$, where $\langle w\rangle_{\mathrm{Kup}}$ is the \textit{Kuperberg bracket} 
computed recursively by the rules below.
\begin{enumerate}
\item $\left\langle w \amalg \raisebox{-2mm}
{\includegraphics[width=15px]{section21/cirrem}}\right\rangle_{\mathrm{Kup}}= [3]\langle w\rangle_{\mathrm{Kup}}$.
\item $\langle  \raisebox{-0.5mm}{\includegraphics[width=45px]{section21/dgrem}}\rangle_{\mathrm{Kup}} 
= [2]\langle  \raisebox{0.5mm}{\includegraphics[width=30px]{section21/line}}\rangle_{\mathrm{Kup}}$.
\item $\left\langle \raisebox{-3mm}{\includegraphics[width=25px]{section21/sqrem}}
\right\rangle_{\mathrm{Kup}} = 
\left\langle  \raisebox{-3mm}{\includegraphics[width=15px]{section21/vert}}\right
\rangle_{\mathrm{Kup}} + \left\langle  \raisebox{-1.9mm}{\includegraphics[width=23px]{section21/horiz}}
\right\rangle_{\mathrm{Kup}}$.
\end{enumerate}
The relations above correspond to the decomposition of $\F^c(w)$ into direct 
summands. The idempotents corresponding to these direct summands are the 
terms on the r.h.s. of the relations (NC), (DR) and (SqR), respectively.
For any $c\ne 0$, the complex vector space $\F^c(w)$ is filtered and 
its associated graded vector space is $\F^0(w)$. See~\cite{kh} and~\cite{mv} for details. 

\begin{rem}
\label{rem:flowsfoams}
Given $u,v\in B^S$, the observations above 
and Theorem~\ref{thm:upptriang} 
show that there exists a homogeneous basis of 
$\F^0(u^*v)$ parametrized by the flows on 
$u^*v$. We have, in fact, constructed such a basis, but it is not unique. 
There is 
also no ``preferred choice'', unless one requires the basis to have other 
nice properties, e.g. in the $\mathfrak{sl}_2$ case, 
Brundan and Stroppel prove that there is a cellular basis of $H^n$. The 
construction of a ``good'' basis of the $\mathfrak{sl}_3$-web algebra 
$K^S$ (and similarly for Gornik's deformation $G^S$) 
is still work in progress and will, 
hopefully, be the contents of a subsequent paper. Although we do not need 
such a basis in this paper, it is important that the reader keep this remark 
in mind while reading Section~\ref{sec:grothendieck}. 
\end{rem}


\subsection{Quantum 2-algebras}
\subsubsection{The quantum general and special linear algebras}
First we recall the quantum general and special linear algebras. Most parts in this section are copied from section two and three in~\cite{msv2}. 

The $\mathfrak{gl}_n$-weight lattice is isomorphic to $\bZ^n$. Let 
$\epsilon_i=(0,\ldots,1,\ldots,0)\in \bZ^n$, with $1$ being on the $i$-th 
coordinate, and $\alpha_i=\epsilon_i-\epsilon_{i+1}
=(0,\ldots,1,-1,\ldots,0)\in\bZ^{n}$, for 
$i=1,\ldots,n-1$. Recall that the Euclidean inner product on $\bZ^n$ is defined by  
$(\epsilon_i,\epsilon_j)=\delta_{i,j}$. 
   
\begin{defn} For $n\in\bN_{>1}$ the \textit{quantum general linear algebra} 
${\mathbf U}_q(\mathfrak{gl}_n)$ is 
the associative unital $\bQ(q)$-algebra generated by $K_i$ and $K_i^{-1}$, for $1,\ldots, n$, 
and $E_{\pm i}$ (beware that some authors use $F_i$ instead of $E_{-i}$), for $i=1,\ldots, n-1$, subject to the relations
\begin{gather*}
K_iK_j=K_jK_i,\quad K_iK_i^{-1}=K_i^{-1}K_i=1,
\\
E_iE_{-j} - E_{-j}E_i = \delta_{i,j}\dfrac{K_iK_{i+1}^{-1}-K_i^{-1}K_{i+1}}{q-q^{-1}},
\\
K_iE_{\pm j}=q^{\pm (\epsilon_i,\alpha_j)}E_{\pm j}K_i,
\\
E_{\pm i}^2E_{\pm j}-(q+q^{-1})E_{\pm i}E_{\pm j}E_{\pm i}+E_{\pm j}E_{\pm i}^2=0,
\qquad\text{if}\quad |i-j|=1,
\\
E_{\pm i}E_{\pm j}-E_{\pm j}E_{\pm i}=0,\qquad\text{else}.
\end{gather*} 
\end{defn}

\begin{defn} 
\label{defn:qsln}
For $n\in\bN_{>1}$ the \textit{quantum special linear algebra} 
${\mathbf U}_q(\mathfrak{sl}_n)\subseteq {\mathbf U}_q(\mathfrak{gl}_n)$ is 
the unital $\bQ(q)$-subalgebra generated by $K_iK^{-1}_{i+1}$ and 
$E_{\pm i}$, for $i=1,\ldots, n-1$.
\end{defn}

Recall that the ${\mathbf U}_q(\mathfrak{sl}_n)$-weight lattice is 
isomorphic to $\bZ^{n-1}$. Suppose that $V$ is a 
${\mathbf U}_q(\mathfrak{gl}_n)$-weight representation with 
weights $\lambda=(\lambda_1,\ldots,\lambda_n)\in\bZ^n$, i.e. 
\[
V\cong \bigoplus_{\lambda}V_{\lambda},
\] 
and $K_i$ acts as multiplication by 
$q^{\lambda_i}$ on $V_{\lambda}$. Then $V$ is also a 
${\mathbf U}_q(\mathfrak{sl}_n)$-weight representation with weights 
$\overline{\lambda}=(\overline{\lambda}_1,\ldots,\overline{\lambda}_{n-1})\in
\bZ^{n-1}$ such that 
$\overline{\lambda}_j=\lambda_j-\lambda_{j+1}$ for $j=1,\ldots,n-1$.
 
Conversely, given a ${\mathbf U}_q(\mathfrak{sl}_n)$-weight 
representation with weights $\mu=(\mu_1,\ldots,\mu_{n-1})$, there is not a 
unique choice of ${\mathbf U}_q(\mathfrak{gl}_n)$-action on $V$. We can 
fix this by choosing the action of $K_1,\cdots, K_n$. In terms of weights, this 
corresponds to the observation that, for any $d\in\bZ$, the equations 
\begin{align}
\label{eq:sl-gl-wts1}
\lambda_i-\lambda_{i+1}&=\mu_i,\\
\label{eq:sl-gl-wts2}
\qquad \sum_{i=1}^{n}\lambda_i&=d,
\end{align}
determine $\lambda=(\lambda_1,\ldots,\lambda_n)$ uniquely, 
if there exists a solution to~\eqref{eq:sl-gl-wts1} and~\eqref{eq:sl-gl-wts2} 
at all. To fix notation, we 
define the map $\phi_{n,d}\colon \bZ^{n-1}\to \bZ^{n}\cup \{*\}$ by 
\[
\phi_{n,d}(\mu)=\lambda, 
\]
if~\eqref{eq:sl-gl-wts1} and \eqref{eq:sl-gl-wts2} have a solution, and we
put $\phi_{n,d}(\mu)=*$ otherwise.   

Note that ${\mathbf U}_q(\mathfrak{gl}_n)$ and 
${\mathbf U}_q(\mathfrak{sl}_n)$ are both Hopf algebras, which implies that 
the tensor product of two of their representations is a representation again. Moreover, the duals of representations are representations and there is a trivial representation. 

Both ${\mathbf U}_q(\mathfrak{gl}_n)$ and ${\mathbf U}_q(\mathfrak{sl}_n)$ 
have plenty of non-weight representations, but we won't discuss them in the paper. 
Therefore we can restrict our attention to the 
Beilinson-Lusztig-MacPherson~\cite{B-L-M} idempotent version of these 
quantum groups, denoted $\Ugl$ and $\U$ respectively. It is worth noting that such algebras can be seen as $1$-categories.

To understand their definition, recall that $K_i$ acts as $q^{\lambda_i}$ on the 
$\lambda$-weight space of any weight representation. 
For each $\lambda\in\bZ^n$ adjoin an idempotent $1_{\lambda}$ to 
${\mathbf U}_q(\mathfrak{gl}_n)$ and add 
the relations
\begin{align*}
1_{\lambda}1_{\mu} &= \delta_{\lambda,\nu}1_{\lambda},   
\\
E_{\pm i}1_{\lambda} &= 1_{\lambda\pm\alpha_i}E_{\pm i},
\\
K_i1_{\lambda} &= q^{\lambda_i}1_{\lambda}.
\end{align*}
\begin{defn} 
\label{defn:Uglndot}
The idempotented quantum general linear algebra is defined by 
\[
\Ugl=\bigoplus_{\lambda,\mu\in\bZ^n}1_{\lambda}{\mathbf U}_q(\mathfrak{gl}_n)1_{\mu}.
\]
\end{defn}
\noindent Let $I=\{1,2,\ldots,n-1\}$. In the sequel we use {\em signed sequences} 
$\ii=(\alpha_1i_1,\ldots,\alpha_mi_m)$, 
for any $m\in\bN$, $\alpha_j\in\{\pm 1\}$ and $i_j\in I$. 
The set of signed sequences 
we denote $\sseq$.

For such an $\ii=(\alpha_1 i_1,\ldots,\alpha_{n-1}i_{n-1})$ we define
\[
E_{\ii}=E_{\alpha_1 i_1}\cdots E_{\alpha_{n-1} i_{n-1}}
\] 
and we define $\ii_{\Lambda}\in\bZ^n$ to be the $n$-tuple such that 
\[
E_{\ii}1_{\mu}=1_{\mu + \ii_{\Lambda}}E_{\ii}.
\]

Similarly, for ${\mathbf U}_q(\mathfrak{sl}_n)$, adjoin an idempotent $1_{\mu}$ 
for each $\mu\in\bZ^{n-1}$ and add the relations
\begin{align*}
1_{\mu}1_{\nu} &= \delta_{\mu,\nu}1_{\lambda},   
\\
E_{\pm i}1_{\mu} &= 1_{\mu\pm\overline{\alpha}_i}E_{\pm i},\quad\text{with}\;
\overline{\alpha}_i=\alpha_i-\alpha_{i+1},
\\
K_iK^{-1}_{i+1}1_{\mu} &= q^{\mu_i}1_{\mu}.
\end{align*}
\begin{defn} The idempotented quantum special linear algebra is defined by 
\[
\U=\bigoplus_{\mu,\nu\in\bZ^{n-1}}1_{\mu}{\mathbf U}_q(\mathfrak{sl}_n)1_{\nu}.
\]
\end{defn}
\noindent Note that $\Ugl$ and $\U$ are both non-unital algebras, 
because their units 
would have to be equal to the infinite sum of all their idempotents.
 
Furthermore, the only ${\mathbf U}_q(\mathfrak{gl}_n)$ 
and ${\mathbf U}_q(\mathfrak{sl}_n)$-representations which factor through 
$\Ugl$ and $\U$ respectively are the weight representations. 
Finally, note that there is no embedding of $\U$ into $\Ugl$, because 
there is no embedding of the $\mathfrak{sl}_n$-weights into the 
$\mathfrak{gl}_n$-weights.  

\subsubsection{The $q$-Schur algebra}

Let $d\in\bN$ and let $V$ be the natural $n$-dimensional representation of 
${\mathbf U}_q(\mathfrak{gl}_n)$. Define  
\[
\Lambda(n,d)=\{\lambda\in \bN^n\mid\,\, 
\sum_{i=1}^{n}\lambda_i=d\}\quad\text{and}
\]  
\[
\Lambda^+(n,d)=\{\lambda\in\Lambda(n,d)\mid d\geq 
\lambda_1\geq\lambda_2\geq\cdots 
\geq\lambda_n\geq 0\}.
\] 
Recall that the weights in $V^{\otimes d}$ are precisely the elements of 
$\Lambda(n,d)$, and that the highest weights are the elements of 
$\Lambda^+(n,d)$.  
The highest weights correspond exactly to the irreducibles $V_{\lambda}$ 
that show up in the decomposition of $V^{\otimes d}$. 

We can define the $q$-Schur algebra as follows. 
\begin{defn}
The \emph{$q$-Schur algebra} $S_q(n,d)$ is the image 
of the representation $\psi_{n,d}$ defined by
\[
\psi_{n,d}\colon {\mathbf U}_q(\mathfrak{gl}_n)\to 
\End_{\mathbb{C}}(V^{\otimes d}).
\]
\end{defn}
For $\lambda\in\Lambda^+(n,d)$, the 
${\mathbf U}_q(\mathfrak{gl}_n)$-action on $V_{\lambda}$ factors through 
the projection
\[
\psi_{n,d}\colon {\mathbf U}_q(\mathfrak{gl}_n)\to S_q(n,d).
\] 
This way we obtain all irreducible representations of $S_q(n,d)$. Note that 
this also implies that all representations of $S_q(n,d)$ have a 
weight decomposition. As a matter of fact, it is well-known that 
\[
S_q(n,d)\cong \prod_{\lambda\in\Lambda^+(n,d)}\End_{\mathbb{C}}(V_{\lambda}).
\]
Therefore $S_q(n,d)$ is a finite dimensional, semisimple, 
unital algebra and its dimension is equal to 
\[
\sum_{\lambda\in\Lambda^+(n,d)}\dim(V_{\lambda})^2=\binom{n^2+d-1}{d}. 
\]
Since $V^{\otimes d}$ is a weight representation, 
$\psi_{n,d}$ gives rise to a homomorphism 
$\Ugl\to S_q(n,d)$, for 
which we use the same notation. This map is still surjective and 
Doty and Giaquinto, in Theorem 2.4 of~\cite{D-G}, showed that 
the kernel of $\psi_{n,d}$ is equal to the ideal generated by all 
idempotents $1_{\lambda}$ such that 
$\lambda\not\in\Lambda(n,d)$. Clearly the image of $\psi_{n,d}$ is isomorphic to
$S_q(n,d)$. By the above observations, 
we see that $S_q(n,d)$ has a Serre presentation. As a matter of fact, 
by Corollary 4.3.2 in~\cite{C-G}, this presentation is simpler than 
that of $\Ugl$, i.e. one does not need to impose the last two Serre relations, 
involving cubical terms, because they are implied by the other relations 
and the finite 
dimensionality.  
\begin{lem} 
$S_q(n,d)$ is isomorphic to the associative, unital 
$\bQ(q)$-algebra generated by $1_{\lambda}$, for $\lambda\in\Lambda(n,d)$, 
and $E_{\pm i}$, for $i=1,\ldots,n-1$, subject to the relations
\begin{align}
\label{eq:schur1} 1_{\lambda}1_{\mu} &= \delta_{\lambda,\mu}1_{\lambda},
\\[0.5ex]
\label{eq:schur2}
\sum_{\lambda\in\Lambda(n,d)}1_{\lambda} &= 1,
\\[0.5ex]
\label{eq:schur3}
E_{\pm i}1_{\lambda} &= 1_{\lambda\pm\alpha_i}E_{\pm i},\quad\text{with}\;\alpha_i=\epsilon_i-\epsilon_{i+1}=(0,\dots ,1,-1,\dots ,0),
\\[0.5ex]
\label{eq:schur4}
E_iE_{-j}-E_{-j}E_i &= \delta_{ij}\sum\limits_{\lambda\in\Lambda(n,d)}
[\overline{\lambda}_i]1_{\lambda}.
\end{align}
We use the convention that $1_{\mu}X1_{\lambda}=0$, if $\mu$ 
or $\lambda$ is not contained in $\Lambda(n,d)$. Again $[a]$ denotes the 
quantum integer from before.
\end{lem}
\vspace*{0.15cm}

Although there is no embedding of $\U$ into $\Ugl$, the projection 
\[
\psi_{n,d}\colon{\mathbf U}_q(\mathfrak{gl}_n)\to S_q(n,d)
\] 
can be restricted to ${\mathbf U}_q(\mathfrak{sl}_n)$ and is still surjective. 
This gives rise to the surjection 
\[
\psi_{n,d}\colon \U\to S_q(n,d),
\]
defined by 
\begin{equation}
\label{eq:psi}
\psi_{n,d}(E_{\pm i}1_{\lambda})=E_{\pm i}1_{\phi_{n,d}(\lambda)},
\end{equation}
where $\phi_{n,d}$ was defined below equations~\eqref{eq:sl-gl-wts1} and 
\eqref{eq:sl-gl-wts2}. By convention we put $1_{*}=0$.   
\vskip0.5cm 

\subsubsection{The general and special quantum 2-algebras}
We note that a lot of this section is copied from~\cite{msv2}. The reader can find even more details there.
\vspace*{0.25cm}

Let $\Ucat$ be Khovanov and Lauda's~\cite{kl3} 
diagrammatic categorification of $\U$. 
In~\cite{msv2} it was shown that there is a quotient 2-category of 
$\Ucat$, denoted by $\Scat(n,n)$, which categorifies $S_q(n,n)$. 
\vspace*{0.15cm}

We recall the definition of these categorified quantum algebras and some notions from above. 
As before, let $I=\{1,2,\ldots,n-1\}$. Again, we use \textit{signed sequences} 
$\ii=(\alpha_1i_1,\ldots,\alpha_mi_m)$, 
for any $m\in\bN$, $\alpha_j\in\{\pm 1\}$ and $i_j\in I$, and 
the set of signed sequences is denoted $\sseq$.
\vspace*{0.15cm}
 
For $\ii=(\alpha_1i_1,\ldots,\alpha_mi_m)\in\sseq$ we 
define $\ii_{\Lambda}=\alpha_1 (i_1)_{\Lambda}+\cdots+\alpha_m (i_m)_{\Lambda}$, 
where 
\[
(i_j)_{\Lambda}=(0,0,\ldots,1,-1,0\ldots,0),
\]
such that the vector starts with $i_j-1$ and ends with $k-1-i_j$ zeros. 
We also define the symmetric $\bZ$-valued bilinear form on $\bC[I]$ 
by $i\cdot i=2$, $i\cdot (i+1)=-1$ and $i\cdot j=0$, for $\vert i-j\vert>1$. 
Recall that $\overline{\lambda}_i=\lambda_i-\lambda_{i+1}$.
\vspace*{0.25cm}

We first recall the definition, given in~\cite{msv2}, of the 2-category which 
conjecturally categorifies $\Ugl$. It is 
a straightforward adaptation of Khovanov and Lauda's $\Ucat$. 
\begin{defn} \label{def_glcat} $\glcat$ is an
additive $\bC$-linear 2-category. The 2-category $\glcat$ consists of
\begin{itemize}
  \item Objects are $\lambda\in\bZ^n$.
\end{itemize}
The hom-category $\glcat(\lambda,\lambda')$ between two objects 
$\lambda$, $\lambda'$ is an additive $\bC$-linear category 
consisting of the following.
\begin{itemize}
  \item Objects\footnote{We refer to objects of the category
$\glcat(\lambda,\lambda')$ as 1-morphisms of $\glcat$.  Likewise, the morphisms of
$\glcat(\lambda,\lambda')$ are called 2-morphisms in $\glcat$. } of
$\glcat(\lambda,\lambda')$, i.e. for 1-morphism in $\glcat$ from $\lambda$ to $\lambda'$
is a formal finite direct sum of 1-morphisms
  \[
 \mathcal{E}_{\ii} \onel\{t\} = \onelp \mathcal{E}_{\ii} \onel\{t\}
= \mathcal{E}_{\alpha_1 i_1}\dotsm\mathcal{E}_{\alpha_m i_m} \onel\{t\}
  \]
for any $t\in \bZ$ and signed sequence $\ii\in\sseq$ such that 
$\lambda'=\lambda+\ii_{\Lambda}$ and $\lambda$, $\lambda'\in\bZ^n$. 
  \item Morphisms of $\glcat(\lambda,\lambda')$, i.e. for 1-morphisms $\mathcal{E}_{\ii} \onel\{t\}$
and  $\mathcal{E}_{\jj} \onel\{t'\}$ in $\glcat$, the hom
sets $\glcat(\mathcal{E}_{\ii} \onel\{t\},\mathcal{E}_{\jj} \onel\{t'\})$ of
$\glcat(\lambda,\lambda')$ are graded $\bC$-vector spaces given by linear
combinations of degree $t-t'$ diagrams, modulo certain relations, built from
compo\-sites of the following.
\begin{enumerate}[(i)]
  \item  Degree zero identity 2-morphisms $1_x$ for each 1-morphism $x$ in
$\glcat$; the identity 2-morphisms $1_{\mathcal{E}_{+i} \onel}\{t\}$ and
$1_{\mathcal{E}_{-i} \onel}\{t\}$, for $i \in I$, are represented graphically by
\[
\begin{array}{ccc}
  1_{\mathcal{E}_{+i} \onel\{t\}} &\quad  & 1_{\mathcal{E}_{-i} \onel\{t\}} \\ \\
   \lambda + i_{\Lambda}\xy
 (0,0)*{\includegraphics[width=09px]{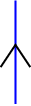}};
 (0,8)*{\scriptstyle i};
 (0,-8)*{\scriptstyle i};
 \endxy\lambda
 & &
 \;\;   
   \lambda - i_{\Lambda}\xy
 (0,0)*{\includegraphics[width=09px]{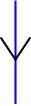}};
 (0,8)*{\scriptstyle i};
 (0,-8)*{\scriptstyle i};
 \endxy\lambda,
\\ \\
   \;\;\text{ {\rm deg} 0}\;\;
 & &\;\;\text{ {\rm deg} 0}\;\;
\end{array}
\]
for any $\lambda + i_{\Lambda} \in\bZ^n$ and any 
$\lambda - i_{\Lambda} \in \bZ^n$, respectively.

More generally, for a signed sequence $\ii=(\alpha_1i_1, \alpha_2i_2, \ldots
\alpha_mi_m)$, the identity $1_{\mathcal{E}_{\ii} \onel\{t\}}$ 2-morphism is
represented as
\begin{equation*}
\begin{array}{ccc}
  \lambda + \ii_{\Lambda}\xy
 (0,0)*{\includegraphics[width=0.7px]{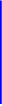}};
 (0,8)*{\scriptstyle i_1};
 (0,-8)*{\scriptstyle i_1};
 \endxy\,\xy
 (0,0)*{\includegraphics[width=0.7px]{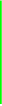}};
 (0,8)*{\scriptstyle i_2};
 (0,-8)*{\scriptstyle i_2};
 \endxy\cdots\xy
 (0,0)*{\includegraphics[width=0.7px]{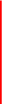}};
 (0,8)*{\scriptstyle i_m};
 (0,-8)*{\scriptstyle i_m};
 \endxy\lambda,
\end{array}
\end{equation*}
where the strand labeled $i_{k}$ is oriented up if $\alpha_{k}=+$
and oriented down if $\alpha_{k}=-$. We will often place labels with no
sign on the side of a strand and omit the labels at the top and bottom.  The
signs can be recovered from the orientations on the strands. 

\item Recall that $-\cdot -$ is the bilinear form from above. For each $\lambda \in \bZ^n$ the 2-morphisms 
\[
\begin{tabular}{|l|c|c|c|c|}
\hline
 {\bf Notation:} \xy (0,-5)*{};(0,7)*{}; \endxy& 
\xy
 (0,0)*{\includegraphics[width=7px]{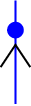}};
 (2,-5)*{\scriptstyle i,\lambda};
 \endxy &
\xy
 (0,0)*{\includegraphics[width=7px]{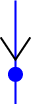}};
 (2,-5)*{\scriptstyle i,\lambda};
 \endxy  &  \xy
 (0,0)*{\includegraphics[width=20px]{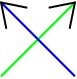}};
 (8,-4.5)*{\scriptstyle i,j,\lambda};
 \endxy  
 & 
\xy
 (0,0)*{\includegraphics[width=20px]{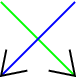}};
 (8,-4.5)*{\scriptstyle i,j,\lambda};
 \endxy\\
 \hline
 {\bf 2-morphism:} \xy (0,-5)*{};(0,9)*{}; \endxy&   \, \xy
 (0,0)*{\includegraphics[width=09px]{section23/upsimpledot}};
 (1.5,-5)*{\scriptstyle i};
 (3,0)*{\scriptstyle\lambda};
 (-5,0)*{\scriptstyle\lambda+i_{\Lambda}};
 \endxy\,
 &
    \, \xy
 (0,0)*{\includegraphics[width=09px]{section23/downsimpledot}};
 (1.5,-5)*{\scriptstyle i};
 (-3,0)*{\scriptstyle\lambda};
 (5,0)*{\scriptstyle \lambda+i_{\Lambda}};
 \endxy\,
 &
   \xy
 (0,1)*{\includegraphics[width=25px]{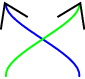}};
 (-5,-3)*{\scriptstyle i};
 (5,-3)*{\scriptstyle j};
 (5.5,0)*{\scriptstyle\lambda};
 \endxy\,
 &
   \xy
 (0,1)*{\includegraphics[width=25px]{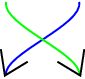}};
 (-5,-4)*{\scriptstyle i};
 (5,-4)*{\scriptstyle j};
 (5.5,0)*{\scriptstyle\lambda};
 \endxy\,
\\ & & & &\\
\hline
 {\bf Degree:} & \;\;\phantom{.a-}\text{$i\cdot i$}\;\;\phantom{.a-}
 &\;\;\phantom{.a-}\text{$i\cdot i$}\;\;\phantom{.a-}& \;\phantom{.a}\text{$-i\cdot j$}\;\;\phantom{..-}
 & \;\,\phantom{.a}\text{$-i\cdot j$}\;\,\phantom{..-} \\
 \hline
\end{tabular}
\]

\[
\begin{tabular}{|l|c|c|c|c|}
\hline
 {\bf Notation:} \xy (0,-5)*{};(0,7)*{}; \endxy& \xy
 (0,0)*{\includegraphics[width=20px]{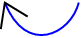}};
 (6,-2)*{\scriptstyle i,\lambda};
 \endxy &
 \xy
 (0,0)*{\includegraphics[width=20px]{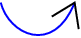}};
 (6,-2)*{\scriptstyle i,\lambda};
 \endxy  &  \xy
 (0,0)*{\includegraphics[width=20px]{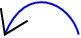}};
 (6,-2)*{\scriptstyle i,\lambda};
 \endxy  
 & \xy
 (0,0)*{\includegraphics[width=20px]{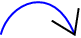}};
 (6,-2)*{\scriptstyle i,\lambda};
 \endxy  \\
 \hline
 {\bf 2-morphism:} \xy (0,-5)*{};(0,7)*{}; \endxy&   \xy
 (0,0)*{\includegraphics[width=25px]{section23/leftcup}};
 (0,-3)*{\scriptstyle i};
 (5,0)*{\scriptstyle\lambda};
 \endxy\,
 &
    \xy
 (0,0)*{\includegraphics[width=25px]{section23/rightcup}};
 (0,-3)*{\scriptstyle i};
 (5,0)*{\scriptstyle\lambda};
 \endxy\,
 &
   \xy
 (0,0)*{\includegraphics[width=25px]{section23/leftcap}};
 (0,3)*{\scriptstyle i};
 (5,0)*{\scriptstyle\lambda};
 \endxy\,
 &
   \xy
 (0,0)*{\includegraphics[width=25px]{section23/rightcap}};
 (0,3)*{\scriptstyle i};
 (5,0)*{\scriptstyle\lambda};
 \endxy\,
\\ & & & &\\
\hline
 {\bf Degree:} \xy (0,-1)*{};(0,5)*{}; \endxy& \phantom{.m}\text{$1-\overline{\lambda}_i$}\phantom{.m}
 &\phantom{.m}\text{$1+\overline{\lambda}_i$}\phantom{.m}&\phantom{.m}\text{$1+\overline{\lambda}_i$}\phantom{.m}
 & \phantom{.m}\text{$1-\overline{\lambda}_i$}\phantom{.m} \\
 \hline
\end{tabular}
\]
\end{enumerate}

\item Biadjointness and cyclicity.
\begin{enumerate}[(i)]
\item\label{it:sl2i}  $\mathbf{1}_{\lambda+i_{\Lambda}}\mathcal{E}_{+i}\onel$ and
$\onel\mathcal{E}_{-i}\mathbf{1}_{\lambda+i_{\Lambda}}$ are biadjoint, up to grading shifts, i.e.
\begin{equation} \label{eq_biadjoint1}
  \xy
 (0,0)*{\includegraphics[width=50px]{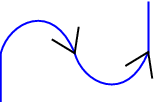}};
 (-4,5.5)*{\scriptstyle \lambda+i_{\Lambda}};
 (4,-5.5)*{\scriptstyle \lambda};
 \endxy
    \; =
    \;
\, \xy
 (0,0)*{\includegraphics[width=09px]{section23/upsimple}};
 (1.5,-5)*{\scriptstyle i};
 (3,0)*{\scriptstyle\lambda};
 (-5,0)*{\scriptstyle\lambda+i_{\Lambda}};
 \endxy\,
\qquad \quad  \xy
 (0,0)*{\includegraphics[width=50px]{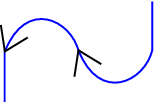}};
 (-4,5.5)*{\scriptstyle \lambda};
 (5,-5.5)*{\scriptstyle \lambda+i_{\Lambda}};
 \endxy
    \; =
    \;
\, \xy
 (0,0)*{\includegraphics[width=09px]{section23/downsimple}};
 (1.5,-5)*{\scriptstyle i};
 (-3,0)*{\scriptstyle\lambda};
 (5,0)*{\scriptstyle\lambda+i_{\Lambda}};
 \endxy\,
\end{equation}

\begin{equation} \label{eq_biadjoint2}
  \xy
 (0,0)*{\includegraphics[width=50px]{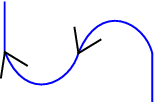}};
 (-4,-5.5)*{\scriptstyle \lambda+i_{\Lambda}};
 (4,5.5)*{\scriptstyle \lambda};
 \endxy
    \; =
    \;
\, \xy
 (0,0)*{\includegraphics[width=09px]{section23/upsimple}};
 (1.5,-5)*{\scriptstyle i};
 (3,0)*{\scriptstyle\lambda};
 (-5,0)*{\scriptstyle\lambda+i_{\Lambda}};
 \endxy\,
\qquad \quad  \xy
 (0,0)*{\includegraphics[width=50px]{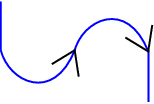}};
 (-4,-5.5)*{\scriptstyle \lambda};
 (5,5.5)*{\scriptstyle \lambda+i_{\Lambda}};
 \endxy
    \; =
    \;
\, \xy
 (0,0)*{\includegraphics[width=09px]{section23/downsimple}};
 (1.5,-5)*{\scriptstyle i};
 (-3,0)*{\scriptstyle\lambda};
 (5,0)*{\scriptstyle\lambda+i_{\Lambda}};
 \endxy\,
\end{equation}
\item
\begin{equation} \label{eq_cyclic_dot}
  \xy
 (0,0)*{\includegraphics[width=50px]{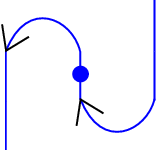}};
 (-4,8)*{\scriptstyle \lambda+i_{\Lambda}};
 (4,-8)*{\scriptstyle \lambda};
 \endxy
 \; =
    \;
\, \xy
 (0,0)*{\includegraphics[width=09px]{section23/downsimpledot}};
 (1.5,-5)*{\scriptstyle i};
 (-3,0)*{\scriptstyle\lambda};
 (5,0)*{\scriptstyle\lambda+i_{\Lambda}};
 \endxy\,
    \; =
    \;
  \xy
 (0,0)*{\includegraphics[width=50px]{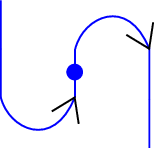}};
 (-4,-8)*{\scriptstyle \lambda};
 (4,8)*{\scriptstyle \lambda+i_{\Lambda}};
 \endxy    
\end{equation}
\item All 2-morphisms are cyclic with respect to the above biadjoint
   structure. This is ensured by the relations \eqref{eq_cyclic_dot}, and, for arbitrary $i,j$, the
   relations
\begin{equation} \label{eq_cyclic_cross-gen}
  \xy
 (0,0)*{\includegraphics[width=125px]{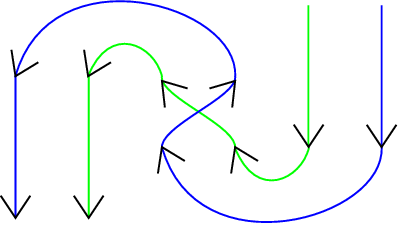}};
 (-12,-14)*{\scriptstyle j};
 (-20.3,-14)*{\scriptstyle i};
 (12,14)*{\scriptstyle j};
 (20.3,14)*{\scriptstyle i};
 (6,-0.5)*{\scriptstyle\lambda};
 \endxy\,
 \; =
    \;
\,\xy
 (0,1)*{\includegraphics[width=25px]{section23/downcrosscurved}};
 (-5,-4)*{\scriptstyle i};
 (5,-4)*{\scriptstyle j};
 (-5.5,0)*{\scriptstyle\lambda};
 \endxy
    \; =
    \;
  \xy
 (0,0)*{\includegraphics[width=125px]{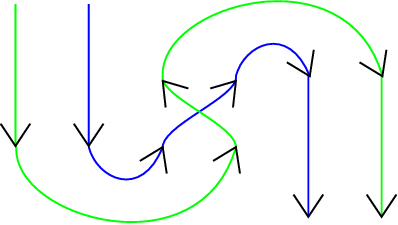}};
 (-12,14)*{\scriptstyle i};
 (-20.3,14)*{\scriptstyle j};
 (12,-14)*{\scriptstyle i};
 (20.3,-14)*{\scriptstyle j};
 (6,-1)*{\scriptstyle\lambda};
 \endxy\,.    
\end{equation}
Note that we can take either the first or the last diagram above as the 
definition of the up-side-down crossing. The cyclic condition on 
2-morphisms, expressed by \eqref{eq_cyclic_dot} and 
\eqref{eq_cyclic_cross-gen}, ensures that diagrams related by isotopy represent
the same 2-morphism in $\glcat$.

It will be convenient to introduce degree zero 2-morphisms.
\begin{equation} \label{eq_crossl-gen}
  \xy
 (0,0)*{\includegraphics[width=25px]{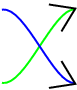}};
 (-5.5,4)*{\scriptstyle i};
 (-5.5,-4)*{\scriptstyle j};
 (5.5,0)*{\scriptstyle\lambda};
 \endxy\,
 \; =
    \;
 \xy
 (0,1)*{\includegraphics[width=75px]{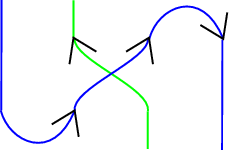}};
 (-4.3,11.5)*{\scriptstyle j};
 (-13.3,11.5)*{\scriptstyle i};
 (3.8,-9.5)*{\scriptstyle j};
 (12.5,-9.5)*{\scriptstyle i};
 (14,0)*{\scriptstyle\lambda};
 \endxy\,
    \; =
    \;
  \xy
 (0,0)*{\includegraphics[width=75px]{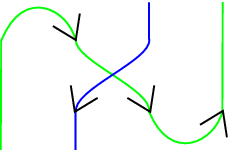}};
 (-4,-10.5)*{\scriptstyle i};
 (-13,-10.5)*{\scriptstyle j};
 (4,10.5)*{\scriptstyle i};
 (13,10.5)*{\scriptstyle j};
 (14,0)*{\scriptstyle\lambda};
 \endxy\,   
\end{equation}

\begin{equation} \label{eq_crossr-gen}
  \,\xy
 (0,0)*{\includegraphics[width=25px]{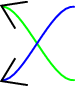}};
 (-5.5,0)*{\scriptstyle\lambda};
 (5.5,4)*{\scriptstyle i};
 (5.5,-4)*{\scriptstyle j};
 \endxy
 \; =
    \;
 \,\xy
 (0,1)*{\includegraphics[width=75px]{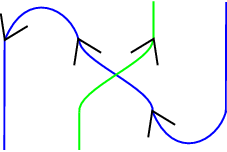}};
 (-4,-9.5)*{\scriptstyle j};
 (-13,-9.5)*{\scriptstyle i};
 (4.3,11.5)*{\scriptstyle j};
 (13.3,11.5)*{\scriptstyle i};
 (-14,0)*{\scriptstyle\lambda};
 \endxy\,
    \; =
    \;
  \,\xy
 (0,0)*{\includegraphics[width=75px]{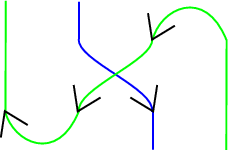}};
 (-4,10.5)*{\scriptstyle i};
 (-13,10.5)*{\scriptstyle j};
 (4,-10.5)*{\scriptstyle i};
 (13,-10.5)*{\scriptstyle j};
 (-14,0)*{\scriptstyle\lambda};
 \endxy\, ,    
\end{equation}
where the second equality in \eqref{eq_crossl-gen} and \eqref{eq_crossr-gen}
follow from \eqref{eq_cyclic_cross-gen}.  

\item All dotted bubbles of negative degree are zero. That is,
\begin{equation} \label{eq_positivity_bubbles}
 \xy
 (0,0)*{\includegraphics[width=30px]{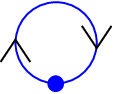}};
 (0,-5.5)*{\scriptstyle m};
 (-6,0)*{\scriptstyle i};
 (6,0)*{\scriptstyle\lambda};
 \endxy
  = 0,
 \qquad
  \text{if $m<\llambda_i-1$}, \qquad
 \xy
 (0,0)*{\includegraphics[width=30px]{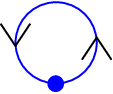}};
 (0,-5.5)*{\scriptstyle m};
 (-6,0)*{\scriptstyle i};
 (6,0)*{\scriptstyle\lambda};
 \endxy = 0,\quad
  \text{if $m< -\llambda_i-1$}
\end{equation}
for all $m \in \bZ_+$, where a dot carrying a label $m$ denotes the
$m$-fold iterated vertical composite of $\xy
 (0,0)*{\includegraphics[width=5px]{section23/upsimpledot}};
 (3,-3)*{\scriptstyle i,\lambda};
 \endxy$ or
$\xy
 (0,0)*{\includegraphics[width=5px]{section23/downsimpledot}};
 (3,-3)*{\scriptstyle i,\lambda};
 \endxy$ depending on the orientation.  A dotted bubble of degree
zero equals $\pm 1$, that is
\begin{equation}\label{eq:bubb_deg0}
 \xy
 (0,0)*{\includegraphics[width=30px]{section23/circleclock}};
 (0,-5.5)*{\scriptstyle m};
 (-6,0)*{\scriptstyle i};
 (6,0)*{\scriptstyle\lambda};
 \endxy
  = (-1)^{\lambda_{i+1}}, \quad \text{for $\llambda_i \geq 1$,}
  \qquad \quad
  \xy
 (0,0)*{\includegraphics[width=30px]{section23/circlecounter}};
 (0,-5.5)*{\scriptstyle m};
 (-6,0)*{\scriptstyle i};
 (6,0)*{\scriptstyle\lambda};
 \endxy
  = (-1)^{\lambda_{i+1}-1}, \quad \text{for $\llambda_i \leq -1$.}
\end{equation}
\item For the following relations we employ the convention that all summations
are increasing, so that a summation of the form $\sum_{f=0}^{m}$ is zero 
if $m < 0$.
\begin{eqnarray}
\label{eq:redtobubbles}
  \;\;\;\;\qquad\text{$\,\xy
 (0,0)*{\includegraphics[width=50px]{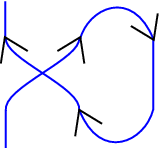}};
 (-8,0)*{\scriptstyle i};
 (9,0)*{\scriptstyle \lambda};
 \endxy$} \; = \; -\sum_{f=0}^{-\llambda_i}
   \xy
 (0,0)*{\includegraphics[width=9px]{section23/upsimpledot}};
 (1,-4.5)*{\scriptstyle i};
 (5.5,5)*{\scriptstyle -\overline{\lambda}_i-f};
 \endxy\,\xy
 (0,0)*{\includegraphics[width=30px]{section23/circleclock}};
 (0,-5.5)*{\scriptstyle \overline{\lambda}_i-1+f};
 (-6,0)*{\scriptstyle i};
 (6,0)*{\scriptstyle\lambda};
 \endxy
\quad\text{  and  } \quad
  \text{$\xy
 (0,0)*{\includegraphics[width=50px]{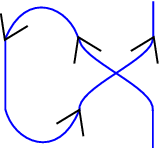}};
 (-9.5,0)*{\scriptstyle \lambda};
 (8,0)*{\scriptstyle i};
 \endxy$} \; = \;
 \sum_{g=0}^{\llambda_i}
   \xy
 (0,0)*{\includegraphics[width=30px]{section23/circlecounter}};
 (0,-5.5)*{\scriptstyle -\overline{\lambda}_i-1+g};
 (-6,0)*{\scriptstyle i};
 (6,0)*{\scriptstyle\lambda};
 \endxy\,\xy
 (0,0)*{\includegraphics[width=9px]{section23/upsimpledot}};
 (1,-4.5)*{\scriptstyle i};
 (5.5,5)*{\scriptstyle \overline{\lambda}_i-g};
 \endxy
\end{eqnarray}

\begin{equation}
\label{eq:EF}
 \xy
 (0,0)*{\includegraphics[width=9px]{section23/upsimple}};
 (-1,-3.5)*{\scriptstyle i};
 (-2.5,0)*{\scriptstyle\lambda};
 \endxy\,\xy
 (0,0)*{\includegraphics[width=9px]{section23/downsimple}};
 (1,-3.5)*{\scriptstyle i};
 (2.5,0)*{\scriptstyle\lambda};
 \endxy
  =
 \xy
 (0,0)*{\includegraphics[width=17px]{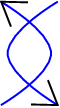}};
 (-3.9,-5.5)*{\scriptstyle i};
 (3.8,-5.5)*{\scriptstyle i};
 (3.2,0)*{\scriptstyle\lambda};
 \endxy
   - 
\sum_{f=0}^{\llambda_i-1} \sum_{g=0}^{f}
    \xy
 (0,0)*{\includegraphics[width=30px]{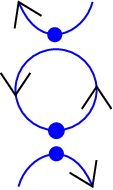}};
 (-6,0)*{\scriptstyle\lambda};
 (8,5)*{\scriptstyle \overline{\lambda}_i-1-f};
 (7,-4)*{\scriptstyle -\overline{\lambda}_i-1+g};
 (-5,-5)*{\scriptstyle f-g};
 (-4,4)*{\scriptstyle i};
 \endxy
 \quad\text{and}\quad
 \xy
 (0,0)*{\includegraphics[width=9px]{section23/downsimple}};
 (-1,-3.5)*{\scriptstyle i};
 (-2.5,0)*{\scriptstyle\lambda};
 \endxy\,\xy
 (0,0)*{\includegraphics[width=9px]{section23/upsimple}};
 (1,-3.5)*{\scriptstyle i};
 (2.5,0)*{\scriptstyle\lambda};
 \endxy
 = 
 \xy
 (0,0)*{\includegraphics[width=17px]{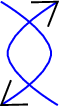}};
 (-3.9,-5.5)*{\scriptstyle i};
 (3.8,-5.5)*{\scriptstyle i};
 (3.2,0)*{\scriptstyle\lambda};
 \endxy
   - 
\sum_{f=0}^{-\llambda_i-1} \sum_{g=0}^{f}
    \xy
 (0,0)*{\includegraphics[width=30px]{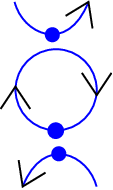}};
 (7,5)*{\scriptstyle -\overline{\lambda}_i-1-f};
 (-6,0)*{\scriptstyle\lambda};
 (8,-4)*{\scriptstyle \overline{\lambda}_i-1+g};
 (-5,-5)*{\scriptstyle f-g};
 (-4,4)*{\scriptstyle i};
 \endxy
\end{equation}
for all $\lambda\in \bZ^n$
(see~\eqref{eq_crossl-gen} and~\eqref{eq_crossr-gen} for the definition of sideways
crossings). 
Notice that for some values of 
$\lambda$ the dotted
bubbles appearing above have negative labels. A composite of $\xy
 (0,0)*{\includegraphics[width=5px]{section23/upsimpledot}};
 (3,-3)*{\scriptstyle i,\lambda};
 \endxy$
or $\xy
 (0,0)*{\includegraphics[width=5px]{section23/downsimpledot}};
 (3,-3)*{\scriptstyle i,\lambda};
 \endxy$ with itself a negative number of times does not make
sense. These dotted bubbles with negative labels, called {\em fake bubbles}, are
formal symbols inductively defined by the equation
\begin{equation}\label{eq_infinite_Grass}
\qquad\;\;\left(\xy
 (0,0)*{\includegraphics[width=30px]{section23/circlecounter}};
 (-4,4)*{\scriptstyle i};
 (4,4)*{\scriptstyle \lambda};
 (-1,-5)*{\scriptstyle -\overline{\lambda}_i-1};
 \endxy t^{0}+\cdots+\xy
 (0,0)*{\includegraphics[width=30px]{section23/circlecounter}};
 (-4,4)*{\scriptstyle i};
 (4,4)*{\scriptstyle \lambda};
 (0,-5)*{\scriptstyle -\overline{\lambda}_i-1+r};
 \endxy t^{r}+\cdots\right)\left(\xy
 (0,0)*{\includegraphics[width=30px]{section23/circleclock}};
 (-4,4)*{\scriptstyle i};
 (4,4)*{\scriptstyle \lambda};
 (-0.1,-5)*{\scriptstyle \overline{\lambda}_i-1};
 \endxy t^{0}+\cdots+\xy
 (0,0)*{\includegraphics[width=30px]{section23/circleclock}};
 (-4,4)*{\scriptstyle i};
 (4,4)*{\scriptstyle \lambda};
 (0,-5)*{\scriptstyle \overline{\lambda}_i-1+r};
 \endxy t^{r}+\cdots\right)=-1
\end{equation}
and the additional condition
\[
\xy
 (0,0)*{\includegraphics[width=30px]{section23/circleclock}};
 (-4,4)*{\scriptstyle i};
 (4,4)*{\scriptstyle \lambda};
 (0,-5)*{\scriptstyle -1};
 \endxy=(-1)^{\lambda_{i+1}}\quad\text{and}\quad\xy
 (0,0)*{\includegraphics[width=30px]{section23/circlecounter}};
 (-4,4)*{\scriptstyle i};
 (4,4)*{\scriptstyle \lambda};
 (0,-5)*{\scriptstyle -1};
 \endxy=(-1)^{\lambda_{i+1}-1},\quad\text{if}\;\overline{\lambda}_i=0.
\]
Although the labels are negative for fake bubbles, one can check that the overall
degree of each fake bubble is still positive, so that these fake bubbles do not
violate the positivity of dotted bubble axiom. The above equation, called the
infinite Grassmannian relation, remains valid even in high degree when most of
the bubbles involved are not fake bubbles.  

\item The NilHecke relations are
\begin{equation}\label{eq_nil_rels}
 \xy
 (0,0)*{\includegraphics[width=17px]{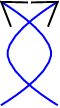}};
 (-3.9,-5.5)*{\scriptstyle i};
 (3.8,-5.5)*{\scriptstyle i};
 (3.2,0)*{\scriptstyle\lambda};
 \endxy=0,\quad\xy
 (0,0)*{\includegraphics[width=50px]{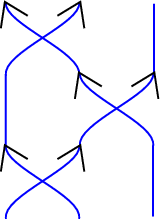}};
 (-9.5,-11.5)*{\scriptstyle i};
 (1,-11.5)*{\scriptstyle i};
 (7.3,-11.5)*{\scriptstyle i};
 (10,0)*{\scriptstyle \lambda};
 \endxy=\xy
 (0,0)*{\includegraphics[width=50px]{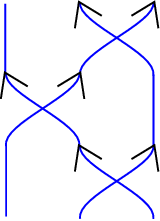}};
 (-7.3,-11.5)*{\scriptstyle i};
 (-1,-11.5)*{\scriptstyle i};
 (9.5,-11.5)*{\scriptstyle i};
 (10,0)*{\scriptstyle \lambda};
 \endxy
\end{equation}

\begin{equation}\label{eq_nil_dotslide}
 \xy
 (0,0)*{\includegraphics[width=9px]{section23/upsimple}};
 (-1,-3.5)*{\scriptstyle i};
 \endxy\,\xy
 (0,0)*{\includegraphics[width=9px]{section23/upsimple}};
 (1,-3.5)*{\scriptstyle i};
 (2.5,0)*{\scriptstyle\lambda};
 \endxy\;=\;\xy
 (0,1)*{\includegraphics[width=25px]{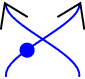}};
 (-5,-3)*{\scriptstyle i};
 (5,-3)*{\scriptstyle i};
 (5.5,0)*{\scriptstyle\lambda};
 \endxy\;-\xy
 (0,1)*{\includegraphics[width=25px]{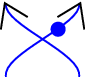}};
 (-5,-3)*{\scriptstyle i};
 (5,-3)*{\scriptstyle i};
 (5.5,0)*{\scriptstyle\lambda};
 \endxy\;=\xy
 (0,1)*{\includegraphics[width=25px]{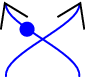}};
 (-5,-3)*{\scriptstyle i};
 (5,-3)*{\scriptstyle i};
 (5.5,0)*{\scriptstyle\lambda};
 \endxy\;-\xy
 (0,1)*{\includegraphics[width=25px]{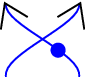}};
 (-5,-3)*{\scriptstyle i};
 (5,-3)*{\scriptstyle i};
 (5.5,0)*{\scriptstyle\lambda};
 \endxy.
\end{equation}
\end{enumerate}

\item We have for $i \neq j$
\begin{equation} \label{eq_downup_ij-gen}
 \xy
 (0,0)*{\includegraphics[width=17px]{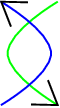}};
 (-3.9,-5.5)*{\scriptstyle i};
 (3.8,-5.5)*{\scriptstyle j};
 (3.2,0)*{\scriptstyle\lambda};
 \endxy=\xy
 (0,0)*{\includegraphics[width=9px]{section23/upsimple}};
 (-1,-3.5)*{\scriptstyle i};
 \endxy\,\xy
 (0,0)*{\includegraphics[width=9px]{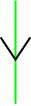}};
 (1,-3.5)*{\scriptstyle j};
 (2.5,0)*{\scriptstyle\lambda};
 \endxy\quad\text{and}\quad\xy
 (0,0)*{\includegraphics[width=17px]{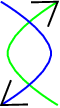}};
 (-3.9,-5.5)*{\scriptstyle i};
 (3.8,-5.5)*{\scriptstyle j};
 (3.2,0)*{\scriptstyle\lambda};
 \endxy=\xy
 (0,0)*{\includegraphics[width=9px]{section23/downsimple}};
 (-1,-3.5)*{\scriptstyle i};
 \endxy\,\xy
 (0,0)*{\includegraphics[width=9px]{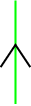}};
 (1,-3.5)*{\scriptstyle j};
 (2.5,0)*{\scriptstyle\lambda};
 \endxy
\end{equation}

\item \begin{enumerate}[(i)]
\item We have for $i \neq j$
\begin{equation}\label{eq_r2_ij-gen}
 \xy
 (0,0)*{\includegraphics[width=17px]{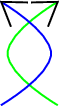}};
 (-3.9,-5.5)*{\scriptstyle i};
 (3.8,-5.5)*{\scriptstyle j};
 (3.2,0)*{\scriptstyle\lambda};
 \endxy =\begin{cases}\phantom{(i-j)(d}
 \xy
 (0,0)*{\includegraphics[width=9px]{section23/upsimple}};
 (-1,-3.5)*{\scriptstyle i};
 \endxy\,\xy
 (0,0)*{\includegraphics[width=9px]{section23/upsimplegreen}};
 (1,-3.5)*{\scriptstyle j};
 (2.5,0)*{\scriptstyle\lambda};
 \endxy, & \text{if}\,i\cdot j=0,\\[3ex]
 (i-j)\left(\xy
 (0,0)*{\includegraphics[width=9px]{section23/upsimpledot}};
 (-1,-3.5)*{\scriptstyle i};
 \endxy\,\xy
 (0,0)*{\includegraphics[width=9px]{section23/upsimplegreen}};
 (1,-3.5)*{\scriptstyle j};
 (2.5,0)*{\scriptstyle\lambda};
 \endxy-\xy
 (0,0)*{\includegraphics[width=9px]{section23/upsimple}};
 (-1,-3.5)*{\scriptstyle i};
 \endxy\,\xy
 (0,0)*{\includegraphics[width=9px]{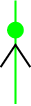}};
 (1,-3.5)*{\scriptstyle j};
 (2.5,0)*{\scriptstyle\lambda};
 \endxy\right), & \text{if}\,i\cdot j=-1.
 \end{cases}
\end{equation}
Notice that $(i-j)$ is just a sign, which takes into account the standard 
orientation of the Dynkin diagram.

\begin{equation}\label{eq_dot_slide_ij-gen}
\xy
 (0,1)*{\includegraphics[width=25px]{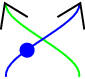}};
 (-5,-3)*{\scriptstyle i};
 (5,-3)*{\scriptstyle j};
 (5.5,0)*{\scriptstyle\lambda};
 \endxy\;=\xy
 (0,1)*{\includegraphics[width=25px]{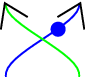}};
 (-5,-3)*{\scriptstyle i};
 (5,-3)*{\scriptstyle j};
 (5.5,0)*{\scriptstyle\lambda};
 \endxy\quad\text{and}\quad\xy
 (0,1)*{\includegraphics[width=25px]{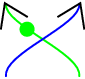}};
 (-5,-3)*{\scriptstyle i};
 (5,-3)*{\scriptstyle j};
 (5.5,0)*{\scriptstyle\lambda};
 \endxy\;=\xy
 (0,1)*{\includegraphics[width=25px]{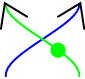}};
 (-5,-3)*{\scriptstyle i};
 (5,-3)*{\scriptstyle j};
 (5.5,0)*{\scriptstyle\lambda};
 \endxy.
\end{equation}

\item Unless $i = k$ and $i \cdot j=-1$, we have
\begin{equation}\label{eq_r3_easy-gen}
 \xy
 (0,0)*{\includegraphics[width=50px]{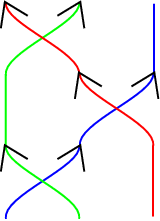}};
 (-9.1,-11.5)*{\scriptstyle i};
 (1,-11.5)*{\scriptstyle j};
 (7.1,-11.5)*{\scriptstyle k};
 (10,0)*{\scriptstyle \lambda};
 \endxy=\xy
 (0,0)*{\includegraphics[width=50px]{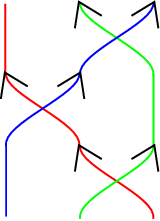}};
 (-9.3,-11.5)*{\scriptstyle i};
 (-1,-11.5)*{\scriptstyle j};
 (9.1,-11.5)*{\scriptstyle k};
 (10,0)*{\scriptstyle \lambda};
 \endxy
\end{equation}

\item We have for $i \cdot j =-1$
\begin{equation}\label{eq_r3_hard-gen}
\xy
 (0,0)*{\includegraphics[width=50px]{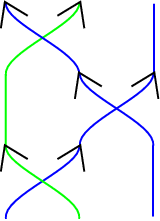}};
 (-9.1,-11.5)*{\scriptstyle i};
 (1,-11.5)*{\scriptstyle j};
 (7.1,-11.5)*{\scriptstyle i};
 (10,0)*{\scriptstyle \lambda};
 \endxy -\xy
 (0,0)*{\includegraphics[width=50px]{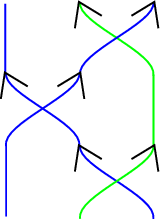}};
 (-9.3,-11.5)*{\scriptstyle i};
 (-1,-11.5)*{\scriptstyle j};
 (9.1,-11.5)*{\scriptstyle i};
 (10,0)*{\scriptstyle \lambda};
 \endxy =
 (i-j)\xy
 (0,0)*{\includegraphics[width=60px]{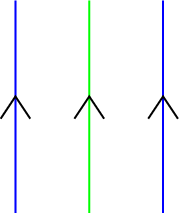}};
 (-9.4,-11.5)*{\scriptstyle i};
 (1,-11.5)*{\scriptstyle j};
 (7.7,-11.5)*{\scriptstyle i};
 (12,0)*{\scriptstyle \lambda};
 \endxy.
\end{equation}
\end{enumerate}
\item The additive, linear composition functor $\glcat(\lambda,\lambda')
\times \glcat(\lambda',\lambda'') \to \glcat(\lambda,\lambda'')$ is given on
1-morphisms of $\glcat$ by
\begin{equation}
  \mathcal{E}_{\jj}\mathbf{1}_{\lambda'}\{t'\} \times \mathcal{E}_{\ii}\onel\{t\} \mapsto
  \mathcal{E}_{\jj\ii}\onel\{t+t'\}
\end{equation}
for $\ii_{\Lambda}=\lambda-\lambda'$, and on 2-morphisms of $\glcat$ by juxtaposition of
diagrams, e.g.
\[
\left(\xy
 (0,0)*{\includegraphics[width=95px]{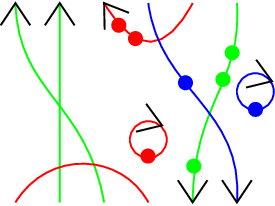}};
 (-16,0)*{\scriptstyle \lambda};
 (19,0)*{\scriptstyle\lambda'};
 \endxy\right)\times\left(\xy
 (0,0)*{\includegraphics[width=35px]{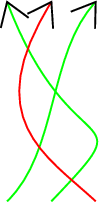}};
 (-8,0)*{\scriptstyle \lambda'};
 (8,0)*{\scriptstyle\lambda''};
 \endxy\right)\mapsto\xy
 (0,0)*{\includegraphics[width=125px]{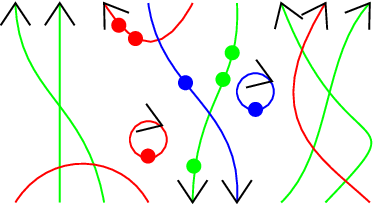}};
 (-20,0)*{\scriptstyle \lambda};
 (22,0)*{\scriptstyle\lambda''};
 \endxy.
\]
\end{itemize}
\end{defn}
This concludes the definition of $\glcat$. 
\vskip0.5cm
Note that for two $1$-morphisms $x$ and $y$ in 
$\glcat$ the $2$-hom-space $\HomGL(x,y)$ 
only contains $2$-morphisms of degree zero and is therefore finite dimensional. 
Following Khovanov and Lauda we introduce the graded $2$-hom-space 
\[
\HOMGL(x,y)=\bigoplus_{t\in\bZ}\HomGL(x\{t\},y),
\]
which is infinite dimensional. We also define the $2$-category 
$\glcat^*$ which has the same objects and $1$-morphisms as $\glcat$, 
but for two $1$-morphisms $x$ and $y$ the vector space of $2$-morphisms is 
defined by 
\begin{equation}
\label{eq:ast}
\glcat^*(x,y)=\HOMGL(x,y).
\end{equation}

It should be noted that $\Ucat$ is defined just as $\glcat$, but labeling 
all the regions of the diagrams with $\mathfrak{sl}_n$-weights, i.e.
elements of 
$\mathbb{Z}^{n-1}$. Note that one also has to re-normalize the signs of the 
left cups and caps, so that the bubble relations all become 
dependent on the $\mathfrak{sl}_n$-weights. For much more details, see~\cite{msv2}. 
\subsubsection{The $q$-Schur $2$-algebra}
The categorification of $S_q(n,n)$ is now obtained from 
$\glcat$ by taking a quotient. 
\begin{defn}
The $2$-category $\Scat(n,n)$ is the quotient of $\glcat$ by the ideal 
generated by all $2$-morphisms containing a region with a label not in 
$\Lambda(n,n)$. 
\end{defn}

We remark that we only put real bubbles, whose interior has a label outside 
$\Lambda(n,n)$, equal to zero. To see what happens to a fake bubble, one 
first has to write it in terms of real bubbles with the opposite orientation 
using the infinite Grassmannian relation~\eqref{eq_infinite_Grass}.

A main result of~\cite{msv2}, given in Theorem 7.11 in that paper, is the following.  
\begin{thm}
Let $\dot{\mathcal S}(n,n)$ denote the Karoubi envelope of 
${\mathcal S}(n,n)$. The 
$\bQ(q)$-linear map 
\[
\gamma_S\colon S_q(n,n)\to K^{\oplus}_0(\dot{\mathcal S}(n,n))_{\bQ(q)},
\]
determined by 
\[
\gamma_S(E_{\ii}1_{\lambda})=[\mathcal{E}_{\ii}1_{\lambda}]
\]
is an isomorphism of algebras. 
\end{thm}

Recall also (see Definition 4.1 in~\cite{msv2}) that there is an 
essentially surjective and full additive $2$-functor
\[
\Psi_{n,n}\colon \Ucat\to \Scat(n,n),
\]
whose precise definition is not relevant here. Up to signs related to 
cups and caps, it is obtained by mapping any string diagram to itself 
and applying $\phi_{n,n}$ to the labels of the regions. By convention, any diagram 
with a region labeled $*$ is taken to be zero. 
It is important to note that 
\[
K^{\oplus}_0(\Psi_{n,n})_{\bQ(q)}\colon K^{\oplus}_0(\UcatD)_{\bQ(q)}
\to K^{\oplus}_0(\ScatD(n,n))_{\bQ(q)}
\]
corresponds to the aforementioned surjective homomorphism
\[
\psi_{n,n}\colon \U\to S_q(n,n).
\] 
\subsubsection{The cyclotomic KLR-algebras}
In this subsection, we recall the definition of the cyclotomic KLR-algebras, 
due to Khovanov and Lauda~\cite{kl1} and, independently, 
to Rouquier~\cite{rou}. We also recall two important results about them. 

Fix $\nu\in\mathbb{Z}_{\leq 0}[I]$. Let $\mathrm{Seq}(\nu)$ be the set of all 
sequences $\ii=(-i_1,-i_2,\cdots,-i_m)$, such that $i_k\in I$ for 
each $k$ and $\nu_j=\#\{k\mid i_k=j\}$. 
\begin{defn} For any $\ii,\jj\in \mathrm{Seq}(\nu)$ and any 
$\mathfrak{gl}_n$-weight $\lambda\in\mathbb{Z}^{n}$, let 
\[
{}_{\ii}R^{\nu}_{\jj}\subset 
\mathrm{End}_{\glcat}({\mathcal E}_{\jj}1_{\lambda},{\mathcal E}_{\ii}1_{\lambda})
\]  
be the subalgebra containing only diagrams which are oriented downwards. So, 
only strands oriented downwards with dots and crossings are allowed. No 
strands oriented upwards, no cups and no caps. The relations 
in $\glcat$ involving only downward strands do not depend on $\lambda$. 
Therefore, the definition above makes sense. In~\cite{kl1}, 
the authors do not label the regions of the diagrams. 

Then $R^{\nu}$ is defined as 
\[
R^{\nu}=\bigoplus_{\ii,\jj\in \mathrm{Seq}(\nu)} {}_{\ii}R^{\nu}_{\jj}.
\]

The ring $R$ is defined as 
\[
R=\bigoplus_{\nu\in\mathbb{Z}_{\leq 0}[I]}R^{\nu}.
\]
\end{defn}
As remarked above, the definition of $R^{\nu}$ does not depend on $\lambda$. 
However, when we use a particular $\lambda$, we will write $R^{\nu}1_{\lambda}$. 

Note that $R^{\nu}$ is unital, whereas $R$ has infinitely many idempotents. 

Let $R^{\nu}\text{-}\mathrm{p\textbf{Mod}}_{\mathrm{gr}}$ be the category of graded, 
finitely generated, projective $R^{\nu}$-modules and define 
\[
R\text{-}\mathrm{p\textbf{Mod}}_{\mathrm{gr}}=
\bigoplus_{\nu\in\mathbb{Z}_{\leq 0}[I]}R^{\nu}\text{-}\mathrm{p\textbf{Mod}}_{\mathrm{gr}}.
\]

In Proposition 3.18 in~\cite{kl1}, Khovanov and Lauda showed that 
$R\text{-}\mathrm{p\textbf{Mod}}_{\mathrm{gr}}$ categorifies the negative half of 
$\U$ and $R^{\nu}\text{-}\mathrm{p\textbf{Mod}}_{\mathrm{gr}}$ categorifies the 
$\nu$-root space. 
\vskip0.5cm
We can now recall the definition of the \textit{cyclotomic KLR-algebras}. The reader can find more details in~\cite{kl1} or ~\cite{rou}, for example. 
\begin{defn}
Choose a dominant $\Ugl$-weight $\lambda\in\Lambda(n,n)^+$. 
Let $R_{\lambda}^{\nu}$ be the quotient algebra 
of $R^{\nu}1_{\lambda}$ by the ideal generated by all diagrams of the form
\begin{align*}
\xy
(0,0)*{\includegraphics[width=75px]{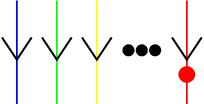}};
(-11,-8.5)*{\scriptstyle i_1};
(-5.5,-8.5)*{\scriptstyle i_2};
(0,-8.5)*{\scriptstyle i_3};
(12,-8.5)*{\scriptstyle i_m};
(14.5,-5)*{\scriptstyle \overline{\lambda}_m};
(14,0)*{\scriptstyle\lambda};
\endxy.
\end{align*}
Recall that $\overline{\lambda}_m=\lambda_m-\lambda_{m+1}$, the $m$-th entry 
of the $\mathfrak{sl}_n$-weight corresponding to $\lambda$. 

Define 
\[
R_{\lambda}=\bigoplus_{\nu\in\mathbb{Z}_{\leq 0}[I]}R_{\lambda}^{\nu}.
\]
\end{defn}

Note that we mod out by relations involving dots on the last strand, rather 
than the first strand as in~\cite{kl1}. This is to make the definition 
compatible with the other definitions in our paper. 

It turns out that $R_{\lambda}$ is a finite dimensional, unital algebra (Corollary 2.2 in~\cite{bk2}). Let 
$R_{\lambda}\text{-}\mathrm{p\textbf{Mod}}_{\mathrm{gr}}$ be its category of 
finite dimensional, graded, projective modules and  
$R_{\lambda}\text{-}\mathrm{\textbf{Mod}}_{\mathrm{gr}}$ its category of 
all finite dimensional modules.  
\vskip0.5cm
There is a strong $\mathfrak{sl}_n$-$2$-representation 
on $R_{\lambda}\text{-}\mathrm{\textbf{Mod}}_{\mathrm{gr}}$, 
which can be restricted to 
$R_{\lambda}\text{-}\mathrm{p\textbf{Mod}}_{\mathrm{gr}}$ 
(Section 4.4 in~\cite{bk}). This basically 
means that the $1$-morphisms of $\Ucat$ act as endofunctors and the 
$2$-functors as natural transformations, such that the relations in $\Ucat$ 
are preserved. For a precise definition of a strong $2$-representation 
see~\cite{cl}, which is almost equal to Rouquier's definition of 
a $2$-representation in~\cite{rou}) but satisfies one extra condition.  

Brundan and Kleshchev~\cite{bk} proved the following result, which was first conjectured by Khovanov and Lauda (see also~\cite{kakash},
~\cite{lv},~\cite{vv} and~\cite{we1}). From now on, we will always use the 
notation 
\[
\mathcal{V}_{\lambda}=R_{\lambda}\text{-}\mathrm{p\textbf{Mod}}_{\mathrm{gr}}.
\]
Since $\mathcal{V}_{\lambda}$ is a strong $\mathfrak{sl}_2$-$2$-representation, 
$K_0^{\oplus}(\mathcal{V}_{\lambda})$ is a $\U$-module. 
\begin{thm}
\label{thm:bk}
There exists a degree preserving isomorphism of $\UZ(\mathfrak{sl}_n)$-modules
\[
\gamma_V\colon V_{\lambda}\to K^{\oplus}_0(\mathcal{V}_{\lambda})_{\bQ(q)}.
\]
Here $V_{\lambda}$ is the irreducible $\UZ(\mathfrak{sl}_n)$-module with 
highest weight $\overline{\lambda}$.  
\end{thm}
In Section~\ref{sec:Grothendieck}, we will recall some extra properties of $\gamma_V$ which Brundan and 
Kleshchev showed. 

Rouquier showed that, in a certain sense, 
$R_{\lambda}$ is the universal categorification of 
$V_{\lambda}$. For a proof of the following 
result, see Lemma 5.4, Proposition 5.6 and Corollary 5.7 in~\cite{rou}.

\begin{prop}
\label{prop:rouquier}
Let $\mathcal V$ be any additive, idempotent complete category, which allows 
an integrable, graded, categorical action by $\Ucat$ (for the precise definition 
see~\cite{rou}). Suppose $V_h$ is a highest weight object in $\mathcal V$, i.e 
an object that is killed by ${\mathcal E}_{+i}$, for all $i\in I$, and 
$\mathrm{End}_{\mathcal V}(V_h)\cong \mathbb{C}$. Suppose also 
that any object in $\mathcal V$ is a direct summand of $XV_h$, for some 
object $X\in\Ucat$. Then there exists an equivalence of 
categorical $\Ucat$-representations
\[
\Phi\colon \mathcal{V}_{\lambda}\to {\mathcal V}.
\]
\end{prop}
We will not spell out the precise definition of $\Phi$ in this paper. Let us 
just remark that, up to natural isomorphism, $\Phi$ is uniquely determined by 
the fact that it sends the highest weight object to the highest weight object 
and intertwines the $\mathfrak{sl}_n$-$2$-representations.  

There are some subtle differences between 
Rouquier's approach to categorification and Khovanov and Lauda's. However, 
Proposition~\ref{prop:rouquier} holds in both setups, as already remarked by 
Webster in Section 1.4 in~\cite{we1}. The proof of Proposition~\ref{prop:rouquier} consists of Rouquier's remarks in Section 5.1.2 and of the 
contents of his proofs of Lemma 5.4 and Proposition 5.6 in~\cite{rou}, which 
only rely on the assumptions in the statement of our 
Proposition~\ref{prop:rouquier} and the fact that $\mathcal{E}_{+i}$ and 
$\mathcal{E}_{-i}$ are biadjoint in $\Ucat$, for any $i\in I$. The precise 
definition of the units and the counits, i.e. the cups and the caps, 
is not relevant for 
the validity of the proof. Note that we have included the hypothesis 
\[
\mathrm{End}_{\mathcal V}(V_h)\cong \mathbb{C}
\] 
in Proposition~\ref{prop:rouquier}, which is not one of 
Rouquier's assumptions. There are categorifications of 
$V_{\lambda}$ without that property. See Conjecture 7.16 in~\cite{msv2} for 
example. However, in order to get a categorification which is 
really equivalent to $\mathcal{V}_{\lambda}$, 
i.e. with hom-spaces of the same graded dimension, 
one needs to add that assumption because it holds in the latter category.


\section{The $\mathfrak{sl}_{3}$-web algebra ${\mathcal W}^S_c$}
\label{sec:webalgebra}
\setcounter{subsection}{1}
For the rest of this section, let $S$ be a fixed sign string of length 
$\ell(S)$. We are going to define the \textit{web algebra} ${\mathcal W}^S_c$.

\begin{defn}
\textbf{(Web algebra)} \label{defn:webalg} For $u,v\in B^S$, we define 
\[
{}_{u}\mathcal{W}^c_{v}=\F^c(u^*v)\{\ell(S)\},
\] 
where $\{\ell\}$ denotes a grading shift upwards by $\ell(S)$ degrees.

The \textit{web algebra} ${\mathcal W}^S_c$ is defined by 
\[
{\mathcal W}^S_c=\bigoplus_{u,v\in B^S}{}_{u}\mathcal{W}^c_{v}.
\]
 
The multiplication on ${\mathcal W}^S_c$ is defined by taking 
\[
{}_{u}{\mathcal W}^c_{v_1}\otimes {}_{v_2}{\mathcal W}^c_{w} \to 
{}_{u}{\mathcal W}^c_{w}
\]
to be zero, if $v_1\ne v_2$, and by the map to be defined in Definition~\ref{multfoam}, 
if $v_1=v_2=v$. 
\end{defn}

\begin{rem}
In Proposition~\ref{prop:multqgrade} we prove that the 
multiplication foam always has
degree $n$, so the degree shift in the definition above 
makes ${\mathcal W}^S_0$ into a graded algebra and, for any $c\ne 0$, it 
makes ${\mathcal W}^S_c$ into a filtered algebra. 
\end{rem}

\begin{defn} \textbf{(Multiplication of closed webs)}
\label{multfoam}
The \textit{multiplication} 
\[
{}_{u}{\mathcal W}^c_{v}\otimes {}_{v}\mathcal{W}^c_{w} \to 
{}_{u}{\mathcal W}^c_{w}
\] 
is induced by the \textit{multiplication foam}  
\[
m_{u,v,w}\colon u^*vv^*w \xrightarrow{Id_{u^*}m_{v}Id_{w}}u^*w,
\]
where $m_v\colon vv^*\to \Ver_{n}$, with $\Ver_{n}$ being the web of $n$ 
parallel oriented vertical line segments, is defined by the following 
inductive algorithm.
\begin{enumerate}
\item Express $v$ using the growth algorithm, label each level of the 
growth algorithm starting from zero. Then form $vv^*$.
\item At the $k$th level in the growth algorithm, \textit{resolve} 
the corresponding pair 
of arc, H or Y-rules in $v$ and $v^*$ by applying the foams.
\begin{align}
 \xy(0,0)*{\label{multrules}\includegraphics[scale=0.7]{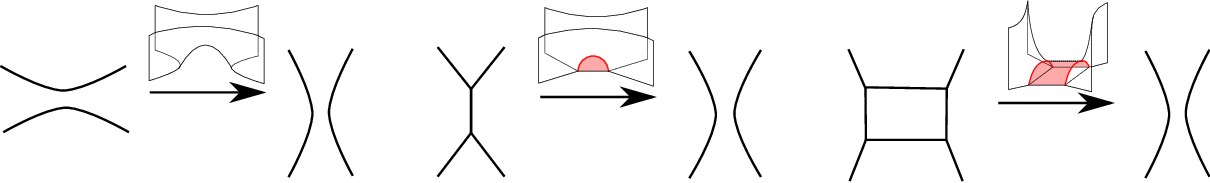}};\endxy
\end{align}
\end{enumerate}
Note that at the last level in the growth algorithm of $v$, only pairs of arcs 
are present.
\end{defn}

\begin{ex}
Let $w$ and $v$ be the following webs. 
\begin{align}
   \xy(0,0)*{\includegraphics[width=180px]{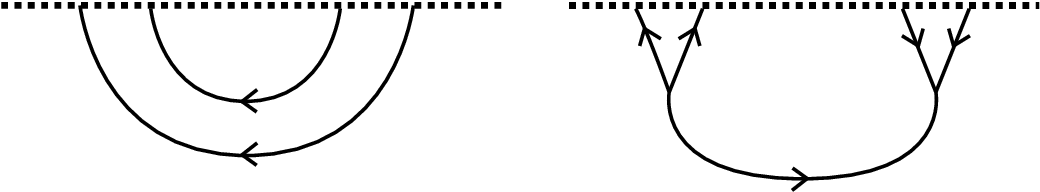}};(-5,0)*{w};
 (28,0)*{v};\endxy
\end{align}
The multiplication foam $m_{w,v,v}$ is given by the following steps.
\begin{align}
   \xy(0,0)*{\includegraphics[width=300px]{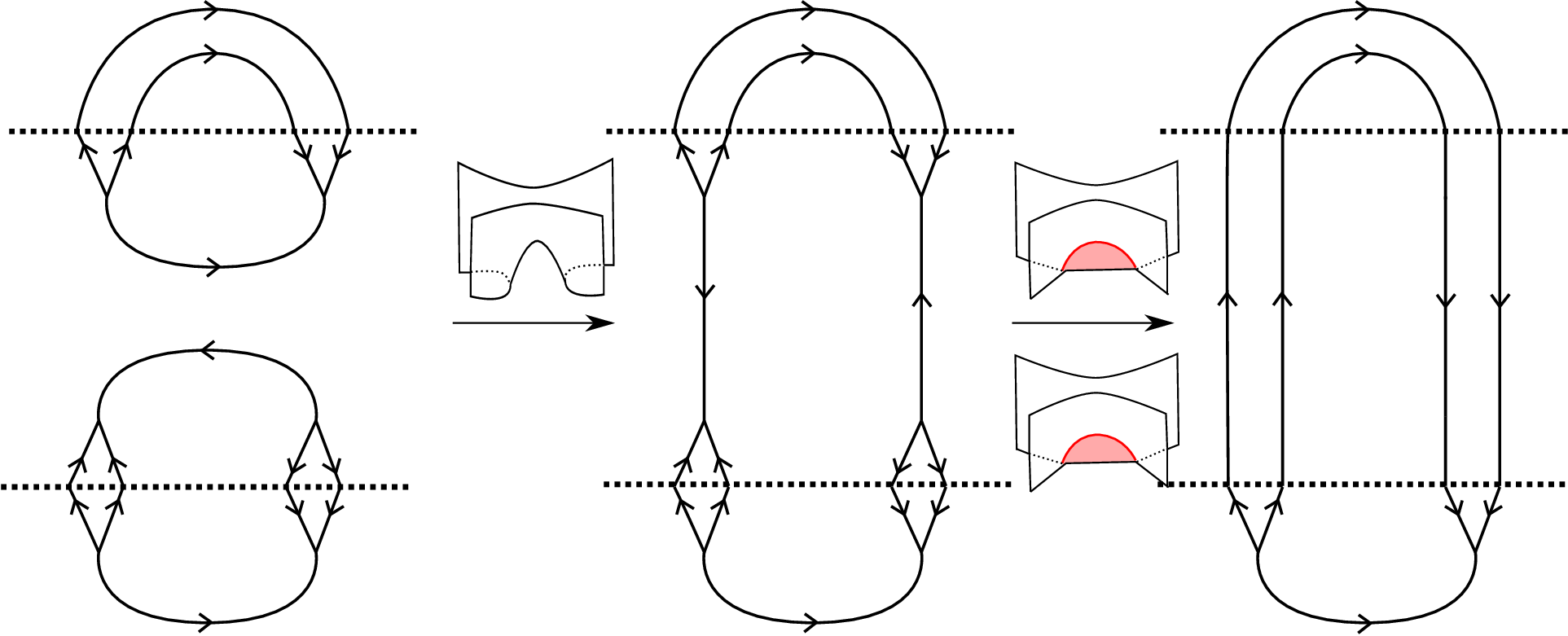}};(51,17)*{w^*};
 (-28,-17)*{v};(-26.3,17)*{w^*};
 (-27.2,-7)*{v^*};(-27,8)*{v};
 (50,-17)*{v};\endxy
\end{align}
\end{ex}
\begin{prop}
The foam $m_v$ in Definition~\ref{multfoam} only depends on the isotopy type of $v$.
\end{prop}

\begin{proof} We have to show that $m_v$ is 
independent of the way $v$ is expressed using the growth algorithm 
(Definition~\ref{growth}). Let $G_{1}$ and $G_{2}$ be two different expressions of 
$v$ using the growth algorithm. We have to compare $G_1$ and $G_2$ 
walking backwards in the growth algorithm. Note that we only have to worry 
about two consecutive steps in the same region of $v$. Reordering steps in 
``distant'' regions of $v$ corresponds to an isotopy which simply alters 
the height function on $m_v$. With these observations, the only possible 
remaining difference between the last two steps in $G_1$ and $G_2$ is 
the following.
\begin{align}\label{Yreplace}
	\xy(0,0)*{\includegraphics[width=90px]{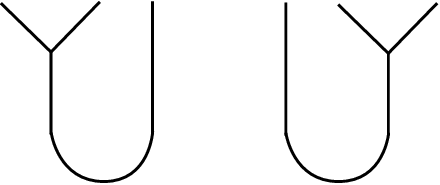}};\endxy
\end{align}
If the last two steps in $G_1$ and $G_2$ are equal, we have to go further back 
in the growth algorithm. Besides two-step differences of the same 
sort as above, we can encounter another one of the following sort.
\begin{align}\label{Hreplace}
   \xy(0,0)*{\includegraphics[width=90px]{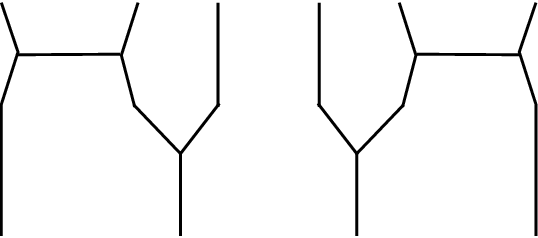}};\endxy
\end{align}
We have to check that the above two-step differences in $G_1$ and $G_2$ 
correspond to equivalent foams. In the first case, the foams in the 
multiplication algorithm are given by 
\begin{figure}[H]
   \centering
     \includegraphics[width=220px]{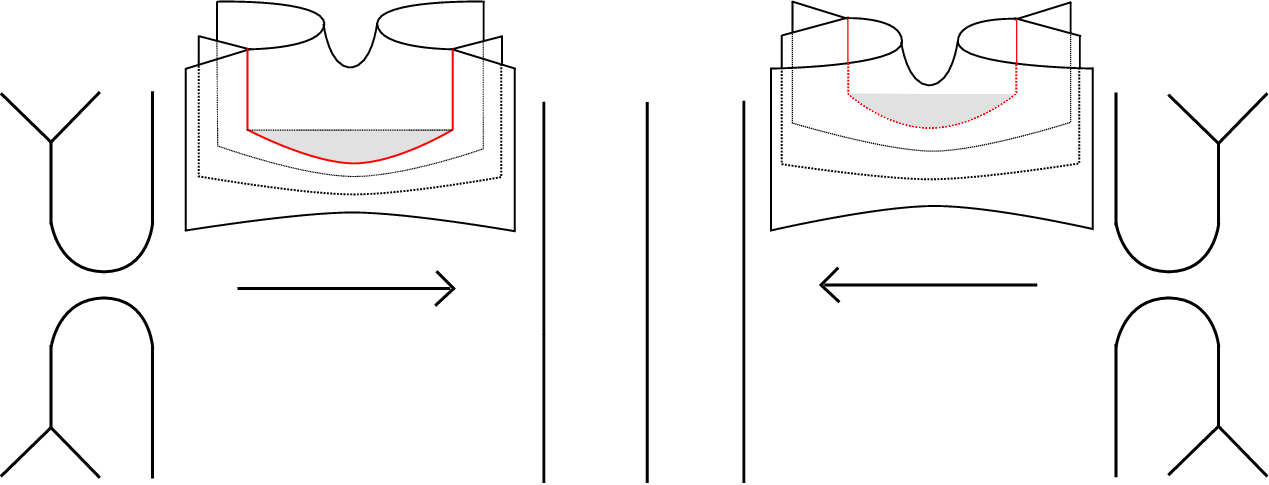}
     \caption{A possible local difference between $m_{G_{1}}$ and $m_{G_{2}}$.}
     \label{Yfoam}
\end{figure}

In the second case, we get 
\begin{figure}[H] 
   \centering
    \includegraphics[width=250px]{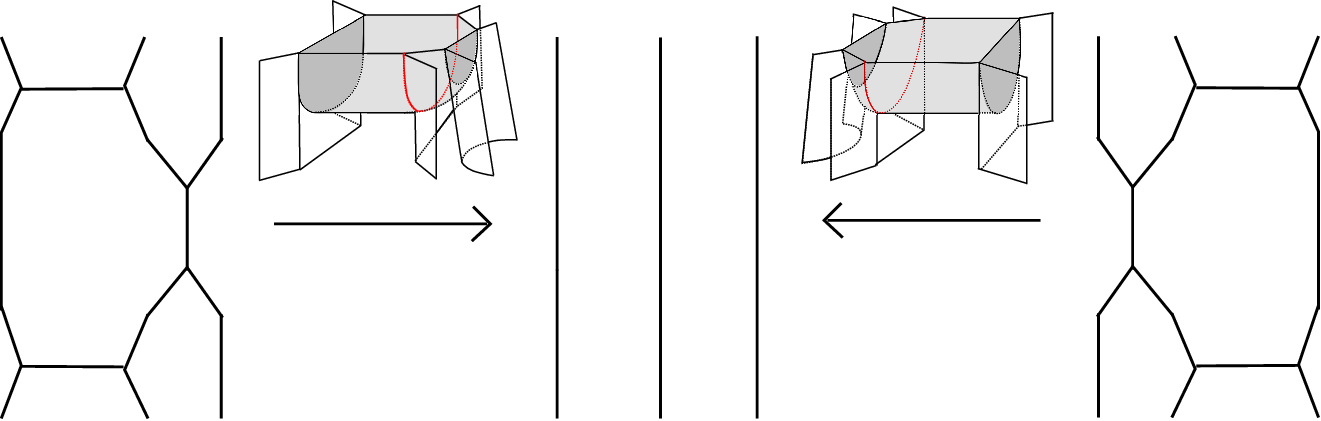}.
    \caption{The other possible local difference between $m_{G_{1}}$ and $m_{G_{2}}$.}
    \label{Hfoam}
\end{figure}

The two foams in Figure~\ref{Yfoam} are isotopic - one foam can be produced 
from the other by sliding the red singular arc over the saddle as illustrated below.
\begin{align}
   \xy(0,0)*{\includegraphics[width=170px]{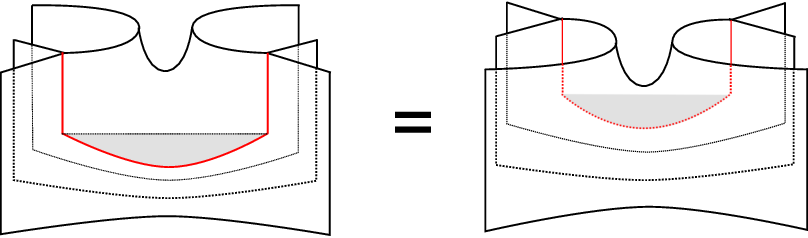}};\endxy
\end{align}
The two foams in Figure~\ref{Hfoam} are also isotopic - one foam can be 
produced from the other by moving the red singular arc to the right or to the 
left as illustrated below.
\begin{align}
   \xy(0,0)*{\includegraphics[width=180px]{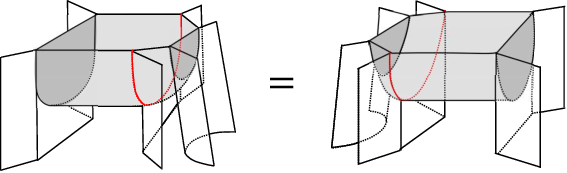}};\endxy 
\end{align}
The cases above are the only possible ones, so their verification provides 
the proof. 
\end{proof}

\begin{prop} \label{prop:multqgrade}
The foam $m_v$ has $q$-grading $\ell(S)$.
\end{prop}
\begin{proof}
We proceed by backward induction on the level of the growth algorithm 
expressing $v$. At the final level of the growth algorithm, 
the only possible rule is the arc rule. Resolving a corresponding pair of 
arcs in $v$ and $v^*$ results in two new vertical strands and is obtained 
by a saddle point cobordism, which has $q$-grading 2. 

Let $n_{k}$ be the number of vertical strands and $m_v^{k}$ be the foam after 
resolving the last $k$ rules in the growth algorithm of $v$. Suppose that 
$n_k$ is equal to the $q$-degree of $m_v^k$. In the next 
step of the multiplication we can have three cases.
\begin{enumerate}
\item The resolution of a pair of arc-rules. In this case we have 
$n_{k+1}=n_k+2$ and $m_v^{k+1}$ is obtained from $m_v^k$ by adding a saddle, 
which adds 2 to the $q$-grading.
\item The resolution of a pair of Y-rules. In this case we have 
$n_{k+1}=n_k+1$ and $m_v^{k+1}$ is obtained from $m_v^k$ by adding an unzip, 
which adds 1 to the $q$-grading.
\item The resolution of a pair of H-rules. In this case we have 
$n_{k+1}=n_k$ and $m_v^{k+1}$ is obtained from $m_v^k$ by adding a square foam, 
which adds $0$ to the $q$-grading. 
\end{enumerate}
\end{proof}

There is a useful alternative definition of $\mathcal{W}^c_S$, 
which we give below. 
As a service to the reader, we state it as a lemma and prove that it really is 
equivalent to our definition above. Both definitions have their advantages 
and disadvantages, so it is worthwhile to catalogue both in this paper. 

\begin{lem}
\label{lem:webalgaltern}
For any $c\in\mathbb{C}$ and any $u,v\in B^S$, we have a 
grading preserving isomorphism 
\[
\foamt^c(u,v)\cong{}_{u}\mathcal{W}^c_{v}.
\] 

Using this isomorphism, the multiplication  
\[
{}_{u}{\mathcal W}^c_{v}\otimes {}_{v'}{\mathcal W}^c_{w}\to 
{}_{u}{\mathcal W}^c_{w}
\]
corresponds to the composition 
\[
\foamt^c(u,v)\otimes\foamt^c(v',w)\to\foamt^c(u,w),
\]
if $v=v'$, and is zero otherwise. 
\end{lem}
\begin{proof}
The isomorphism of the first claim is sketched in the following figure.
\begin{align*}
	\xy(0,0)*{\includegraphics[width=180px]{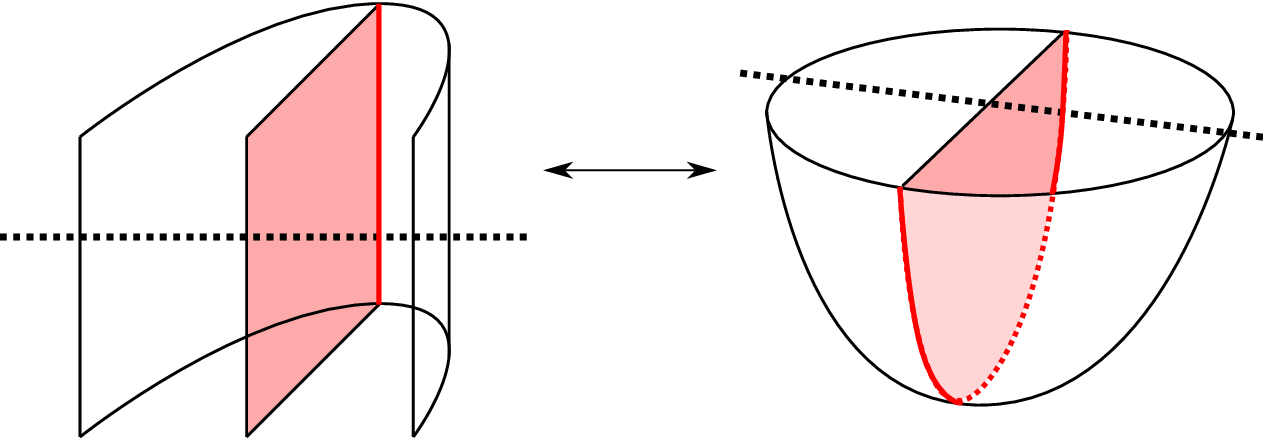}};(-5.5,5)*{v^*};
 (28.5,10.5)*{v^*};(-6,-6)*{v};
 (23,-1)*{v};\endxy
\end{align*}
The proof of the second claim follows from analysing what the 
isomorphism does to the resolution of a pair of arc, Y or H-rules in the 
multiplication foam. This is done below. 
\begin{align*}
	\xy(0,0)*{\includegraphics[width=180px]{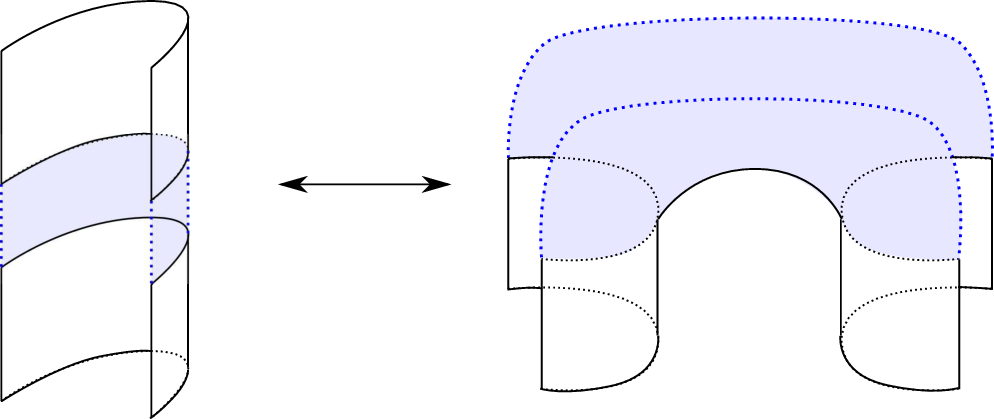}};(-17,8)*{f_2};
 (-1,-7.5)*{f_1};(-17,-7)*{f_1};
 (33,-7.5)*{f_2};\endxy\\
	\xy(0,0)*{\includegraphics[width=180px]{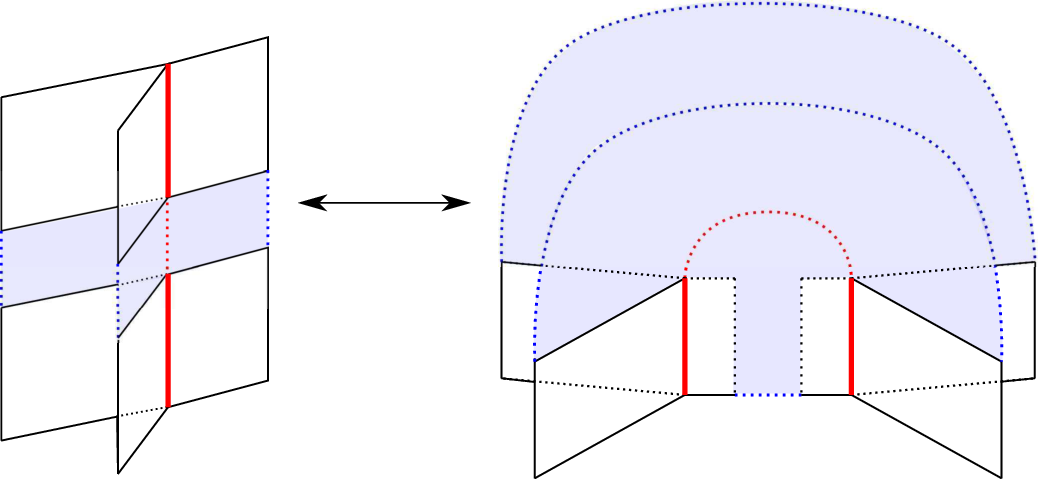}};(-13,9)*{f_2};
 (-4,-9.5)*{f_1};(-13,-5.5)*{f_1};
 (33,-9.5)*{f_2};\endxy\\
	\xy(0,0)*{\includegraphics[width=180px]{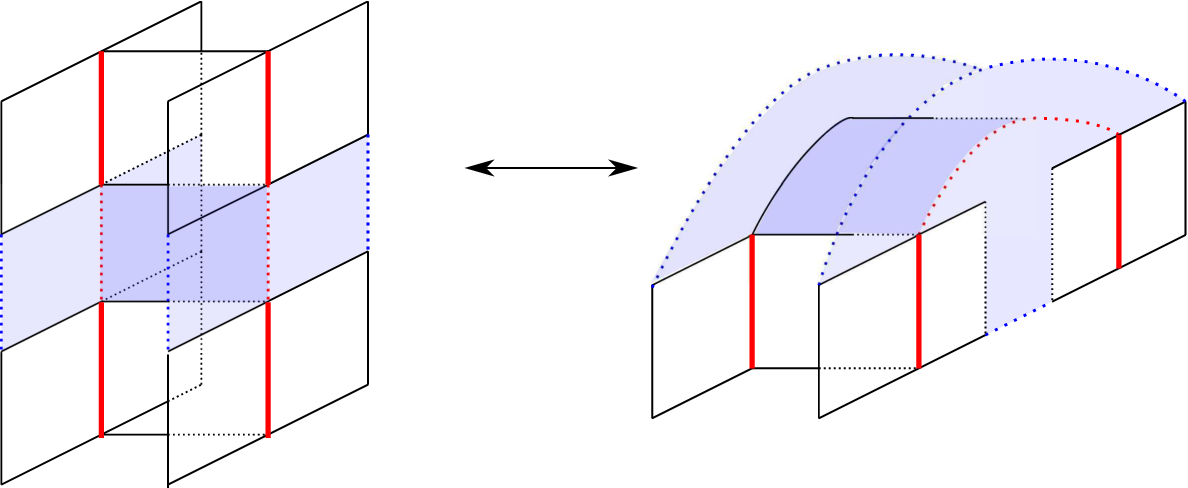}};(-9,10)*{f_2};
 (20,-8.6)*{f_1};(-9,-4.5)*{f_1};
 (30,-3)*{f_2};\endxy
\end{align*}
\end{proof}

Note that Lemma~\ref{lem:webalgaltern} implies that ${\mathcal W}^S_c$ is 
associative and unital, something that is not immediately clear from 
Definition~\ref{defn:webalg}. 
For any $u\in B^S$, the identity $1_u\in\foamt^c(u,u)$ defines an idempotent. 
We have 
\[
1=\sum_{u\in B^S} 1_u\in {\mathcal W}^S_c.
\]

Alternatively, one can see ${\mathcal W}^S_c$ as a category whose 
objects are the elements in $B^S$ such that the module of morphisms 
between $u\in B^S$ and $v\in B^S$ is given by $\foamt^c(u,v)$. 
In this paper we will mostly see 
${\mathcal W}^S_c$ as an algebra, but will sometimes refer to the category 
point of view. 
\vskip0.5cm
In this paper, we will study ${\mathcal W}^S_c$ for two special 
values of $c\in\mathbb{C}$. 
\begin{defn}
Let $K^S$ and $G^S$ be the complex algebras obtained from ${\mathcal W}^S_c$ 
by setting $c=0$ and $c=1$, respectively. We call them 
\textit{Khovanov's web algebra} and \textit{Gornik's web algebra}, respectively, 
to distinguish them throughout the paper.  
\end{defn}
Note that $G^S$ is a filtered algebra. Its associated 
graded algebra is $K^S$. By Lemma~\ref{lem:webalgaltern}, 
both $K^S$ and $G^S$ are finite dimensional, unital, associative algebras. 
They also have similar decompositions as shown below.
\[
K^S=\bigoplus_{u,v\in B^S}{}_uK_v\,,\quad\quad G^S=\bigoplus_{u,v\in B^S}{}_uG_v.
\] 
We now recall the definition of complex, 
graded and filtered Frobenius algebras. 
Let $A$ be a finite dimensional, graded, complex algebra and let  
$\mathrm{Hom}_{\mathbb{C}}(A,\mathbb{C})$ be the complex vector space of 
grading preserving 
maps. The \textit{dual} of $A$ is defined by      
\[
A^{\vee}=\bigoplus_{n\in\mathbb{Z}}\mathrm{Hom}_{\bC}(A,\mathbb{C}\{n\}),
\] 
where $\{n\}$ denotes an upward degree shift of size $n$.
Note that $A^{\vee}$ is also a graded module, such that 
\begin{equation}
\label{eq:dualgrading}
(A^{\vee})_i=(A_{-i})^{\vee},
\end{equation}
for any $i\in \mathbb{Z}$.
Then $A$ is called a \textit{graded, 
symmetric Frobenius algebra of Gorenstein parameter} $\ell$, 
if there exists an isomorphism of graded $(A,A)$-bimodules 
\[
A^{\vee}\cong A\{-\ell\}.
\] 
If $A$ is a complex, finite dimensional, filtered algebra, 
let $\mathrm{Hom}_{\mathbb{C}}(A,\mathbb{C})$ be the complex vector space of 
filtration preserving 
maps. The \textit{dual} of $A$ is defined by     
\[
A^{\vee}=\bigoplus_{n\in\mathbb{Z}}\mathrm{Hom}_{\bC}(A,\mathbb{C}\{n\}),
\]
where $\{n\}$ denotes an upward suspension of size $n$.
Note that $A^{\vee}$ is also a filtered module, such that 
\begin{equation}
\label{eq:dualfiltration}
(A^{\vee})_i=(A_{-i})^{\vee},
\end{equation}
for any $i\in \mathbb{Z}$.
Then $A$ is called a \textit{filtered, 
symmetric Frobenius algebra of Gorenstein parameter} $\ell$, 
if there exists an isomorphism of filtered $(A,A)$-bimodules 
\[
A^{\vee}\cong A\{-\ell\}.
\]
For more information on graded Frobenius algebras, see~\cite{ue} and 
the references therein, for example. We do not have a good reference for 
filtered Frobenius algebras, but it is a straightforward generalization of 
the graded case. We will explain some basic results on the character theory of 
filtered and graded symmetric Frobenius algebras in 
Section~\ref{sec:grothendieck}. 
   
\begin{thm}\label{thm:frob} For any sign string $S$, 
the algebra $K^S$ is a graded symmetric Frobenius algebra 
and $G^S$ is a filtered symmetric Frobenius algebra, both of Gorenstein 
parameter $2\ell(S)$. 
\end{thm}
\begin{proof} First, let $c=0$. 
We take, by definition, the trace form 
\[
\mathrm{tr}\colon K^S\to \mathbb{C}
\] 
to be zero on ${}_uK_v$, when $u\ne v\in B^S$. 
For any $v\in B^S$, we define 
\[
\mathrm{tr}\colon {}_vK_v\to\mathbb{C}
\] 
by closing any foam $f_v$ with $1_v$, e.g.as pictured below.
\begin{align*}
	\xy(0,0)*{\includegraphics[width=60px]{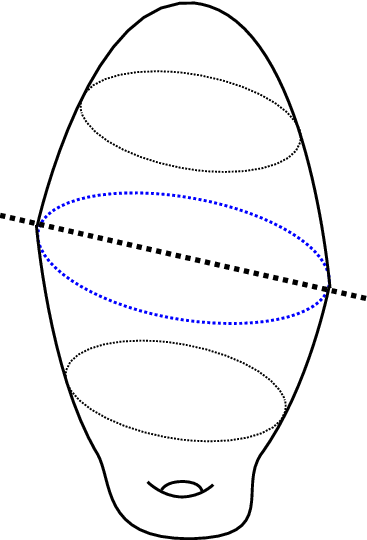}};(-9,10)*{1_v};
 (-9,-8)*{f_v};(10.5,3)*{v^*};
 (9.5,-3.5)*{v};\endxy
\end{align*}  
Equivalently, in $\foamt^0(v,v)$, closing $f_v$ by $1_v$,
\begin{align*}
	\xy(0,0)*{\includegraphics[width=180px]{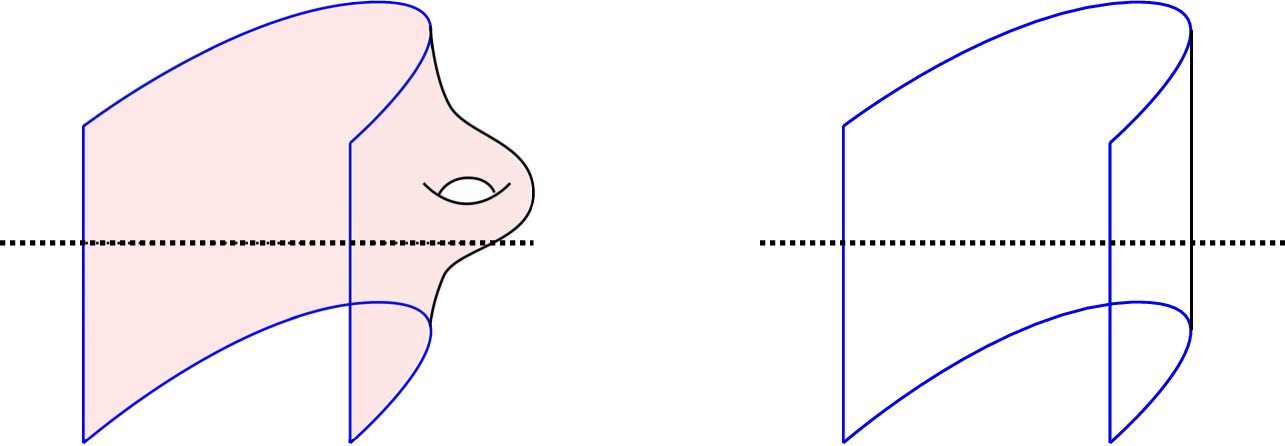}};(-30,2)*{f_v};(-4.5,5)*{v^*};
 (-5,-5.5)*{v};(8,2)*{1_v};(30,5)*{v^*};
 (29.5,-5.5)*{v};\endxy
\end{align*}  
gives
\begin{align*}
	\xy(0,0)*{\includegraphics[width=95px]{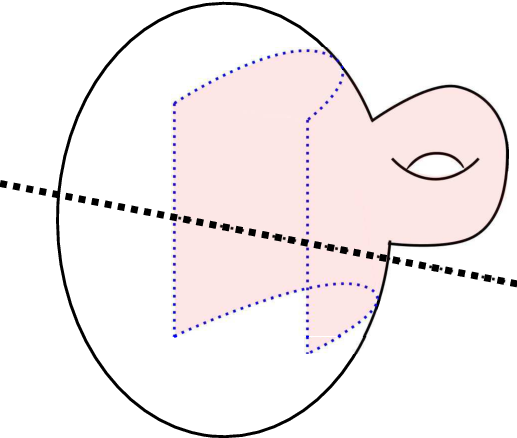}};(-15,4.5)*{1_v};
 (19,4.5)*{f_v};\endxy
\end{align*}
The fact that the trace form is non-degenerate follows immediately 
from the closure relation in Subsection~\ref{subsec:foams}.

The fact that $\mathrm{tr}(gf)=\mathrm{tr}(fg)$ holds follows from sliding 
$f$ around the closure until it appears on the other side of $g$, e.g. as shown below.
\begin{align*}
\xy(0,1)*{\includegraphics[width=90px]{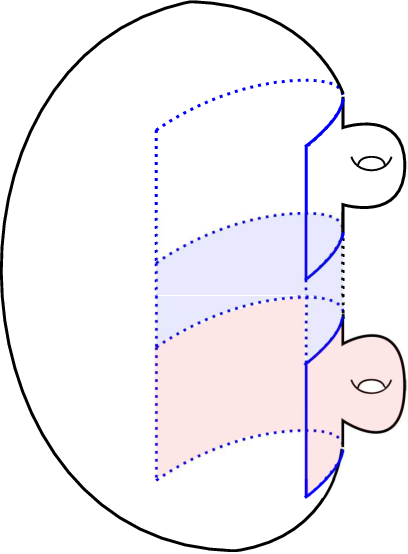}};(-18.5,1)*{1_u};
 (18.5,1)*{1_v};(17.5,-8)*{f};
 (17.5,9)*{g};\endxy\quad=\quad\xy(0,0)*{\includegraphics[width=88px]{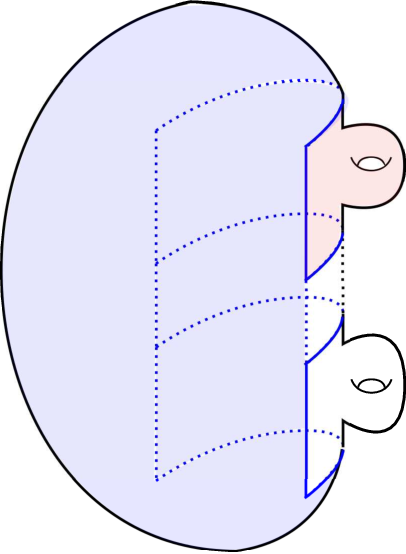}};(-18,0)*{1_v};
 (18.5,0)*{1_u};(17.5,-9)*{g};
 (17.5,8)*{f};\endxy
\end{align*}
Note that a closed foam can only have non-zero evaluation if it has 
degree zero. Therefore, for any $u\in B^S$ and any 
two homogeneous elements $f\in \F^0(u^*v)$ and $g\in \F^0(v^*u)$, we have 
$\mathrm{tr}(fg)\ne 0$ unless $\deg(f)=-\deg(g)$. By the shift $\ell(S)$ in  
\[
{}_uK_v=\mathcal{F}^0(u^*v)\{\ell(S)\}
\]
and by~\eqref{eq:dualgrading}, 
this implies that 
the non-degenerate trace form on $K^S$ gives rise to a graded 
$(K^S,K^S)$-bimodule isomorphism   
\begin{equation}
\label{eq:frobK}
(K^S)^{\vee}\cong K^S\{-2\ell(S)\}.
\end{equation}

Now, let $c=1$. Then the construction above also gives a 
non-degenerate bilinear form on $G^S$. Moreover, it induces 
a filtration preserving bijective $\mathbb{C}$-linear map 
of filtered $(G^S,G^S)$-bimodules 
\begin{equation}
\label{eq:frobG}
G_s\{-2n\}\to G_s^{\vee}.
\end{equation}
The associated graded map is precisely the isomorphism in 
\eqref{eq:frobK}. By Proposition~\ref{prop:Srid}, this implies that 
the map in~\eqref{eq:frobG} is a strict isomorphism 
of filtered $(G^S,G^S)$-bimodules.   
\end{proof}

We now explain some of Gornik's results, which are relevant for $G^S$. 
Recall that $R^1_{u^*v}$ is the commutative ring associated to $u^*v$, 
generated by the edge variables of $u^*v$ and mod out by the ideal, which,  
for each trivalent vertex in $u^*v$, is generated by the relations 
\begin{equation}
\label{eqn:R1}
x_1+x_2+x_3=0,\qquad x_1x_2+x_1x_3+x_2x_3=0,\qquad x_1x_2x_3=1,
\end{equation}
where $x_1,x_2$ and $x_3$ are the edge variables around the vertex. 
The algebra $R^1_{u^*v}$ acts on ${}_uG_v$ in such a way 
that each edge variable corresponds to adding a dot on the incident facet.
See~\cite{g},~\cite{kv} and~\cite{mv} for the precise definition and more details. 

In what follows, 3-colorings will always be assumed to be admissible and 
we therefore omit the adjective. Theorem 3 in~\cite{g} proves the following.
\begin{thm}\label{thm:Gornik}(\textbf{Gornik})
There is a complete set of orthogonal idempotents $e_T\in R^1_{u^*v}$, indexed 
by the 3-colorings $T$ of $u^*v$. The number of 3-colorings of $u^*v$ is 
exactly equal to $\dim_q({}_uG_v)$. 

These idempotents are not filtration preserving, 
but as an $R^1_{u^*v}$-module (i.e. forgetting the filtration on ${}_uG_v$ 
and its left ${}_uG_u$ and right ${}_vG_v$-module structures) 
we have  
\[
{}_uG_v\cong \bigoplus_T \mathbb{C}e_T.
\]
\end{thm}

Let us have a closer look at Gorniks idempotents.  
First of all, in the proof of Theorem 3 in~\cite{g} Gornik notes 
that for any edge $i$ and any 3-coloring $T$ of $u^*v$, we have 
\begin{equation}
\label{eqn:edgeaction}
x_ie_T=\zeta^{T_i}e_T\in R^1_{u^*v},
\end{equation}
where $\zeta$ is a primitive third root of unity, $x_i$ is the edge variable 
and $T_i$ the color of the edge 
(see (4) in~\cite{mv} for this result in the context of foams). 

Furthermore, a 3-coloring of $u^*v$ 
actually corresponds to a pair of 3-colorings of $u$ and $v^*$ 
that match at the boundary. Of course, there is a bijective correspondence 
between 3-colorings of $u$ and $v^*$, so we see that a 3-coloring of 
$u^*v$ corresponds to a matching pair of 3-colorings of $u$ and $v$.

Recall that ${}_uG_v$ is a left ${}_uG_u$-module and a right ${}_vG_v$-module.  
Let $T_1$ and $T_2$ be a pair of matching 3-colorings of $u$ and $v$, 
respectively, which together give a 3-coloring $T$ of $u^*v$. 
Then the action of $e_T$ on any $f\colon u\to v$ can be written as 
\[
e_{T_1}fe_{T_2}.
\]
To show that this notation really makes sense, define 
\textit{Gornik's symmetric idempotent} associated to $T_1$ as
\[
e_{u,T_1}=e_{T_1}1_ue_{T_1}.
\]
So we let the Gornik idempotent associated to the symmetric 3-coloring 
of $u^*u$, given by $T_1$ both on $u$ and $u^*$, act on $1_u$. Then 
we have 
\[
e_{T_1}fe_{T_2}=e_{u,T_1}fe_{v,T_2},
\]
where on the right-hand side we really mean composition. 

We immediately see that 
\[
e_{T_1}1_ue_{T_2}=0\Leftrightarrow T_1\ne T_2
\]
and 
\[
e_{T_1}1_ue_{T_1}e_{T_2}1_ue_{T_2}=\delta_{T_1,T_2} e_{T_1}1_ue_{T_1}\quad\text{and}\quad 
\sum_{T}e_{T}1_ue_{T}=1_u,
\]
where the sum is over all 3-colorings of $u$. 
This shows that the $e_{u,T}$, for all 3-colorings 
$T$ of a given $u\in B^S$, are orthogonal idempotents in ${}_uG_u$.  
It also implies that 
\[
e_{T_1}1_ue_{T_1}=e_{T_1}1_u=1_ue_{T_1},
\]
so it is enough to label just the source or just the target of $1_u$.
For this purpose, we define $R^1_u$ to be ``half'' of $R^1_{u^*u}$, i.e. the 
subring which is only generated by the edge variables of $u$. To be precise, 
we have 
\[
R^1_{u^*u}\cong R^1_u\otimes_{S}R^1_u,
\]
where $\otimes_{S}$ indicates that we impose the relation 
$x\otimes 1=1\otimes x$, for any $x$ corresponding to a boundary edge of $u$. 

If $u$ has no closed cycles, then all the 3-colorings of 
$u^*u$ are symmetric, because they are completely determined by the colors on 
the boundary of $u$. In that case 
\[
e_T\mapsto e_{u,T}
\]
defines an isomorphism of algebras $R^1_u\cong {}_uG_u$. In particular, 
${}_uG_u$ is commutative. This is not true in general, but we can prove 
the following.
\begin{lem}
\label{lem:embedding} 
For any $u\in B^S$, the map 
\[
x\mapsto x1_u
\]
defines a strict embedding of filtered $R^1_u$-modules
\[
\iota\colon R^1_u\to {}_uG_u.
\]
In particular, we see that $(R^1_u)_0\cong\mathrm{im}(\iota)_0\cong\mathbb{C}1_u$.
\end{lem}
\begin{proof}
The map is clearly a homomorphism of filtered algebras. 

The relations~\eqref{eq:dotm} correspond precisely to 
the relations in $R^1_u$, because the only singular edges in $1_u$ are the 
ones corresponding to the trivalent vertices of $u$. 
This shows that it is a strict embedding. 
\end{proof}

For any $u\in B^S$, we define the graded ring 
\[
R^0_u=E(R^1_u).
\]
This ring is the one which appears 
in Khovanov's original paper~\cite{kv}. In $R^0_u$ we have the relations
\begin{equation}
\label{eqn:R0}
x_1+x_2+x_3=0,\qquad x_1x_2+x_1x_3+x_2x_3=0,\qquad x_1x_2x_3=0.
\end{equation}
The reader should compare them to~\eqref{eqn:R1}.

There are no analogues of the Gornik idempotents in $R^0_u$, but we do 
have an analogue of Lemma~\ref{lem:embedding}.

\begin{lem}
\label{lem:gradedembedding} 
For any $u\in B^S$, the map 
\[
x\mapsto x1_u
\]
defines an embedding of graded $R^0_u$-modules
\[
E(\iota)\colon R^0_u\to {}_uK_u.
\]
In particular, we see that $(R^0_u)_0\cong\mathrm{im}(E(\iota))_0\cong\mathbb{C}1_u$.
\end{lem}
 
Another interesting consequence of Theorem~\ref{thm:Gornik} is the following. 
\begin{prop}
\label{prop:Gsemisimple}
As a complex algebra, i.e. without taking the filtration into account, 
$G^S$ is semisimple.
\end{prop}
\begin{proof}
For any $u\in B^S$ and any 3-coloring $T$ of $u$, define the projective 
$G^S$-module 
\[
P_{u,T}=(G^S)e_{u,T},
\]
where $e_{u,T}$ is Gornik's symmetric idempotent in $G^S$ defined above. 
Theorem~\ref{thm:Gornik} and our subsequent analysis of Gornik's idempotents 
show that the $P_{u,T}$ form a complete set of indecomposable, projective 
$G^S$-modules. Furthermore, we have 
\[
\text{Hom}_{G^S}(P_{u,T},P_{v,T'})\cong e_{u,T}(G^S)e_{v,T'}\cong 
\begin{cases}
\mathbb{C},&\quad\text{if}\; T\;\text{and}\;T'\;\text{match at}\; S,\\
\{0\},&\quad\text{else}.
\end{cases}
\]  
This shows that $P_{u,T}\cong P_{v,T'}$ if and only if 
$T$ and $T'$ match at the common boundary. It also shows that 
if $P_{u,T}\not\cong P_{v,T'}$, then 
\[
\text{Hom}_{G^S}(P_{u,T},P_{v,T'})=\text{Hom}_{G^S}(P_{v,T'},P_{u,T})=\{0\}.
\]
Finally, it shows that each $P_{u,T}$ has only one composition factor, i.e. 
$P_{u,T}$ is irreducible. 

It is well-known that this implies that $G^S$ is semisimple, e.g. see 
Proposition 1.8.5 in~\cite{be}.   
\end{proof}

By Proposition~\ref{prop:Gsemisimple}, it is clear that for each 
$u\in B^S$ and each coloring $T$ of $u$, the corresponding block 
in $G^S$ is isomorphic to $\mathrm{End}(P_{u,T})$. 
In Section~\ref{sec:center}, we will determine the central idempotents 
of $G^S$. 

\section{The center of the web algebra and 
the cohomology ring of the Spaltenstein variety}
\label{sec:center}
For the rest of this section, choose arbitrary but fixed non-negative integers 
$n\geq 2$ and $k\leq n$, such that $d=3k\geq n$. Let 
\[
\Lambda(n,d)=\left\{\mu\in \mathbb{N}^n\mid \sum_{i=1}^n\mu_i=d\right\}
\]
be the set of \textit{compositions} of $d$ of length $n$. 
By $\Lambda^+(n,d)\subset\Lambda(n,d)$ we denote the 
subset of \textit{partitions}, i.e. all $\mu\in\Lambda(n,d)$ such that 
\[
\mu_1\geq\mu_2\geq \ldots\geq \mu_n\geq 0.
\] 
Also for the rest of this section, choose an arbitrary but fixed sign string 
$S$ of length $n$. We associate to $S$ a unique element 
$\mu=\mu_S\in\Lambda(n,d)$, such that    
\[
\mu_i=
\begin{cases}
1,&\quad \mathrm{if}\quad s_i=+,\\
2,&\quad \mathrm{if}\quad s_i=-. 
\end{cases}
\] 
Let $\Lambda(n,d)_{1,2}\subset \Lambda(n,d)$ be the subset of compositions 
whose entries are all $1$ or $2$. For any sign string $S$, we have 
$\mu_S\in \Lambda(n,d)_{1,2}$. 

Let $\lambda=(3^k)\in \Lambda(n,d)$. Let $\mathrm{Col}_{\mu}^{\lambda}$ be the 
set of column strict tableaux of shape $\lambda$ and type $\mu$, both 
of length $n$. It is well-known that there is a bijection between 
$\mathrm{Col}_{\mu}^{\lambda}$ and the tensor basis of 
\[
V^{\mu}=V^{\mu_1}\otimes \cdots\otimes V^{\mu_n},
\]
where $V^1=V^+$ and $V^2=V^1\wedge V^1\cong V^-$ (see Section 3 in~\cite{ms}, 
for example). However, we are interested in tensors as summands in 
the decomposition of elements in $B^S$. Therefore, we prove 
Proposition~\ref{prop:tableauxflows} in Subsection~\ref{subsec:tableauxflows}. 
The reader, who is not interested in the details of the proof of this 
proposition, can choose to skip this subsection at a first reading and 
just read the statement of the proposition.  


\subsection{Tableaux and flows}
\label{subsec:tableauxflows}
Let $p_S$ be the number of positive entries and $n_S$ the number of negative 
entries of $S$. By definition, we have that $d=p_S+2n_S$. The key idea 
in this subsection is to reduce all proofs to the case where $n_S=0$. 

\begin{defn} \label{def:posstr}
Fix any state string $J$ of length $n$, we define a new state string 
$\hat{J}$ of length $d$ by the following algorithm.
\begin{enumerate}
\item Let ${}_{0}\hat{J}$ be the empty string.
\item For $1\leq i\leq n$, let $_{i}\hat{J}$ be the result of 
concatenating $j_{i}$ to $_{i-1}\hat{J}$ if $\mu_{i}=1$. 
If $\mu_{i}=2$ then
\begin{enumerate}
\item concatenate $(1, 0)$ to $_{i-1}\hat{J}$ if $j_{i} = 1$.
\item concatenate $(0, -1)$ to $_{i-1}\hat{J}$ if $j_{i} = -1$.
\item concatenate $(1, -1)$ to $_{i-1}\hat{J}$ if $j_{i} = 0$.
\end{enumerate}
\end{enumerate}
We set $\hat{J} = {}_{n}\hat{J}$. Lastly, for any $c\in \{-1,0,1\}$, 
we define $\hat{J}^{c}$ to be the number of entries in $\hat{J}$ that is 
equal to $c$.
\end{defn}

\begin{prop} 
\label{prop:tableauxflows}
There is a bijection between $\mathrm{Col}_{\mu}^{\lambda}$ and the 
set of state strings $J$ such that there exists a $w\in B^S$ and 
a flow $f$ on $w$ which extends $J$. 
\end{prop}
\noindent The proof of Proposition~\ref{prop:tableauxflows} follows 
directly from 
Lemmas~\ref{lem:ttostr} and~\ref{lem:strtoflow}.

\begin{lem} \label{lem:ttostr}
There is a bijection between $\mathrm{Col}_{\mu}^{\lambda}$ and 
state strings $J$ of length $n$ such that 
\begin{equation}\label{eqn:conds}
\hat{J}^{-1} = \hat{J}^{0} = \hat{J}^{1}.
\end{equation}
where the $\hat{J}^{c}$ are as defined in Definition~\ref{def:posstr}.
\end{lem}

\begin{proof}
Given a state string $J$ satisfying (\ref{eqn:conds}), we first give an 
algorithm to build a 3-column tableau $Y_J$, filled 
with integers from 1 to $n$. Afterwards, we show that $Y_J$ has 
shape $\lambda$. 

Begin by labeling the three columns with $1, 0$ and $-1$, 
reading from left to right. We are going to build up $Y_J$ from top to bottom. 
Start by taking $Y_J$ to be the empty tableau. Then, from $i=1$ to $i=n$, 
do the following:
\begin{enumerate}
\item If $\mu_{i}=1$, add one box labeled $i$ to column $j_i$ in $Y_J$.  
\item If $\mu_{i} = 2$, add two boxes labeled $i$ to columns $c_{1}$ and 
$c_{2}$, such that $c_1\ne c_2$ and $c_1+c_2=j_i$.
\end{enumerate}

We have to show that $Y_J$ belongs to $\mathrm{Col}_{\mu}^{\lambda}$. 
Since the algorithm builds up from top to bottom, $Y_J$ is strictly column 
increasing. To see that $Y_J$ has shape $\lambda$, we need to show 
that every row in $Y_J$ has three entries. Observe that the number of 
filled boxes in column $c$ of $Y_J$ is exactly equal to $\hat{J}^{c}$. 
Since we have assumed condition~\eqref{eqn:conds}, all three columns 
have the same length, therefore every row in $Y_J$ must have 
exactly three entries.


Conversely, let $T\in\mathrm{Col}_{\mu}^{\lambda}$. We define a state string 
$J$ as follows.
\[
j_{i}=\sum_{i\mathrm{\,appears\,in\,column\,c}} c.
\]
Since $\mu$ corresponds to a sign string and $T$ is column strict, 
we see that, for each $1\leq i\leq n$, $i$ 
can appear at most twice in $T$ but never twice in the same 
column. Thus, $j_{i}\in\{-1,0,1\}$, i.e. $J$ is a state string. It follows from 
the definition of $\hat{J}$ that $\hat{J}^{c}$ is equal to the length of 
column $c$ of $T$. Since $T$ is of shape $\lambda$, the number of boxes in 
each column is the same. Hence, condition (\ref{eqn:conds}) holds for $J$. 

It is straightforward to check that the above two constructions are inverse 
to each other and therefore determine a bijection. 
\end{proof}

\begin{lem} \label{lem:strtoflow}
A state string $J$ corresponds to the boundary state of a flow on a 
web $w\in B^S$ if and only if condition~\eqref{eqn:conds} holds for $J$.
\end{lem}

\begin{proof}
Let $w\in B^S$ be equipped with a flow with boundary state string 
$J$. We are going to show that $J$ satisfies condition~\eqref{eqn:conds} 
by induction on $n$. For $n=2$, $w$ can only be an arc. 
In this case it is simple to check that all flows on $w$ have corresponding 
boundary state strings satisfying condition~\eqref{eqn:conds}. 

For $n>2$, we express $w$ using the growth algorithm in an arbitrary, but 
fixed way, with the restriction that only one rule is applied per level. 
Let $_{k}J$ denote the boundary state string at the beginning of the $k$-th 
level in the growth algorithm and $_{k}\hat{J}$ the associated string as in 
Definition~\ref{def:posstr}. Similarly, let $_{k}\mu$ denote the composition 
corresponding to the sign string at the $k$-th level. 
Let us compare $_{k+1}J$ and $_{k}J$. They can 
only differ in the following ways.
\begin{enumerate}
\item In case an arc-rule is applied at the $k$-th level, $_{k}J$ 
can be obtained from $_{k+1}J$ by inserting the substring 
$(1,-1)$, $(0, 0)$ or $(-1,1)$ 
between the $i$-th and $i+1$-th entries in $_{k+1}J$. $_{k}\mu$ can be obtained 
from $_{k+1}\mu$ by inserting the substring $(1, 2)$ or $(2, 1)$ 
between the $i$-th and $i+1$-th entries in $_{k+1}\mu$.
\begin{align} \label{fig:arcflow}
   \xy(0,0)*{\includegraphics[width=150px]{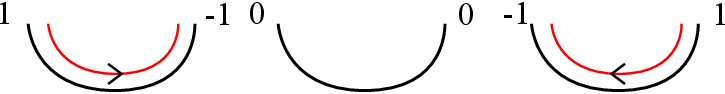}};\endxy
\end{align}
\item In case a Y-rule is applied, $_{k}J$ can be obtained from $_{k+1}J$ 
by replacing the $i$-th entry in $_{k+1}J$ with a length two substring 
whose sum is equal to the $i$-th entry. $_{k}\mu$ can be obtained from 
$_{k+1}\mu$ by replacing the $i$-th entry in $_{k+1}\mu$ with the 
substring $(3-{}_{k+1}\mu_{i}, 3-{}_{k+1}\mu_{i})$.
\begin{align} \label{fig:yflow}
   \xy(0,0)*{\includegraphics[width=300px]{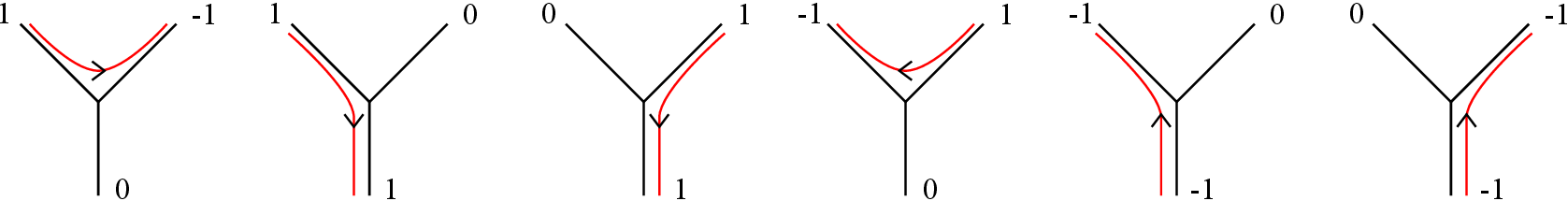}};\endxy
\end{align}
\item In case an H-rule is applied, $_{k}\mu$ can be obtained from 
$_{k+1}\mu$ by replacing a substring $(1, 2)$ or $(2, 1)$, at the $i$-th and 
$(i+1)$-th position in $_{k+1}\mu$, with 
$(3-{}_{k+1}\mu_{i}, 3-{}_{k+1}\mu_{i+1})$. $_{k}J$ can be obtained from 
$_{k+1}J$ by replacing a substring of length two in $_{k+1}J$ at the $i$-th 
and $(i+1)$-th position according to the schema.
\begin{align} \label{fig:hflow}
  &\xy(0,0)*{\includegraphics[width=220px]{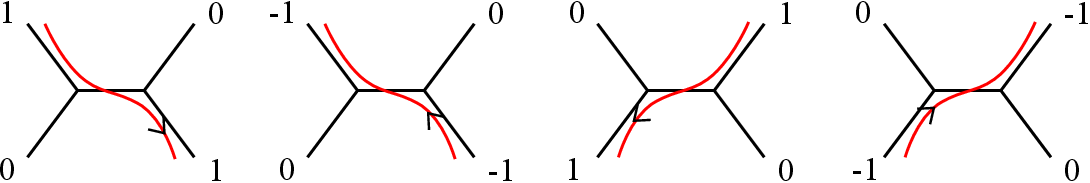}};\endxy\\
   &\xy(0,0)*{\includegraphics[width=220px]{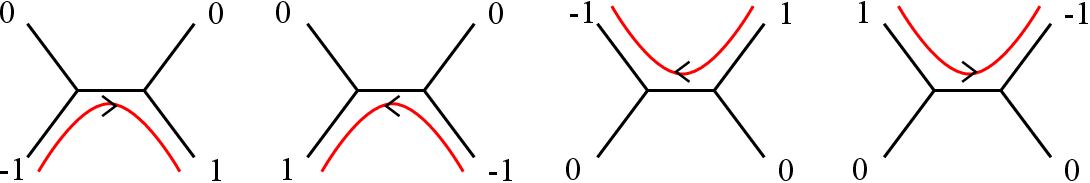}};\endxy
\end{align}
\end{enumerate}
It is straightforward to check that $_{k+1}J$ satisfies 
condition~\eqref{eqn:conds}, with composition $_{k+1}\mu$, if and only if 
$_{k}J$ does, with composition $_{k}\mu$. Take, for example, an instance where 
a Y-rule is applied; suppose also that the $i$-th entry in $_{k+1}\mu$ is 
2 and the $i$-th entry in $_{k+1}J$ is 0. Thus, the $i$-th entry in 
$_{k+1}J$ contributes a pair $(1, -1)$ to $_{k+1}\hat{J}$. 
By~\eqref{fig:yflow}, $_{k}J$ is 
obtained from $_{k+1}J$ by replacing the $i$-th entry in $_{k+1}J$ 
with $(1, -1)$ and the $i$-th entry in $_{k+1}\mu$ with $(1, 1)$. 
We see that $_{k}\hat{J}$ is in fact exactly equal to $_{k+1}\hat{J}$. 
Therefore $_{k+1}\hat{J}$ satisfies condition~\eqref{eqn:conds} if and only if  
$_{k}\hat{J}$ does. Similar analysis apply to all cases 
in~\eqref{fig:arcflow}, \eqref{fig:yflow} and~\eqref{fig:hflow}.

Let $k$ be the first level in the growth algorithm of $w$ where a Y or an 
arc-rule is applied. From the $(k+1)$-th level down we have a 
non-elliptic web $w'$ with flow, whose boundary state string $_{k+1}J$ and 
composition $_{k+1}\mu$ both have length less than $n$. Thus, by our 
induction hypothesis, $_{k+1}J$, with composition $_{k+1}\mu$, satisfies 
condition~\eqref{eqn:conds}. 

By the above argument, then ${}_{i}J$ also satisfy condition~\eqref{eqn:conds}, 
for any $0\leq i\leq k$. In particular, $J={}_0J$ satisfies that condition, 
which is what we had to prove.  
\vskip0.5cm

Conversely, let $J$ satisfy condition~\eqref{eqn:conds}, with composition 
$\mu$. We show, by induction on $n$, that there is a $w\in B^S$ 
with flow whose boundary state string is exactly $J$. More specifically, 
we first construct a $w\in W^S$ and then show that $w$ is non-elliptic, i.e. 
$w\in B^S$. 

For $n=2$, then $w$ must be an arc. 
It is simple to check that if $J$ satisfies condition~\eqref{eqn:conds}, 
$J$ is the boundary state of a flow on an arc. 

For $n>2$, suppose it is possible to apply an arc or Y-rule to 
the pair $\mu$ and $J$, depicted in~\eqref{fig:arcflow} 
and~\eqref{fig:yflow}. Then we obtain a new pair $\mu'$ and 
$J^{\prime}$ with length less than $n$. Thus, by induction, there exist a web $w'\in 
W_{S^{\prime}}$ and flow extending $J^{\prime}$. Gluing the arc or Y on top 
of $w'$ results in a web $w\in W^S$ with a flow extending $J$.

Suppose, then, that it is not possible to apply an arc or Y-rule to 
$\mu$ and $J$. This means that one of the following must hold.
\begin{enumerate}
\item \label{case1} $\mu$ does not contain a substring of type 
$(1,2)$ or $(2,1)$ and $J = (1,...,1)$, $J=(-1,...,-1)$ or $J=(0,...,0)$.
\item \label{case2}$\mu$ contains at least one substring of the form $(1, 2)$ 
or $(2, 1)$. For every substring in $\mu$ of the form $(1, 2)$ or $(2, 1)$, 
the corresponding substring in $J$ is $(\pm1, \pm1)$, $(0, 1)$ or $(1,0)$. 
For every substring in $\mu$ of the form $(1, 1)$ or $(2,2)$, 
the corresponding substring in $J$ is $(1,1)$, $(-1,-1)$ or $(0,0)$.
\end{enumerate}
Case 1 contradicts the assumption that $J$ satisfies 
condition~\eqref{eqn:conds}. 

Case~\ref{case2} contains several subcases, each of which contains details 
which are slightly different. However, the general idea is the same for all of 
them and is very simple: apply H-moves until you can apply an arc or a Y-rule 
and finish the proof by induction. 

We first suppose, without loss of generality, that 
$\mu$ contains a substring $(\mu_{i}, \mu_{i+1})=(1, 2)$ and that the 
corresponding substring in $J$ is $(j_{i}, j_{i+1})=(1, 1)$ 
(the subcase for $(j_{i}, j_{i+1})=(-1, -1)$ is analogous). 
We see that $(1, 1)$ in $J$ contributes a substring $(1, 0, 1)$ to $\hat{J}$. 
Thus, our assumption that $\hat{J}$ satisfies condition~\eqref{eqn:conds} 
implies that $\hat{J}$ contains at least one more entry equal to $-1$. 
This means that for some $r\neq i, i+1$, $1\leq r\leq n$, one of the 
following is true.
\begin{itemize}
\item[(a)] $j_{r} = -1$, $\mu_{r} = 1$, denoted for brevity by
\begin{align} 
   \xy(0,0)*{\includegraphics[width=70px]{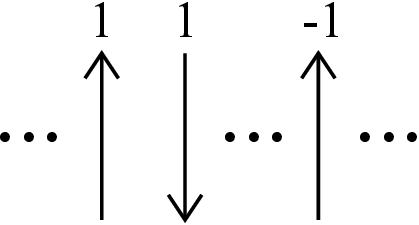}};\endxy
\end{align}
\item[(b)] $j_r = -1$, $\mu_{r} = 2$, denoted
\begin{align} 
   \xy(0,0)*{\includegraphics[width=70px]{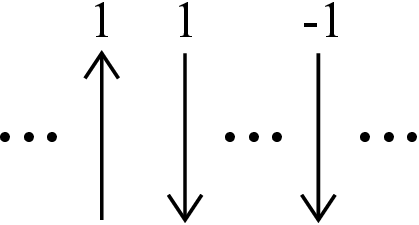}};\endxy
\end{align}
\item[(c)] $j_r = 0$, $\mu_{r} = 2$, denoted
\begin{align} 
   \xy(0,0)*{\includegraphics[width=70px]{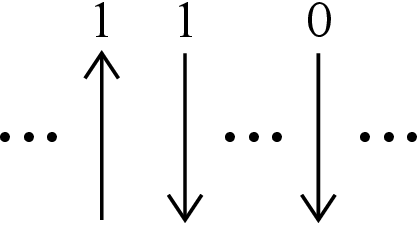}};\endxy
\end{align}
\end{itemize}
Without loss of generality, let us assume $i+1<r$. 
Consider subcases (a) and (b). If $j_{m}\neq 0$ for all $i+1<m<r$, then it 
is possible to apply an arc or Y-move to $J$ and $\mu$, contrary to our 
assumption in case~\ref{case2}. Thus, in all three scenarios above it suffices 
to analyze the following two configurations:
\begin{align} \label{fig:110}
   \xy(0,0)*{\includegraphics[width=70px]{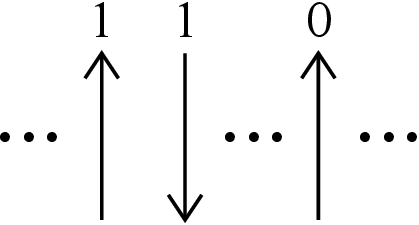}};\endxy && 			\xy(0,0)*{\includegraphics[width=70px]{section41/udd110}};\endxy
\end{align}
Let $i+1< r \leq n$ be smallest integer where $j_r = 0$. We must have that 
$\mu_{r-1}=3-\mu_{r}$ and $j_{r-1}=\pm 1$. For any other values of $\mu_{r-1}$ 
and $j_{r-1}$ we would be able to apply an arc or a Y-move, 
contradicting our assumptions for case~\ref{case2}. In both situations, 
we can apply an H-rule to the substrings $(j_{r-1}, j_{r})$ and 
$(\mu_{r-1}, \mu_{r})$ as shown below.
\begin{align} \label{fig:hmove}
   \xy(0,0)*{\includegraphics[width=80px]{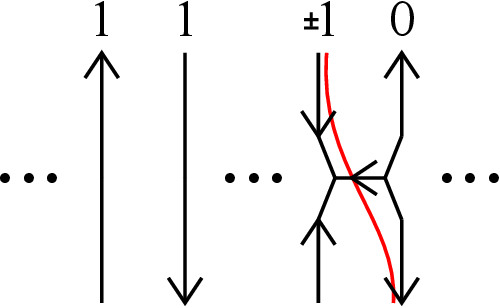}};\endxy &&
   \xy(0,0)*{\includegraphics[width=80px]{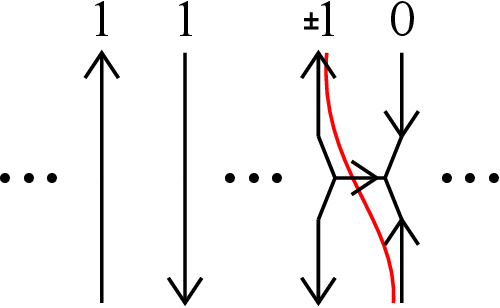}};\endxy
\end{align}
This results in new sign and state strings, each with length $n$ satisfying 
condition~\eqref{eqn:conds}. The application of the H-rule in~\eqref{fig:hmove} 
moves the zero at the $r$-th position to the $r-1$ position. 
Either we can now apply an arc or Y-rule to the new strings or 
by repeatedly applying an H-rule in the manner of~\eqref{fig:hmove}, 
we obtain one of the following pairs.
\begin{align} 
   \xy(0,0)*{\includegraphics[width=45px]{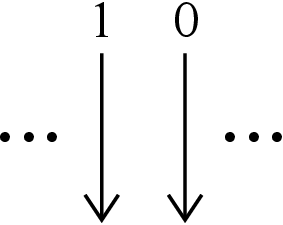}};\endxy &&
\xy(0,0)*{\includegraphics[width=45px]{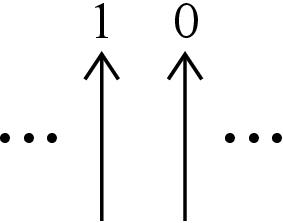}};\endxy
\end{align}
To either of the above diagrams we can apply a Y-rule, after which we 
can use induction. 

To complete our analysis of case~\ref{case2}, now suppose, without loss of 
generality, that $\mu$ contains a substring $(\mu_{i}, \mu_{i+1})=(1, 2)$ and 
that the corresponding substring in $J$ is $(j_{i}, j_{i+1})=(1, 0)$ 
(the subcases for $(0,\pm 1)$ or $(-1,0)$ are analogous).  

We see that $(1, 0)$ in $J$ contributes a substring $(1, 1,-1)$ 
to $\hat{J}$. Thus, our assumption that $\hat{J}$ satisfies 
condition~\eqref{eqn:conds} implies that $\hat{J}$ contains at least one more 
entry equal to $-1$. This means that for some $r$, with $1\leq r\leq n$, 
one of the following is true.
\begin{itemize}
\item[(a)] $j_{r} = -1$, $\mu_{r} = 1$, denoted for brevity by
\begin{align} 
   \xy(0,0)*{\includegraphics[width=70px]{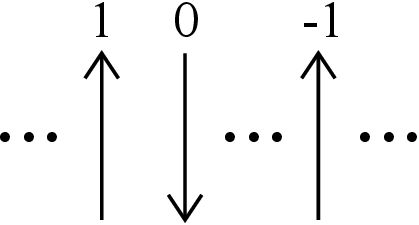}};\endxy
\end{align}
\item[(b)] $j_{r} = -1$, $\mu_{r} = 2$, denoted
\begin{align} 
   \xy(0,0)*{\includegraphics[width=70px]{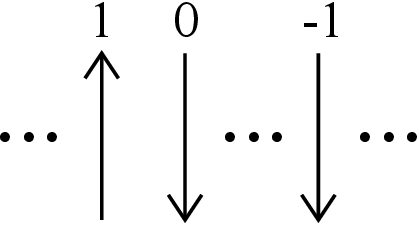}};\endxy
\end{align}
\item[(c)] $j_{r} = 0$, $\mu_{r} = 2$, denoted
\begin{align} 
   \xy(0,0)*{\includegraphics[width=70px]{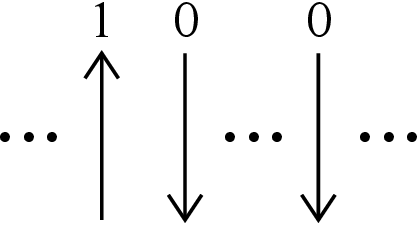}};\endxy
\end{align}
\end{itemize}
For subcases (a) and (b), if $\mu_{i+2}=1$ and $j_{i+2} =-1$, we may apply an 
H-rule to $(\mu_{i+1}, \mu_{i+2})$, $(j_{i+1}, j_{i+2})$ to obtain a new 
pair $\mu'$ and $J^{\prime}$. Subsequently we can apply a Y-rule 
to the $i$-th and $(i+1)$-th entries of $\mu'$ and $J^{\prime}$:
\begin{align}
   \xy(0,0)*{\includegraphics[width=60px]{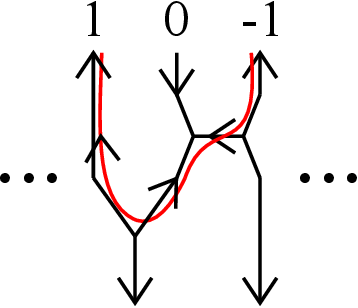}};\endxy 
\end{align}
After applying the Y-rule, we can use induction. 

Otherwise, we can show, just as before, that all three scenarios 
above reduce to an analysis of the following two configurations.
\begin{align} \label{fig:100}
   \xy(0,0)*{\includegraphics[width=70px]{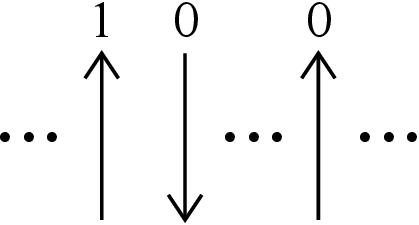}};\endxy && \xy(0,0)*{\includegraphics[width=70px]{section41/udd100}};\endxy
\end{align}
That is, we can assume that $\mu$ contains a substring 
$(\mu_{i}, \mu_{i+1})=(1, 2)$ with the corresponding substring in $J$ being 
$(j_{i}, j_{i+1})=(1, 0)$, and for some $0<r\neq i+1<n$ we have $j_{r} = 0$. 
In particular, this tells us that there exist a $0<r\neq i+1<n$ such that 
$\mu_{r}=1$ and $j_{r} = 0$.
\begin{align}
   \xy(0,0)*{\includegraphics[width=70px]{section41/udu100}};\endxy
\end{align}
This has to hold because otherwise $\hat{J}$ cannot satisfy 
condition~\eqref{eqn:conds}. Let us assume $r$ to be the smallest integer 
such that $i+1<r$, $\mu_{r}=1$ and $j_{r} = 0$. By our assumption that we 
cannot apply an arc or Y-rule to $J$ and $\mu$, we see that 
$\mu_{r-1} = 2$ and $j_{r-1} = \pm 1$. Applying an H-rule to 
$(j_{r-1}, j_{r})$ and $(\mu_{r-1}, \mu_{r})$ 
\begin{align}
   \xy(0,0)*{\includegraphics[width=80px]{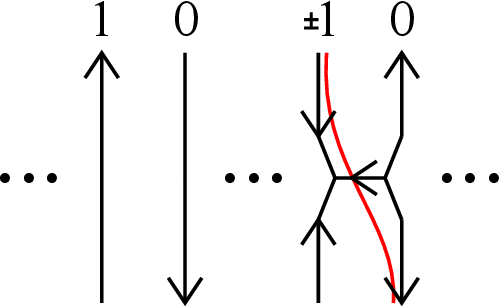}};\endxy
\end{align}
results in new sign and state strings, also with length $n$, satisfying 
condition~\eqref{eqn:conds}. The application of the H-rule in the above 
case moves the zero at the $r$-th position to the $r-1$-th position. 
Either we can now apply an arc or a Y-rule to the new sign and state strings, 
or by repeatedly apply an H-rule in the manner of~\eqref{fig:hmove}, 
we obtain a pair as below.
\begin{align} 
   \xy(0,0)*{\includegraphics[width=45px]{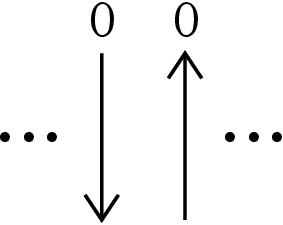}};\endxy
\end{align}
to which we can apply an arc-rule. Finally, apply induction.
\vskip0.5cm
It remains to show that the web $w$ (with flow) produced from the above algorithm is an element of $B^S$, that is, $w$ does not contain digons or squares. 
We note that, just as in~\cite{kk}, in the expression of $w$ using 
the arc, Y and H-rules, digons can only appear as the result of applying an 
arc-rule to the bottom of an H-rule, i.e. we have
\begin{align} 
   \xy(0,0)*{\includegraphics[width=25px]{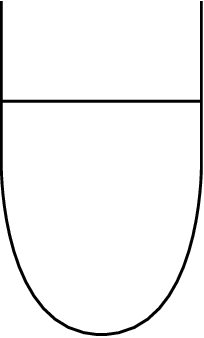}};\endxy
\end{align}

A square can only result from the following sequence of arc, Y and H-rules.
\begin{align} 
   \xy(0,0)*{\includegraphics[width=250px]{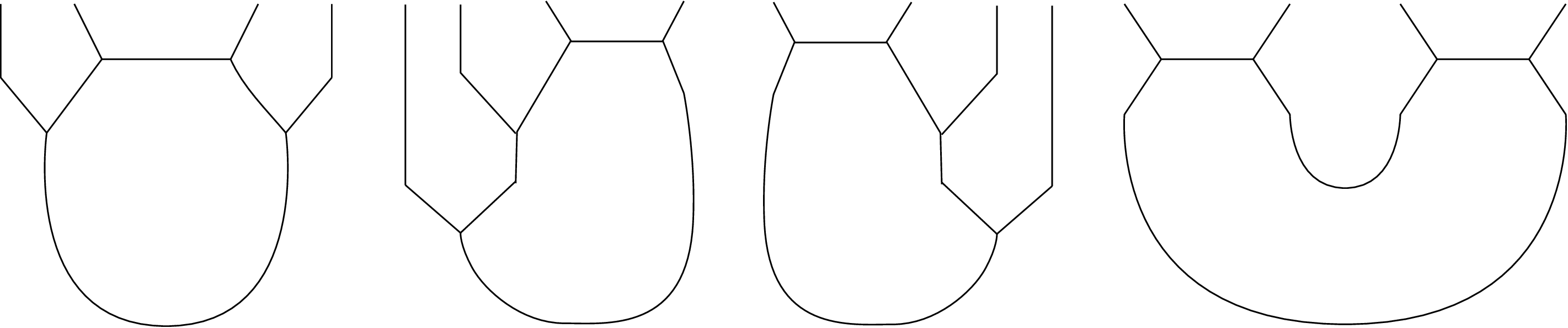}};\endxy
\end{align}
Note that in the above, we do not consider the case in which we apply 
an H-rule to the bottom of another H-rule. This is because such a case cannot arise in 
our construction of $w$.

Recall that in our inductive construction of $w$, we only apply H-rules 
equipped with the following flows.
\begin{align} 
   \xy(0,0)*{\includegraphics[width=90px]{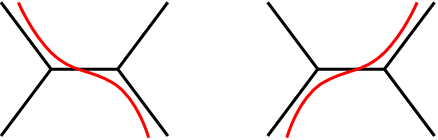}};\endxy
\end{align}

We can immediately see that is it not possible to apply an arc-rule with flow 
to the bottom of such an H-rule as shown below.
\begin{align} 
   \xy(0,0)*{\includegraphics[width=30px]{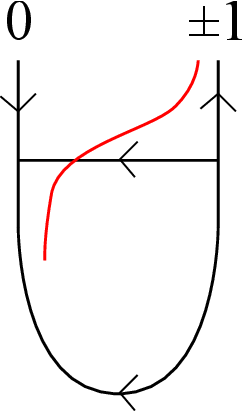}};\endxy
\end{align}

Since we only use the above two H-rules with flow, the induced flows on 
squares are as follows.
\begin{align} 
   \xy(0,0)*{\includegraphics[width=130px]{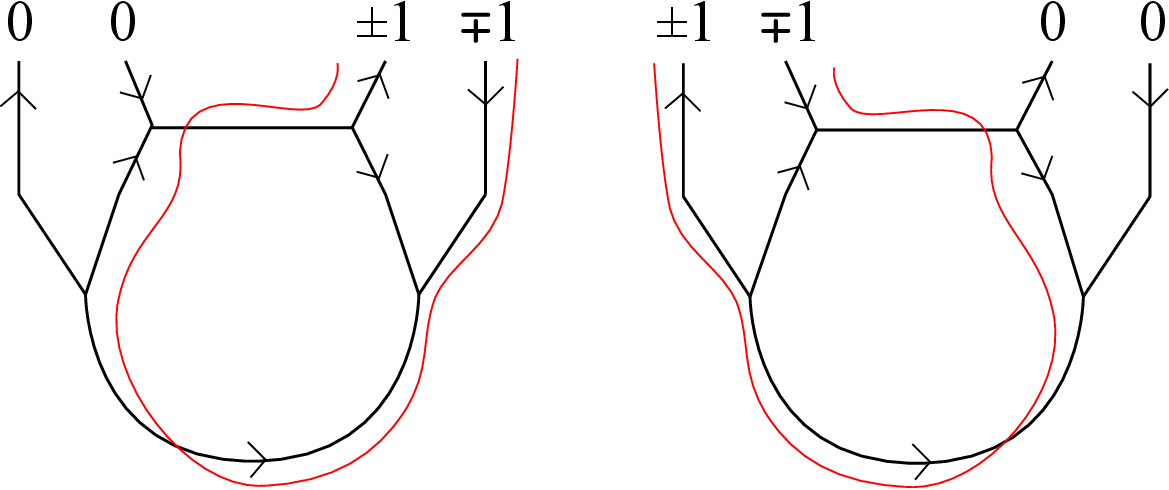}};\endxy
\end{align}

\begin{align} 
   \xy(0,0)*{\includegraphics[width=130px]{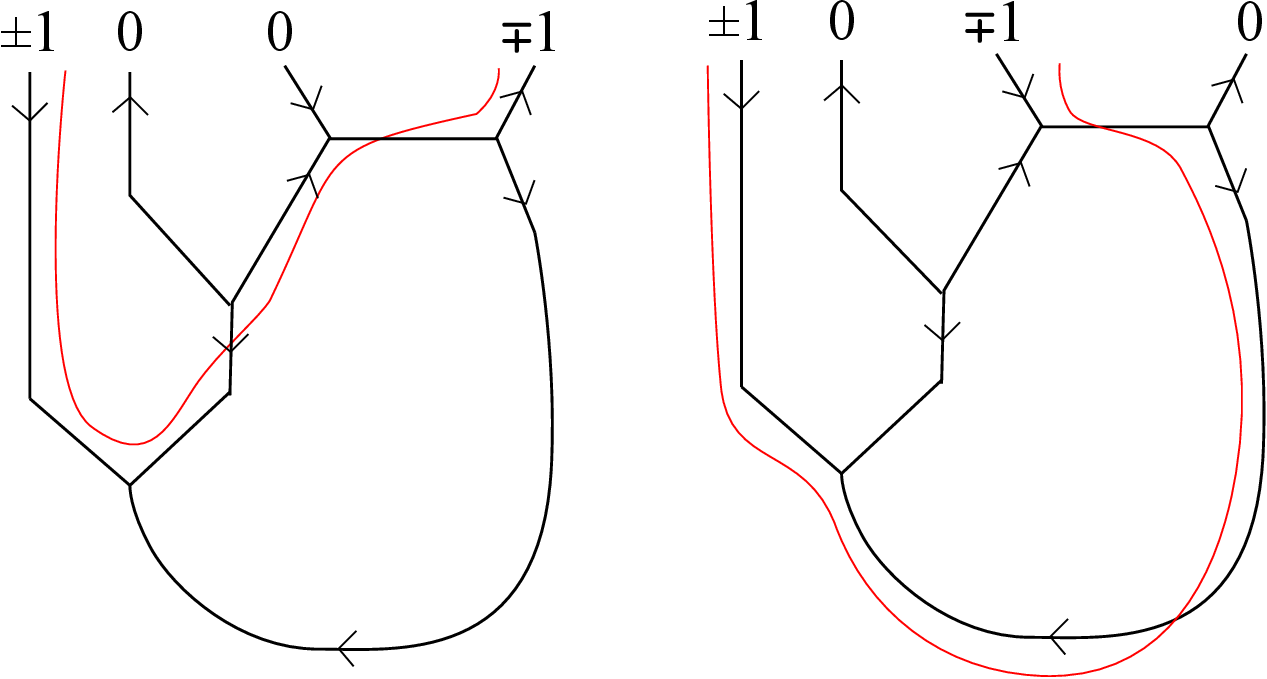}};\endxy
\end{align}

\begin{align} 
   \xy(0,0)*{\includegraphics[width=180px]{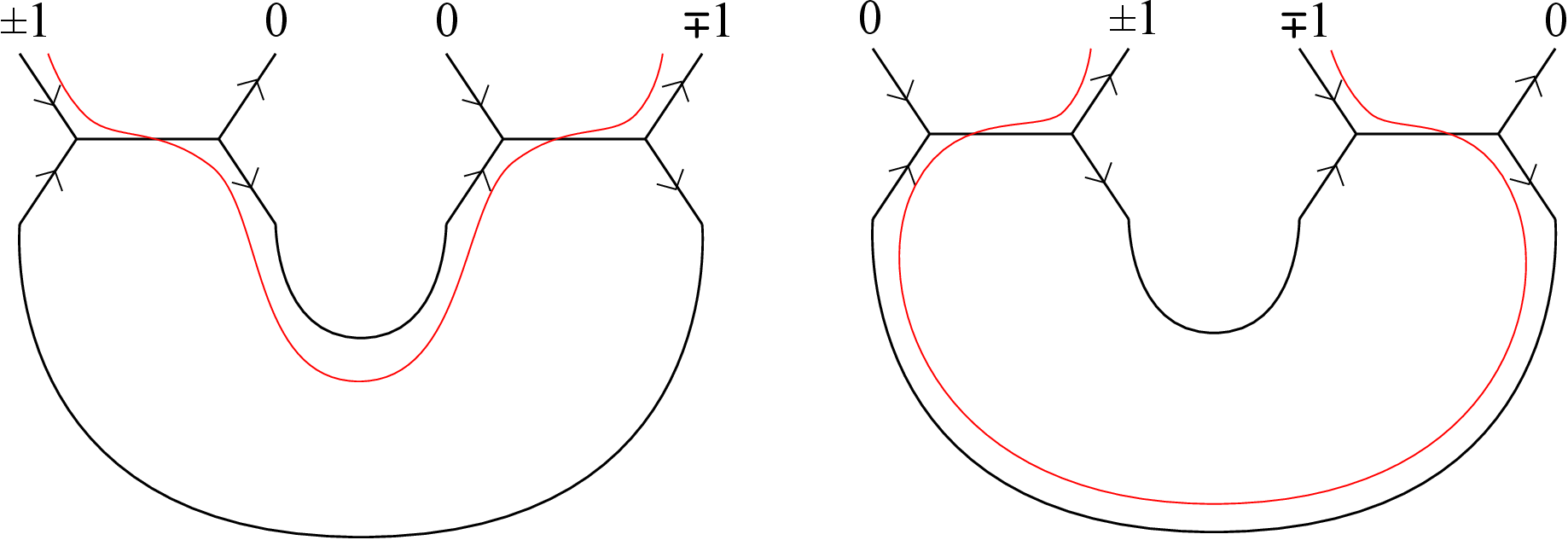}};\endxy
\end{align}
In each case, one can check that it is possible to apply an arc-rule to the 
state and sign strings 
(the same analysis applies to the cases where the faces above are given 
the opposite edge orientations). 
However, recall that an H-rule is used in our construction 
only in the case for which it is not possible to apply any other rules 
to the boundary. This implies that none of the above faces 
can appear during the construction of $w$.
\end{proof}

Implicit in the proof of Lemma~\ref{lem:strtoflow} is a procedure to construct,
from a state string $J$ satisfying condition~\eqref{eqn:conds},  
a non-elliptic web $w$ with flow extending $J$, such that $\partial w=\mu$. Note
that this procedure is not deterministic. That is, it is possible to produce 
different webs with flows extending $J$ by making different choices in the 
construction. 

\begin{ex}
\label{ex:tableauxflows}
The procedure is exemplified below. If we choose to replace the substrings as indicated in the right figure, the tableau on the left gives rise to the web with flow next to it.  
\begin{align} 
   \xy(0,0)*{\includegraphics[width=80px]{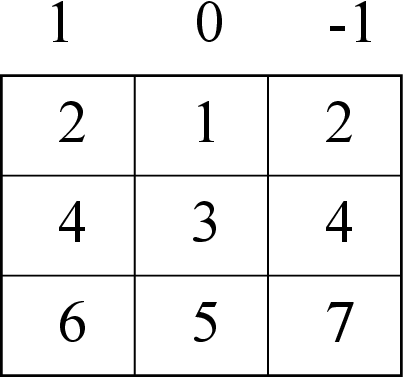}};\endxy && \xy(0,0)*{\includegraphics[width=160px]{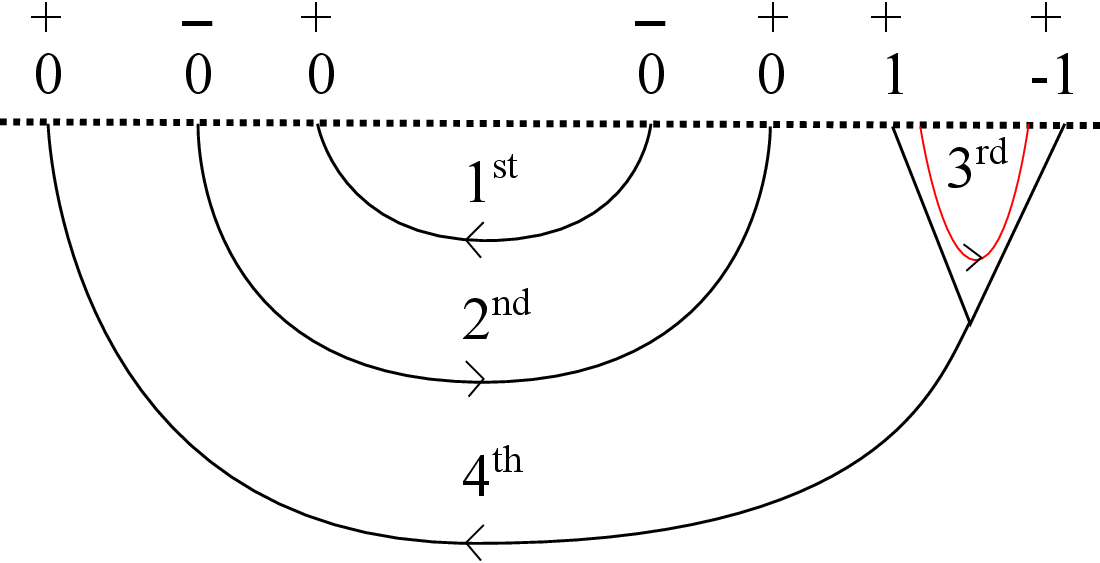}};\endxy
\end{align}
However, for other choices the same tableau generates the following web with flow.
\begin{align} 
   \xy(0,0)*{\includegraphics[width=160px]{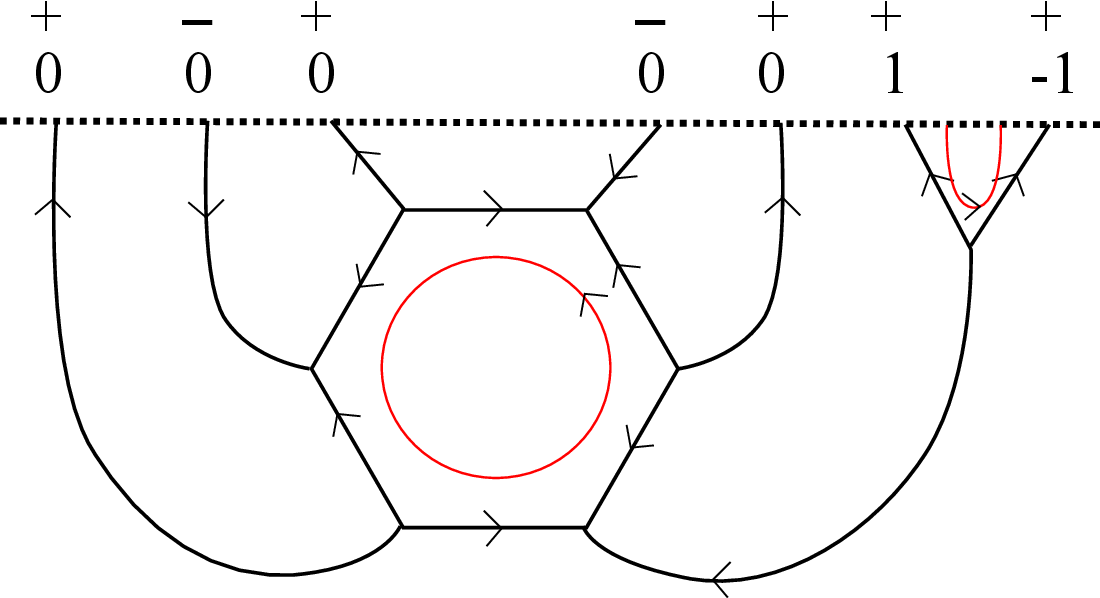}};\endxy
\end{align}
As a matter of fact, we could also invert the orientation of the flow in the internal cycle. The resulting web with flow 
would still correspond to the same tableau. 
\end{ex}

However, when we restrict to semi-standard tableaux, the 
procedure gives a unique web with flow, the canonical flow. 
One can check that the procedure implicit in Lemma~\ref{lem:strtoflow}, 
restricts to the same bijection between $\mathrm{Std}_{\mu}^{\lambda}$ and 
non-elliptic webs as defined by Russell in~\cite{ru}.


\subsection{$Z(G^S)$ and $E(Z(G^S))$}

In this subsection, $S$ continues to be a fixed sign string of length $n$. 
Moreover, we continue to use some of the other notations and conventions from 
the previous subsection as well, e.g. $d=3k\geq n$ etc.  
Let $\mu$ be the composition associated to $S$ and 
let $S_{\mu}$ be the corresponding parabolic subgroup of the 
symmetric group $S_d$. 

Let $Z(K^S)$ be the center of $K^S$ and let $X^{\lambda}_{\mu}$ be 
the Spaltenstein variety, with $\lambda=(3^k)$ and 
the notation as in~\cite{bo}. If $n_s=0$, 
then $X^{\lambda}_{\mu}=X^{\lambda}$, the latter being the Springer fiber 
associated to $\lambda$.\footnote{When comparing to Khovanov's result 
for $\mathfrak{sl}_2$, the reader should be aware that he labels the 
Springer fiber by $\lambda^T$, the transpose of $\lambda$.} 

In Theorem~\ref{thm:center}, we are going to prove that $H^*(X^{\lambda}_{\mu})$ 
and $Z(K^S)$ are isomorphic as graded algebras. 
\vskip0.5cm
Recall the following result by Tanisaki~\cite{ta}. Let 
$P=\mathbb{C}[x_1,\ldots,x_d]$ and let $I^{\lambda}$ be the ideal generated by
\begin{equation}
\label{eqn:SpringerI}
\left\{
e_r(i_1,\ldots,i_m)\;\middle\vert\;
\begin{array}{c}
m\geq 1, 1\leq i_1< \cdots < i_m\leq d\\
r> m-\lambda_{d-m+1}-\cdots-\lambda_n 
\end{array}
\right\}, 
\end{equation}
where $e_r(i_1,\ldots,i_m)\in P$ is the $r$-th elementary symmetric 
polynomial. Write 
\[
R^{\lambda}=P/I^{\lambda}.
\]
Tanisaki showed that 
\[
H^*(X^{\lambda})\cong R^{\lambda}.
\]

Note that $S_{\mu}$ acts on $P$ by permuting the variables and that it maps 
$I^{\lambda}$ to itself. Moreover, let $P^{\mu}=P^{S_{\mu}}\subset P$ be the subring of 
polynomials which are invariant under $S_{\mu}$.

For 
$1\leq i_1\leq \cdots\leq i_m\leq n$ and $r\geq 1$, we let 
$e_r(\mu, i_1,\ldots,i_m)$ denote the $r$-th elementary 
symmetric polynomials in the variables $X_{i_1}\cup\cdots\cup X_{i_m}$, where 
\[
X_{p}=
\left\{x_k\mid \mu_1+\cdots+\mu_{p-1}+1\leq k\leq \mu_1+\cdots+\mu_p\right\}.
\]
So, we have 
\[
e_r(\mu, i_1,\ldots,i_m)=\sum_{r_1+\cdots+r_m=r}e_{r_1}(\mu;i_1)\cdots 
e_{r_m}(\mu;i_m).
\]
If $r=0$, we set $e_r(\mu, i_1,\ldots,i_m)=1$ and if $r<0$, 
we set $e_r(\mu, i_1,\ldots,i_m)=0$. 
Let $I^{\lambda}_{\mu}$ be the ideal generated by 
\begin{equation}
\label{eqn:SpaltensteinI}
\left\{
e_r(\mu, i_1,\ldots,i_m)\;\middle\vert\;
\begin{array}{l}
m\geq 1, 1\leq i_1< \cdots < i_m\leq d\\
r> m-\mu_{i_1}+\cdots+\mu_{i_m}-\lambda_{l+1}-\cdots-\lambda_n \\
\mathrm{where}\;l=\#\{i\mid \mu_i>0, i\ne i_1,\ldots, i_m\}
\end{array}
\right\}.
\end{equation}
Note that $I^{\lambda}_{\mu}\subseteq I^{\lambda}$ holds. Write 
\[
R^{\lambda}_{\mu}=P^{\mu}/I^{\lambda}_{\mu}.
\]
Brundan and Ostrik~\cite{bo} proved that 
\[
H^*(X^{\lambda}_{\mu})\cong R^{\lambda}_{\mu}.
\]

First we want to show that $R^{\lambda}_{\mu}$ acts on $K^S$. Clearly, 
$P^{\mu}$ acts on $K^S$, by converting polynomials into dots on the facets 
meeting $S$. 

\begin{lem}
\label{lem:I}
The ideal $I^{\lambda}_{\mu}$ annihilates any foam in $K^S$. 
\end{lem}
\begin{proof}
The following argument demonstrates that it suffices to show this for the 
case when $n_s=0$.
Let $u,v\in B^S$. For each $1\leq i\leq n$ with $s_i=-$, glue a $Y$ onto 
the $i$-th boundary edge of $u$ and $v$, respectively. Call these 
new webs $\hat{u}$ and $\hat{v}$, respectively. 
Note that $\partial\hat{u}=\partial\hat{v}=\hat{S}$, 
where $\hat{S}=(+^d)$. 
Let $f\in {}_uK_v$ be any foam. For each $1\leq i\leq n$ with $s_i=-$, 
glue a digon foam on top of the 
$i$-th facet of $f$ meeting $S$. The new foam $\hat{f}$, 
obtained in this way, belongs to ${}_{\hat{u}}K_{\hat{v}}$. Note that we can 
re-obtain $f$ by capping off $\hat{f}$ with dotted digon foams. Any 
polynomial $p\in I^{\lambda}_{\mu}\subseteq I^{\lambda}$ acting on $f$ 
also acts on $\hat{f}$ (using the relations in~\eqref{eq:dotm} on the digon foams). 
So, if we know that $p\hat{f}=0$, then it follows that 
$pf=0$. 
\vskip0.5cm
Thus, without loss of generality, assume that $n_s=0$. 
We are now going to show that $I^{\lambda}$ annihilates $K^S$.  

As follows from Definition in~\eqref{eqn:SpringerI}, 
$I^{\lambda}$ is generated by the elementary symmetric polynomials 
$e_r(x_{i_1},\ldots,x_{i_m})$, for the following values of $m$ and $r$.
\[
\begin{array}{lll}
m=2n+1 &;&r>2n-2,\\
m=2n+2 &;&  r>2n-4,\\
\vdots &\vdots& \vdots\\
m=3n-1 &;& r>2,\\
m=3n & ;& r>0.
\end{array}
\]
Note that for $m=3n$, we simply get all completely symmetric polynomials of 
positive degree in the variables $x_1,\ldots,x_d$. Any such polynomial $p$ 
annihilates any foam $f\in {}_uK_v$, because 
by the complete symmetry of $p$, the dots can all be moved to the 
three facets around one singular edge.
The relations~\eqref{eq:dotm} then show that $p$ kills $f$.

Now suppose $m=3n-\ell$, for $\ell>0$. So we must have $r>2\ell$. The 
argument we are going give does 
not depend on the particular choice of 
$i_1,\ldots,i_m\subseteq \{1,2,\ldots,d\}$, so, without loss of 
generality, let us assume that $(i_1,\ldots,i_m)=(1,\ldots,m)$. 

Let $f$ be any foam in ${}_uK_v$. 

First assume that $\ell=1$. 
\begin{align*}
&e_r(x_1,\ldots,x_{d-1})f\\
=&-e_{r-1}(x_1,\ldots,x_{d-1})x_df\\
=&e_{r-2}(x_1,\ldots,x_{d-1})x_d^2f\\
\vdots&\phantom{.}\\
=&(-1)^r x_d^rf,
\end{align*}
All these equations follow from the fact that, for any $j>0$, we have  
\[
e_j(x_1,\ldots,x_{d})=e_j(x_1,\ldots,x_{d-1})+e_{j-1}(x_1,\ldots,x_{d-1})x_d,
\]
and the fact that $e_j(x_1,\ldots,x_{d})f=0$, as we proved above in the 
previous case for $m=3n$.  
Since in this case we have $r>2$, we see that 
\[
(-1)^r x_d^rf=0,
\]
by Relation (3D). This finishes the proof for this case. 

In general, for $\ell\geq 1$, we get that 
$e_r(x_1,\ldots,x_{d-\ell})f$ is equal to a linear combination of 
terms of the form 
\[
x_{d-\ell+1}^{r_1}x_{d-\ell+2}^{r_2}\cdots x_{d}^{r_{\ell}}f,
\]
with $r_1+\cdots+r_{\ell}=r$. Since $r>2\ell$, there exists 
a $1\leq j\leq \ell$ such that $r_j>2$, in each term. 
So each term kills $f$, by Relation (3D). This 
finishes the proof.
\end{proof}

Note that Lemma~\ref{lem:I} shows that there is a well-defined homomorphism of 
graded algebras 
$c_S\colon R^{\lambda}_{\mu}\to Z(K^S)$, defined by 
\[
c_S(p)=p1.
\] 

Similarly, there is a filtration preserving homomorphism 
\[
P^{\mu}\to Z(G^S)
\]
defined by 
$p\mapsto p1$. This homomorphism does not descend to 
$R^{\lambda}_{\mu}$, because the relations in $G^S$ are deformations 
of those in $K^S$, but the associated graded homomorphism 
maps $P^{\mu}$ to $E(Z(G^S))$ and we have 
\[
E(P^{\mu}1)=R^{\lambda}_{\mu}1.
\]
Before giving our following result, we recall that Brundan and Ostrik~\cite{bo} 
showed that 
\[
\dim H^*(X^{\lambda}_{\mu})=\#\mathrm{Col}^{\lambda}_{\mu}.
\]
They actually gave a concrete basis, but we do not need it here. 

\begin{lem}
\label{lem:dimZG}
We have 
\[
\dim Z(G^S)=\#\mathrm{Col}^{\lambda}_{\mu}.
\] 
\end{lem}
\begin{proof}
Let $J$ be any state-string satisfying condition~\eqref{eqn:conds}. 
We define 
\begin{equation}
\label{eq:centralidempotent}
z_J=\sum_{u\in B^S} \sum_{T} e_{u,T}\in G^S,
\end{equation}
where the second sum is over all 3-colorings of $u$ extending $J$.

First we show that $z_J\in Z(G^S)$. For any $u,v\in B^S$, 
let $f\in {}_uG_v$. Choose two arbitrary 
compatible colorings $T_1$ and $T_2$ of $u$ and $v$, respectively. 
Assume that $e_{T_1}fe_{T_2}\ne 0$. Then we have 
\[
z_Je_{T_1}fe_{T_2}=
\begin{cases}
e_{T_1}fe_{T_2},&\quad\text{if}\quad T_1\;\text{extends}\;J,\\
0&\quad\text{else}. 
\end{cases}
\]  
We also have 
\[
e_{T_1}fe_{T_2}z_J=
\begin{cases}
e_{T_1}fe_{T_2},&\quad\text{if}\quad T_2\;\text{extends}\;J,\\
0&\quad\text{else}. 
\end{cases}
\]  
This shows that $z_J\in Z(G^S)$, because $T_1$ and $T_2$ are compatible, 
and so $T_1$ extends $J$ if and only if $T_2$ extends $J$.

Note that 
\[
\sum_J z_J=1\quad\text{and}\quad z_Jz_{J^{\prime}}=\delta_{J,J^{\prime}}z_J.
\]
In particular, the $z_J$'s are linearly independent. 
\vskip0.5cm
For any state-string $J$ satisfying condition~\eqref{eqn:conds}, 
the central idempotent $z_J$ belongs to $P^{\mu}1$. In order to see this, 
first note that, for 
any $u\in B^S$, the element
\[
z_J1_u=\sum_{T\;\mathrm{extends}\;J} e_{u,T},
\]
belongs to $P^{\mu}1_u$. This holds, because only the colors of the boundary 
edges of $u$ are fixed. We can sum over all possible 3-colorings of the other 
edges, which implies that these edges only contribute a factor $1$ to 
$z_J1_u$. Furthermore, we see that $z_J1_u=p_J1_u$, for a fixed polynomial 
$p_J\in P^{\mu}$, i.e. $p_J$ is independent of $u$. Therefore, we have 
\[
z_J=\sum_{u\in B^S} p_J1_u=p_J1\in P^{\mu}1.
\]
It remains to show that $Z(G^S)z_J=\mathbb{C}z_J$. 
Let $z\in Z(G^S)$. By the orthogonality of Gornik's symmetric idempotents, 
we have 
\[
z=\sum_{u,T}e_{u,T}ze_{u,T}.
\] 
By Theorem~\ref{thm:Gornik}, we know that 
\[
e_{u,T}ze_{u,T}=\lambda_{u,T}(z) e_{u,T},
\]
for a certain $\lambda_{u,T}(z)\in\mathbb{C}$. Therefore, we have 
\[
z=\sum_{u,T}\lambda_{u,T}(z)e_{u,T}\in \bigoplus_{u,T}\mathbb{C}e_{u,T}.
\]
By Lemma~\ref{lem:strtoflow}, we know that $z_J\ne 0$. This shows that 
\[
\{z_J\mid J\;\text{satisfying condition~\eqref{eqn:conds}}\}
\] 
forms a basis of $Z(G^S)$. By Proposition~\ref{prop:tableauxflows}, 
the claim of the lemma follows. 
\end{proof}

\begin{thm}
\label{thm:center}
The degree preserving algebra homomorphism
\[
c_S\colon R^{\lambda}_{\mu_S}\to Z(K^S)
\]
is an isomorphism.
\end{thm}
\begin{proof}
In Corollary~\ref{cor:moritacenter} it will be shown that 
\[
\dim H^*(X^{\lambda}_{\mu_S})=\dim Z(K^S),
\] 
so it suffices to show that $c_S$ is injective.   

Lemma~\ref{lem:I} shows that (as graded complex algebras)
\[
R^{\lambda}_{\mu}1\subset Z(K^S).
\] 
As already mentioned above, Brundan and Ostrik~\cite{bo} showed that 
\[
H^*(X^{\lambda}_{\mu})\cong R^{\lambda}_{\mu}
\] 
as graded complex algebras. 

The proof of Lemma~\ref{lem:dimZG} shows that the filtration preserving 
homomorphism 
\[
P^{\mu}\to Z(G^S),
\]
defined by 
$p\mapsto p1$, is surjective. Note the $E(\cdot)$ is not a map. However, 
a filtered algebra $A$ and its associated graded $E(A)$ 
are isomorphic as vector spaces. In particular, they satisfy 
\[
\dim A =\dim E(A). 
\]
Therefore, since $p\mapsto p1$ is a surjection of vector spaces, we have 
\[
\dim Z(G^S)=\dim E(Z(G^S))=\dim E(P^{\mu}1)=\dim P^{\mu}1.
\]
Recall that $E(P^{\mu}1)=R^{\lambda}_{\mu}1$ and 
$\dim Z(G^S)=\dim R^{\lambda}_{\mu}$. This shows 
\[
\dim R^{\lambda}_{\mu}1=\dim P^{\mu}1=\dim Z(G^S)=\dim R^{\lambda}_{\mu},
\]
which implies that the map $c_S$ is injective. 
\end{proof}


\section{Web algebras and the cyclotomic KLR algebras}
\label{sec:grothendieck}

\subsection{Howe duality}
\label{subsec:howe}
We first recall classical Howe duality briefly. Our main references are~\cite{ho2} and~\cite{ho1}, 
where the reader can find the proofs of the results which we recall below and other details. 

Let us briefly explain Howe duality.\footnote{We follow Kamnitzer's exposition in ``The ubiquity of Howe duality'', which is online available at https://sbseminar.wordpress.com/2007/08/10/the-ubiquity-of-howe-duality/.} The two natural actions of $\mathrm{GL}_m=\mathrm{GL}(m,\mathbb{C})$ and of $\mathrm{GL}_n=\mathrm{GL}(n,\mathbb{C})$ on $\mathbb{C}^m\otimes \mathbb{C}^n$ commute and the two groups are each others commutant. We say that the actions of $\mathrm{GL}_m$ and $\mathrm{GL}_n$ are \textit{Howe dual}. 

More interestingly, their actions on the symmetric powers
\[
S^p\left(\mathbb{C}^m\otimes \mathbb{C}^n\right)
\]
and on the alternating powers
\[
\Lambda^p\left(\mathbb{C}^m\otimes \mathbb{C}^n\right)
\]
are also Howe dual, for any $p\in\mathbb{N}$. These are called the 
\textit{symmetric} and the \textit{skew} Howe duality of $\mathrm{GL}_m$ and 
$\mathrm{GL}_n$, respectively. In this paper, we are only considering 
skew Howe duality. 

Note that skew Howe duality implies that 
we have the following decomposition into irreducible 
$\mathrm{GL}_m\times \mathrm{GL}_n$-modules.
\begin{equation}
\label{eq:decomp}
\Lambda^p\left(\mathbb{C}^m\otimes \mathbb{C}^n\right)\cong 
\bigoplus_{\lambda} V_{\lambda}\otimes W_{\lambda'},
\end{equation}
where $\lambda$ ranges over all partitions with $p$ boxes and at most 
$m$ rows and $n$ columns and $\lambda'$ is the transpose of $\lambda$.

Here $V_{\lambda}$ is the unique irreducible $\mathrm{GL}_m$-module of 
highest weight $\lambda$ and $W_{\lambda'}$ is the unique irreducible 
$\mathrm{GL}_n$-module of highest weight $\lambda'$.

Without giving a full 
proof of~\eqref{eq:decomp}, which can be found in Section 4.1 of~\cite{ho2}, 
we note that it is easy to write down 
the highest weight vectors in the decomposition of 
\[
\Lambda^p\left(\mathbb{C}^m\otimes \mathbb{C}^n\right).
\]
Define 
\[
\epsilon_{ij}=\epsilon_i\otimes \epsilon_j,
\]
for any $1\leq i\leq m$ and $1\leq j\leq n$. Here the $\epsilon_i$ and the 
$\epsilon_j$ are the canonical basis elements of $\mathbb{C}^m$ and $\mathbb{C}^n$ 
respectively. Let $\lambda$ be one of the highest $\mathrm{GL}_m$ weights 
in~\eqref{eq:decomp}. Write $\lambda=(\lambda_1,\ldots,\lambda_m)$ with 
$n\geq \lambda_1\geq 
\lambda_2\geq \cdots\geq \lambda_m\geq 0$. Then  
\begin{eqnarray*}
v_{\lambda,\lambda'}&=&\left(\epsilon_{11}\wedge \cdots\wedge \epsilon_{1\lambda_1}\right)\wedge
  \left(\epsilon_{21}\wedge \cdots\wedge \epsilon_{2\lambda_2}\right)\wedge
  \left(\epsilon_{m1}\wedge \cdots\wedge \epsilon_{m\lambda_m}\right)\\
&=&\pm\left(\epsilon_{11}\wedge \cdots\wedge \epsilon_{\lambda'_1 1}\right)\wedge
  \left(\epsilon_{12}\wedge \cdots\wedge \epsilon_{\lambda'_2 2}\right)\wedge
  \left(\epsilon_{1n}\wedge \cdots\wedge \epsilon_{\lambda'_n n}\right)
\end{eqnarray*}
is a highest $\mathrm{GL}_m\times\mathrm{GL}_n$-weight. 
By convention, we exclude factors $\epsilon_{ij}$ for which $\lambda_i=0$ or 
$\lambda_j'=0$.

Now restrict to $\mathrm{SL}_m$ and assume that $p=mk$, 
for some $k\in\mathbb{N}$. 
By Schur's lemma, the decomposition in~\eqref{eq:decomp} implies that 
\begin{equation}
\label{eq:howehom1}
\mathrm{Inv}_{\mathrm{GL}_m}\left(\Lambda^p\left(\mathbb{C}^m\otimes 
\mathbb{C}^n\right)\right)\cong 
\mathrm{Hom}_{\mathrm{SL}_m}\left(\mathbb{C},\Lambda^p\left(\mathbb{C}^m\otimes 
\mathbb{C}^n\right)\right)\cong W_{(m^k)},
\end{equation}  
where $\mathbb{C}$ denotes the trivial representation. 

Decompose 
\[
\mathbb{C}^n\cong \mathbb{C}\epsilon_1\oplus \mathbb{C}\epsilon_2\oplus 
\cdots\oplus \mathbb{C}\epsilon_n
\]
into its one-dimensional $\mathfrak{gl}_n$-weight spaces. Then we have 
\begin{equation}
\label{eq:decomp2}
\Lambda^p\left(\mathbb{C}^m\otimes \mathbb{C}^n\right)\cong 
\bigoplus_{(p_1,\ldots,p_n)\in \Lambda(n,p)}\Lambda^{p_1}\left(\mathbb{C}^m\right) 
\otimes \Lambda^{p_2}\left(\mathbb{C}^m\right) \otimes \cdots\otimes 
\Lambda^{p_n}\left(\mathbb{C}^m\right) 
\end{equation}
as $\mathrm{GL}_m\times T$-modules, where $T$ is the diagonal torus in 
$\mathrm{GL}_n$.  

This decomposition implies that  
\begin{equation}
\label{eq:howehom2}
\mathrm{Inv}_{\mathrm{SL}_m}\left(\Lambda^{p_1}\left(\mathbb{C}^m\right) 
\otimes \Lambda^{p_2}\left(\mathbb{C}^m \right)\otimes \cdots\otimes 
\Lambda^{p_n}\left(\mathbb{C}^m\right)\right)\cong W(p_1,\ldots,p_n), 
\end{equation}
where $W(p_1,\ldots,p_n)$ denotes the 
$(p_1,\ldots,p_n)$-weight space of $W_{(m^k)}$.  
\vskip0.5cm
In the next subsection, 
we use Kuperberg's $\mathfrak{sl}_3$-webs to give a $q$-version 
of the isomorphism in~\eqref{eq:howehom2}, for $U_q(\mathfrak{sl}_3)$ 
and $U_q(\mathfrak{gl}_n)$ with $n=3k$ and $k\in\mathbb{N}$ arbitrary 
but fixed.

When we finished the first complete version of this paper, 
the only available results on the quantum version of skew Howe duality 
were those in Section 6.1 in~\cite{caut}. Cautis's paper does not contain 
Definition~\ref{defn:phi} nor Lemma~\ref{lem:phi}, which are our main 
results in the next subsection. 

Independently and around the same time as our preprint appeared, 
Cautis, Kamnitzer and Morrison finished a paper on $\mathfrak{sl}_n$-webs in 
which they gave the $\mathfrak{sl}_n$-generalization of 
Definition~\ref{defn:phi} and 
Lemma~\ref{lem:phi}. Their paper is now the best reference for 
quantum skew Howe duality in 
general. We therefore refer to their paper 
for the general case of quantum skew Howe duality 
(for some extra details the reader might also want to have a look 
at~\cite{mack1}) and restrict ourselves to the case of interest to us 
in this paper. 

As already mentioned in the introduction, Lauda, Queffelec and Rose~\cite{lqr} 
wrote an independent paper in which they defined and used $\mathfrak{sl}_2$ and 
$\mathfrak{sl}_3$-foams to categorify special cases of quantum skew Howe duality. Their $\mathfrak{sl}_3$ case is very similar to ours, but is used for 
the purpose of studying $\mathfrak{sl}_3$-knot homology. The decategorification of 
their results also contains the analogue of Definition~\ref{defn:phi} and 
Lemma~\ref{lem:phi}.       
\subsection{The uncategorified story}\label{sec-webhowea}
\subsubsection{Enhanced sign sequences}
In this section we slightly generalize the notion of a sign sequence/string. We call this generalization \textit{enhanced sign sequence} or \textit{enhanced sign string}. Note that, with a slight abuse of notation, we use $\hat{S}$ for sign strings and $S$ for enhanced sign string throughout the whole section. 
\begin{defn} An {\em enhanced sign sequence/string} is 
a sequence $S=(s_1,\ldots, s_n)$ with entries $s_i\in\{\circ,-1,+1,\times\}$, for 
all $i=1,\ldots n$. The corresponding weight $\mu=\mu_S\in\Lambda(n,d)$ is 
given by the rules
\[
\mu_i=
\begin{cases}
0,&\quad\text{if}\;s_i=\circ,\\
1,&\quad\text{if}\;s_i=1,\\
2,&\quad\text{if}\;s_i=-1,\\
3,&\quad\text{if}\;s_i=\times.
\end{cases}
\] 
Let $\Lambda(n,d)_3\subset \Lambda(n,d)$ be the subset of weights 
with entries between $0$ and $3$. Recall that $\Lambda(n,d)_{1,2}$ denotes the subset of weights with only $1$ and $2$ as entries.
\end{defn}

Let $n=d=3k$. For any enhanced sign string $S$ such that 
$\mu_S\in\Lambda(n,n)_3$, we define $\hat{S}$ to be the sign sequence 
obtained from $S$ by deleting all entries that are equal to $\circ$ or 
$\times$ and keeping the linear ordering 
of the remaining entries. Similarly, for any $\mu\in\Lambda(n,n)_3$, 
let $\hat{\mu}$ be the weight obtained from $\mu$ by deleting all entries 
which are equal to $\circ$ or $3$. Thus, if $\mu=\mu_S$, for a certain enhanced 
sign string $S$, then $\hat{\mu}=\mu_{\hat{S}}$. 
Note that $\hat{\mu}\in\Lambda(m,d)_{1,2}$, for a certain 
$0\leq m\leq n$ and $d=3(k-(n-m))$. 

Note that for any semi-standard tableau $T\in \mathrm{Std}^{(3^k)}_{\mu}$, 
there is a unique 
semi-standard tableau $\hat{T}\in \mathrm{Std}^{(3^{k-(n-m)})}_{\hat{\mu}}$, 
obtained by deleting any cell in $T$ whose label appears three times and 
keeping the linear ordering of the remaining cells within each column. 

Conversely, let $\mu'\in\Lambda(m,d)_{1,2}$, 
with $m\leq n$ and $d=3(k-(n-m))$. In general, there is more than one 
$\mu\in \Lambda(n,n)_3$ such that $\hat{\mu}=\mu'$, but at least one. Choose one of them, 
say $\mu_0$. Then, given any $T'\in \mathrm{Std}^{(3^{k-(n-m)})}_{\mu'}$, there is 
a unique $T\in \mathrm{Std}^{(3^k)}_{\mu_0}$ such that 
$\hat{T}=T'$.

The construction of $T$ is as follows. Suppose that 
$i$ is the smallest number such that $(\mu_0)_i=3$.
\begin{itemize}
\item[(1)] In each column $c$ of $T'$, there is a unique vertical position such that 
all cells above that position have label smaller than $i$ and 
all cells below that position have label greater than $i$. Insert a new 
cell labeled $i$ precisely in that position, for each column $c$.
\item[(2)] In this way, we obtain a new tableau of shape $(3^{k-(n-m)+1})$. It is easy to 
see that this new tableau is semi-standard. Now apply this procedure 
recursively for each $i=1,\ldots, n$, such that $(\mu_0)_i=3$.
\item[(3)] In this way, we obtain a tableau $T$ of shape $(3^k)$. Since in each step the 
new tableau that we get is semi-standard, we see that $T$ belongs to 
$\mathrm{Std}^{(3^k)}_{\mu_0}$.  
\end{itemize} 
Note also that $\hat{T}=T'$. This shows that 
for a fixed $\mu\in\Lambda(n,n)_3$, we have a bijection
\[
\mathrm{Std}^{(3^k)}_{\mu}\ni T\; \longleftrightarrow\; \hat{T}\in
\mathrm{Std}^{(3^{k-(n-m)})}_{\hat{\mu}}.
\]
Given an enhanced sign sequence $S$, such that $\mu_S\in\Lambda(n,n)_3$, 
we define 
\[
W^S=W^{\hat{S}}.
\]
In other words, as a vector space $W^S$ does not depend on the $\circ$ and 
$\times$-entries of $S$. However, they do play an important role below.   
Similarly, we define 
\[
B^S=B^{\hat{S}}\quad\text{and}\quad K^S=K^{\hat{S}}.
\]
\subsubsection{An instance of $q$-skew Howe duality}
Let $V_{(3^k)}$ be the irreducible 
$U_q(\mathfrak{gl}_n)$-module of highest weight $(3^k)$. 
By restriction, $V_{(3^k)}$ is also a $U_q(\mathfrak{sl}_n)$-module 
and, since it is a weight representation, it is a 
$\dot{\mathbf U}(\mathfrak{sl}_n)$-module, too.
It is well-known (see~\cite{fu} and~\cite{ma}) for example) that 
\[
\dim V_{(3^k)}=\sum_{\mu\in\Lambda(n,n)_3}\#\mathrm{Std}^{(3^k)}_{\mu}.
\] 
Note that a tableau of shape $(3^k)$ can only be semi-standard if its 
filling belongs to $\Lambda(n,n)_3$, so strictly speaking we could 
drop the $3$-subscript. More precisely, if 
\[
V_{(3^k)}=\bigoplus_{\mu\in\Lambda(n,n)_3} V^{\mu}
\]
is the $U_q(\mathfrak{gl}_n)$-weight decomposition of $V_{(3^k)}$, then 
\[
\dim V^{\mu}=\#\mathrm{Std}^{(3^k)}_{\mu}.
\] 
Note that the action of $U_q(\mathfrak{gl}_n)$ on $V_{(3^k)}$ descends to 
$S_q(n,n)$ and recall that there exists a surjective algebra homomorphism 
\[
\psi_{n,n}\colon \U\to S_q(n,n).
\]
The action of $\U$ on $V_{(3^k)}$ is equal to the pull-back of the 
action of $S_q(n,n)$ via $\psi_{n,n}$. 

Define 
\[
W_{(3^k)}=\bigoplus_{S\in \Lambda(n,n)_3} W^S.
\]
Below, we will show that $S_q(n,n)$ acts on $W_{(3^k)}$. Pulling 
back the action via $\psi_{n,n}$, we see that $W_{(3^k)}$ is a 
$\U$-module. We will also show that 
\[
W_{(3^k)}\cong V_{(3^k)}
\]
as $S_q(n,n)$-modules, and therefore also as $\U$-modules, and that the space
$W^S$ corresponds to the $\mu_S$-weight space of $V_{(3^k)}$. 
\vskip0.5cm
Let us define the aforementioned left action of $S_q(n,n)$ on $W_{(3^k)}$. 
The reader should compare this action to the categorical action on the 
objects in Section 4.2 in~\cite{msv2}. Note that our conventions in 
this paper are different from those in~\cite{msv2}. 
 
\begin{defn}
\label{defn:phi}
Let 
\[
\phi\colon S_q(n,n)\to \mathrm{End}_{\mathbb{Q}(q)}\left(W_{(3^k)}\right)
\] 
be the homomorphism of $\mathbb{Q}(q)$-algebras defined by gluing 
the following webs on top of the elements in $W_{(3^k)}$.
\begin{align*}
1_{\lambda}&\mapsto
\;\xy
(0,0)*{\includegraphics[width=60px]{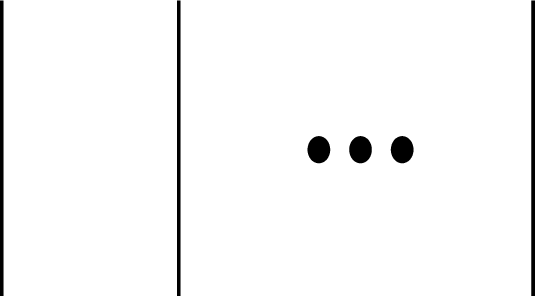}};
(-10,-8)*{\scriptstyle\lambda_1};
(-3,-8)*{\scriptstyle\lambda_{2}};
(10,-8)*{\scriptstyle\lambda_n};
\endxy
\\[0.5ex]
E_{\pm i}1_{\lambda}&\mapsto
\;\xy
(-0,0)*{\includegraphics[width=150px]{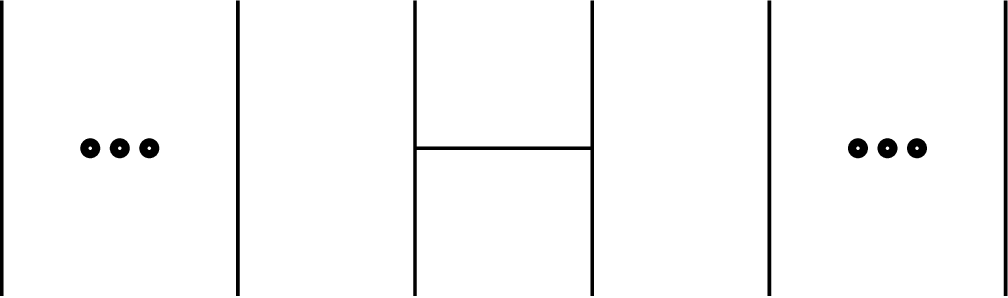}};
(-27,-10)*{\scriptstyle\lambda_1};
(-14,-10)*{\scriptstyle\lambda_{i-1}};
(-5,-10)*{\scriptstyle\lambda_{i}};
(5,-10)*{\scriptstyle\lambda_{i+1}};
(-5,10)*{\scriptstyle\lambda_{i}\pm 1};
(5,10)*{\scriptstyle\lambda_{i+1}\mp 1};
(14,-10)*{\scriptstyle\lambda_{i+2}};
(27,-10)*{\scriptstyle\lambda_n};
\endxy
\end{align*}
We use the convention that vertical edges labeled 1 are oriented upwards, 
vertical edges labeled 2 are oriented downwards and edges labeled 0 or 3 
are erased.
\vspace*{0.25cm}

The orientation of the horizontal edges is uniquely determined by 
the orientation of the vertical edges. With these conventions, one can 
check that the horizontal edge is always oriented from right 
to left for $E_{+i}$ and from left to right for $E_{-i}$.  
\vspace*{0.25cm}

Furthermore, let $\lambda\in\Lambda(n,n)$ and let $S$ be any 
sign string such that $\mu_S\in\Lambda(n,n)_3$. For any $w\in W^S$, 
we define 
\[
\phi(1_{\lambda})w=0,\quad\text{if}\quad\mu_S\ne\lambda.
\]
By 
$\phi(1_{\lambda})w$ we mean the left action of $\phi(1_{\lambda})$ on $w$.
 
In particular, for any $\lambda>(3^k)$, we have $\phi(1_{\lambda})=0$ in 
$\mathrm{End}_{\mathbb{Q}(q)}\left(W_{(3^k)}\right)$.  
\end{defn}
\vspace*{0.15cm}

Let us give two examples to show how these conventions work. We only write 
down the relevant entries of the weights and only draw the important edges. 
We have 
\begin{align*}
E_{+1}1_{(22)}&\mapsto
\;\xy
(-0,0)*{\includegraphics[width=30px]{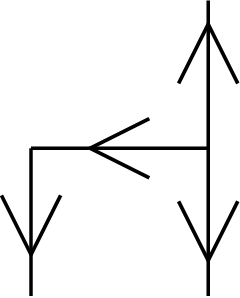}};
(-4,-8)*{\scriptstyle 2};
(4,-8)*{\scriptstyle 2};
(-4,8)*{\scriptstyle 3};
(4,8)*{\scriptstyle 1};
\endxy
\\[0.5ex]
E_{-2}E_{+1}1_{(121)}&\mapsto
\;\xy
(-0,0)*{\includegraphics[width=60px]{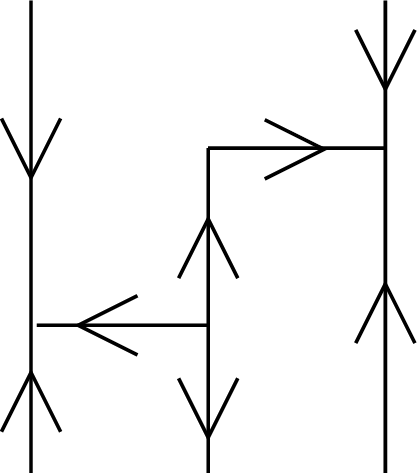}};
(-9,-14)*{\scriptstyle 1};
(0,-14)*{\scriptstyle 2};
(9,-14)*{\scriptstyle 1};
(-9,14)*{\scriptstyle 2};
(0,14)*{\scriptstyle 0};
(9,14)*{\scriptstyle 2};
\endxy
\end{align*}
\vspace*{0.15cm}

\begin{rem} Note that the introduction of enhanced sign strings is necessary 
for the definition of $\phi$ to make sense. Although as a vector space 
$W^S$ does not depend on the entries of $S$ which are equal to 
$\circ$ or $\times$, the $S_q(n,n)$ action on $W^S$ does depend on them. 
\end{rem}

\begin{lem}
\label{lem:phi} 
The map $\phi$ in Definition~\ref{defn:phi} is well-defined. 
\end{lem}
\begin{proof}
It follows immediately from its definition that $\phi$ preserves the three 
relations~\eqref{eq:schur1}, \eqref{eq:schur2} and \eqref{eq:schur3}.  

Checking case by case, one can easily show that $\phi$ 
preserves~\eqref{eq:schur4} by using the relations~\eqref{eq:circle}, 
\eqref{eq:digon} and \eqref{eq:square}. We do just one example and 
leave the other cases to the reader. The figure 
below shows the image of the relation 
$$
E_1E_{-1}1_{(21)}-E_{-1}E_11_{(21)}=1_{(21)}
$$
under $\phi$.   
\begin{align*}
\xy
(0,1)*{\includegraphics[width=35px]{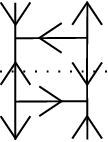}};
(-4.5,-9)*{\scriptstyle 2};
(4,-9)*{\scriptstyle 1};
(-4.5,11)*{\scriptstyle 2};
(4,11)*{\scriptstyle 1};
\endxy
\;-\;
\xy
(0,0)*{\includegraphics[width=35px]{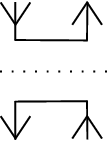}};
(-4.5,-10)*{\scriptstyle 2};
(4,-10)*{\scriptstyle 1};
(-4.5,10)*{\scriptstyle 2};
(4,10)*{\scriptstyle 1};
\endxy
\;=\;
\xy
(0,0)*{\includegraphics[width=35px]{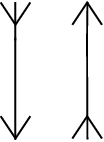}};
(-4.5,-10)*{\scriptstyle 2};
(4,-10)*{\scriptstyle 1};
(-4.5,10)*{\scriptstyle 2};
(4,10)*{\scriptstyle 1};
\endxy
\end{align*}
This relation is exactly the third Kuperberg relation in~\eqref{eq:square}.
\end{proof}

\begin{lem} The map $\phi$ gives rise to an isomorphism 
\label{lem:isoirrep}
\[
\phi\colon V_{(3^k)}\to W_{(3^k)}
\]
of $S_q(n,n)$-modules.  
\end{lem}
\begin{proof}
Note that the empty web $w_h=w_{(3^k)}$, which generates 
$W_{(\times^k,0^{2k})}\cong \mathbb{Q}(q)$, is a highest weight vector. 

The map $\phi$ induces a homomorphism of 
$S_q(n,n)$-modules 
\[
\phi\colon S_q(n,n)1_{(3^k)}\to W_{(3^k)},
\]
defined by 
\[
\phi(x1_{(3^k)})=\phi(x)w_h.
\]

It is well-known that 
\[
V_{(3^k)}\cong 
S_q(n,n)1_{(3^k)}/(\mu > (3^k)),
\]
where $(\mu > (3^k))$ is the ideal generated by all 
elements of the form $x1_{\mu}y1_{(3^k)}$ such that we have $x,y\in S_q(n,n)$ 
and $\mu$ is some weight greater than $(3^k)$. 
Since we are using $\mathfrak{sl}_3$-webs, the kernel of 
$\phi$ contains $(\mu > (3^k))$, so $\phi$ descends to $V_{(3^k)}$. 
Since $V_{(3^k)}$ is irreducible and $\phi$ is clearly non-zero, the map 
$\phi\colon V_{(3^k)}\to W_{(3^k)}$ is injective. 

As we already remarked above, we have
\begin{align*}
\dim V_{(3^k)}&=\sum_{\mu_S\in\Lambda(n,n)_3} 
\#\mathrm{Std}^{(3^k)}_{\mu_S}\\
&=\sum_{\mu_S\in\Lambda(n,n)_3} \dim W^S
=\dim W_{(3^k)}.
\end{align*} 
Therefore, we have 
\[
V_{(3^k)}\cong \phi\left(S_q(n,n)\right)w_h\cong W_{(3^k)}, 
\]
which finishes the proof. 
\end{proof}
\vskip0.2cm
We want to explain two more facts about the 
isomorphism in Lemma~\ref{lem:isoirrep}, which we will need later.  

Recall that there is an inner product on $V_{(3^k)}$. 
First of all, there is a $q$-antilinear involution ({\em bar involution}) 
on 
$\mathbb{Q}(q)$ determined by 
\[
\overline{f(q)}=f(q^{-1}),
\]
for any $f(q)\in\bQ(q)$. Recall Lusztig's $q$-antilinear 
algebra anti-automorphism $\tau$ of $S_q(n,n)$ defined by 
\[
\tau(1_{\lambda})=1_{\lambda},\;\;
\tau(1_{\lambda+\alpha_i}E_i1_{\lambda})= 
q^{-1-\overline{\lambda}_i}1_{\lambda}E_{-i}1_{\lambda+\alpha_i},\;\;
\tau(1_{\lambda}E_{-i}1_{\lambda+\alpha_i})= q^{1+\overline{\lambda}_i}
1_{\lambda+\alpha_i}E_i1_{\lambda}.
\] 
The $q$-\textit{Shapovalov form} $\langle\;\cdot\;,\;\cdot\;\rangle_{\mathrm{Shap}}$ on 
$V_{(3^k)}$ is the unique $q$-sesquilinear form such that 
\begin{itemize}
\item $\langle v_h,v_h\rangle_{\mathrm{Shap}} =1$, for a fixed highest weight vector $v_h$.
\item $\langle x v, v^{\prime}\rangle_{\mathrm{Shap}}=\langle v,\tau(x) v^{\prime}\rangle_{\mathrm{Shap}}$, for any 
$x\in S_q(n,n)$ and any $v,v^{\prime}\in V_{(3^k)}$.
\item $\langle f(q) v, g(q) v^{\prime}\rangle_{\mathrm{Shap}}=
\overline{f(q)} g(q) \langle v,v^{\prime}\rangle_{\mathrm{Shap}}$, 
for any any $v,v^{\prime}\in V_{(3^k)}$ and $f(q),g(q)\in \mathbb{Q}(q)$.
\end{itemize}  
\vskip0.5cm
We can also define an inner product on $W_{(3^k)}$, using the 
Kuperberg bracket. Let $S$ be any enhanced sign string $S$, 
such that $\mu_S\in\Lambda(n,n)_3$. Denote the length of the sign string 
$\hat{S}$ by $\ell(\hat{S})$.  
\begin{defn}
\label{defn:normkuperform}
Define the $q$-sesquilinear 
\textit{normalized Kuperberg form} by 
\begin{itemize}
\item $\langle w_h,w_h\rangle_{\mathrm{Kup}}=1$, for a fixed highest weight vector $w_h$.
\item $\langle u,v\rangle_{\mathrm{Kup}}=q^{\ell(\hat{S})}\langle u^*v\rangle_{\mathrm{Kup}}$, for any $u,v\in B^S$.
\item $\langle f(q)u,g(q)v\rangle_{\mathrm{Kup}}=\overline{f(q)}g(q)\langle u,v\rangle_{\mathrm{Kup}}$, for any $u,v\in B^S$ and $f(q),g(q)\in \mathbb{Q}(q)$. 
\end{itemize}
\end{defn}
The following lemma motivates the normalization of the Kuperberg form. 
\begin{lem}
\label{lem:phiisometry}
The isomorphism of $S_q(n,n)$-modules 
\[
\phi\colon V_{(3^k)}\to W_{(3^k)}
\]
is an isometry. 
\end{lem}
\begin{proof}
First note that  
\[
\langle (E_{\pm i}u)^*v\rangle_{\mathrm{Kup}}=
\langle u^*E_{\mp i}v\rangle_{\mathrm{Kup}},
\]
for any $u,v\in W^S$ and any $i=1,\ldots, n$, which is exactly (2) from above. 
This shows that the result of the lemma holds up to normalization. 

Our normalization of the Kuperberg form matches the 
normalization of the $q$-Shapovalov form. One can easily check this 
case by case. Let us just do two examples. Let $i=1$. Then one has
$E_11_{(a,b,\ldots)}=1_{(a+1,b-1,\ldots)}E_1$. If 
$(a,b,\ldots)\in\Lambda(n,n)_3$ such that $a-b=-1$, 
then 
\[
\ell(\widehat{(a,b)})=\ell(\widehat{(a+1,b-1)}),
\]
where $\ell$ indicates the length of the sign sequence. 
This matches 
\[
\tau(E_11_{(a,b)})=1_{(a,b)}E_{-1}.
\]  
If $(a,b)=(2,1)$, then $E_11_{(2,1,\ldots)}=1_{(3,0,\ldots)}E_1$. Note 
that 
\[
\ell(\widehat{(2,1,\ldots)})=\ell(\widehat{(3,0,\ldots)})+2.
\]
This $+2$ 
cancels exactly with the $-2$, which appears as the exponent of 
$q$ in 
\[
\tau(E_11_{(2,1,\ldots)})=q^{-2}1_{(2,1,\ldots)}E_{-1}.
\]    
\end{proof}

We will need one more fact about $\phi$. For any $i=1,\ldots,n$ 
and any $a\in\mathbb{N}$, let  
\[
E_{\pm i}^{(a)}=\dfrac{E_{\pm i}^a}{[a]!}
\]
denote the \textit{divided power} in $S_q(n,n)$. 
Recall the following relations for the divided powers.
\begin{eqnarray}
\label{eq:divpow1}
E_{\pm i}^{(a)}E_{\pm i}^{(b)}1_{\lambda}&=&\qbin{a+b}{a}E_{\pm i}^{(a+b)}1_{\lambda},\\
\label{eq:divpow2}
E_{+i}^{(a)}E_{-i}^{(b)}1_{\lambda}&=&
\sum_{j=0}^{\min(a,b)}\qbin{a-b+\lambda_i-\lambda_{i+1}}{j}
E_{-i}^{(b-j)}E_{+i}^{(a-j)}1_{\lambda},\\
\label{eq:divpow3}
E_{-i}^{(b)}E_{+i}^{(a)}1_{\lambda}&=&
\sum_{j=0}^{\min(a,b)}\qbin{b-a-(\lambda_i-\lambda_{i+1})}{j}
E_{+i}^{(a-j)}E_{-i}^{(b-j)}1_{\lambda}.
\end{eqnarray}
Here $[a]!$ denotes the \textit{quantum factorial} and $\qbin{a}{b}$ 
denotes the \textit{quantum binomial}.

The images of the divided powers under 
\[
\phi\colon S_q(n,n)\to \mathrm{End}(W_{(3^k)})
\]
are easy to compute. For example, we have (for simplicity, we only 
draw two of the strands and write $E=E_{+i}$)
\[
\phi(E^21_{(0,2)})=
\;\xy
(0,0)*{\includegraphics[width=30px]{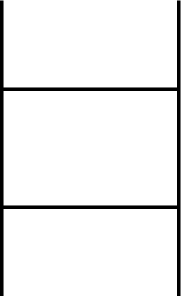}};
(-6,-10)*{\scriptstyle 0};
(6,-10)*{\scriptstyle 2};
(-7,0)*{\scriptstyle 1};
(7,0)*{\scriptstyle 1};
(-6,10)*{\scriptstyle 2};
(6,10)*{\scriptstyle 0};
\endxy 
\;
=
\;\xy
(0,0)*{\includegraphics[width=40px]{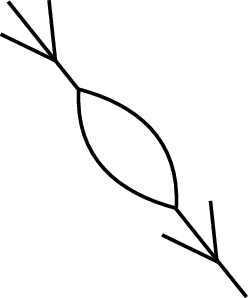}};
(6,10)*{\scriptstyle \circ};
(-6,10)*{\scriptstyle -};
(-6,-10)*{\scriptstyle \circ};
(6,-10)*{\scriptstyle -};
\endxy 
=
\;
[2]
\;\xy
(0,0)*{\includegraphics[width=40px]{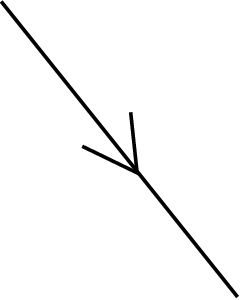}};
(6,10)*{\scriptstyle \circ};
(-6,10)*{\scriptstyle -};
(-6,-10)*{\scriptstyle \circ};
(6,-10)*{\scriptstyle -};
\endxy. 
\]
Therefore, we get 
\[
\phi(E^{(2)}1_{(0,2)})=
\;\xy
(0,0)*{\includegraphics[width=40px]{section51/DivEE02}};
(6,10)*{\scriptstyle \circ};
(-6,10)*{\scriptstyle -};
(-6,-10)*{\scriptstyle \circ};
(6,-10)*{\scriptstyle -};
\endxy. 
\]
Another interesting example is 
\[
\phi(E^{2}1_{(0,3)})=
\;\xy
(-0,0)*{\includegraphics[width=30px]{section51/HHweb}};
(-6,-10)*{\scriptstyle 0};
(6,-10)*{\scriptstyle 3};
(-7,0)*{\scriptstyle 1};
(7,0)*{\scriptstyle 2};
(-6,10)*{\scriptstyle 2};
(6,10)*{\scriptstyle 1};
\endxy
\;
=
\;
[2]
\;\xy
(0,0)*{\includegraphics[width=40px]{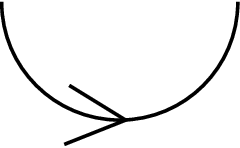}};
(6,10)*{\scriptstyle +};
(-6,10)*{\scriptstyle -};
(-6,-10)*{\scriptstyle \circ};
(6,-10)*{\scriptstyle \times};
\endxy, 
\]
which shows that  
\[
\phi(E^{(2)}1_{(03)})=
\;
\xy
(0,0)*{\includegraphics[width=40px]{section51/rightcup}};
(6,10)*{\scriptstyle +};
(-6,10)*{\scriptstyle -};
(-6,-10)*{\scriptstyle \circ};
(6,-10)*{\scriptstyle \times};
\endxy. 
\]
The final example we will consider is $\phi(E^{(3)}1_{(0,3)})$.
We see that 
\[
\phi(E^{3}1_{(0,3)})=
\;\xy
(-0,0)*{\includegraphics[width=30px]{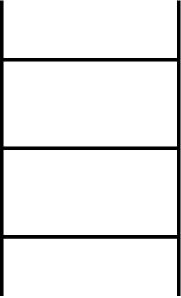}};
(-6,-10)*{\scriptstyle 0};
(6,-10)*{\scriptstyle 3};
(-7,-3)*{\scriptstyle 1};
(7,-3)*{\scriptstyle 2};
(-7,3)*{\scriptstyle 2};
(7,3)*{\scriptstyle 1};
(-6,10)*{\scriptstyle 3};
(6,10)*{\scriptstyle 0};
\endxy
\;
=
\;
\xy
(0,0)*{\includegraphics[width=40px]{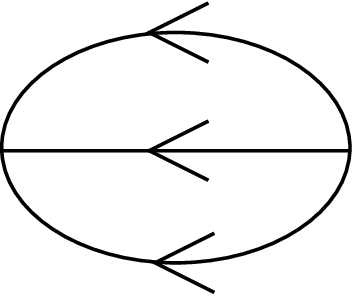}};
(-6,10)*{\scriptstyle \times};
(6,10)*{\scriptstyle \circ};
(-6,-10)*{\scriptstyle \circ};
(6,-10)*{\scriptstyle \times};
\endxy
\;
=
\;
[3]!
\xy
(-6,10)*{\scriptstyle \times};
(6,10)*{\scriptstyle \circ};
(-6,-10)*{\scriptstyle \circ};
(6,-10)*{\scriptstyle \times};
\endxy.
\]
Thus, we have 
\[
\phi(E^{(3)}1_{(0,3)})=\;\xy
(-6,10)*{\scriptstyle \times};
(6,10)*{\scriptstyle \circ};
(-6,-10)*{\scriptstyle \circ};
(6,-10)*{\scriptstyle \times};
\endxy,
\]
which is the unique empty web from $(\circ,\times)$ to 
$(\times,\circ)$.

Note that~\eqref{eq:divpow2} and~\eqref{eq:divpow3} imply that, 
for any $a\in\mathbb{N}$, we have 
\begin{equation}
\label{eq:adjust1}
E_{-i}^{(a)}E_{+i}^{(a)}1_{(\ldots,0,a,\ldots)}=1_{(\ldots,0,a,\ldots)}
\quad\text{and}\quad 
E_{+i}^{(a)}E_{-i}^{(a)}1_{(\ldots,a,0,\ldots)}=1_{(\ldots,a,0,\ldots)}
\end{equation}
in $S_q(n,n)$. Similarly, 
let $S_q(n,n)/I$, where $I$ denotes the two-sided ideal generated by 
all $1_{\mu}$ such that $\mu>(3^k)$. Again by~\eqref{eq:divpow2} 
and~\eqref{eq:divpow3}, we have
\begin{equation}
\label{eq:adjust2}
E_{-i}^{(3-a)}E_{+i}^{(3-a)}1_{(\ldots,a,3,\ldots)}=1_{(\ldots,a,3,\ldots)}
\quad\text{and}\quad 
E_{+i}^{(3-a)}E_{-i}^{(3-a)}1_{(\ldots,3,a,\ldots)}=1_{(\ldots,3,a,\ldots)}
\end{equation}
in $S_q(n,n)/I$.
One can check that $\phi$ maps the two sides of the equations 
in~\eqref{eq:adjust1} and~\eqref{eq:adjust2} to isotopic diagrams. 
For example, $\phi$ maps 
\[
E_{-}^{(2)}E_{+}^{(2)}1_{(0,2)}=1_{(0,2)}
\]
to 
\[
\xy
(0,0)*{\includegraphics[width=22px]{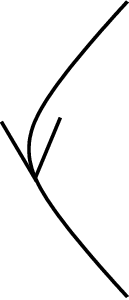}};
(-6,12)*{\scriptstyle \circ};
(6,12)*{\scriptstyle -};
(-6,0)*{\scriptstyle -};
(6,0)*{\scriptstyle \circ};
(-6,-12)*{\scriptstyle \circ};
(6,-12)*{\scriptstyle -};
\endxy
\;
=
\;
\xy
(6,0)*{\includegraphics[width=11.3px]{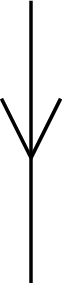}};
(-6,12)*{\scriptstyle \circ};
(6,12)*{\scriptstyle -};
(-6,-12)*{\scriptstyle \circ};
(6,-12)*{\scriptstyle -};
\endxy.
\]

The proof of the following lemma is based on an algorithm, which we call the 
\textit{enhanced inverse growth algorithm}. The result is needed later to 
show essential surjectivity in Theorem~\ref{thm:equivalence}. 
\begin{lem}
\label{lem:phisurj}
Let $S$ be any enhanced sign string such that $\mu_S\in\Lambda(n,n)_3$. 
For any $w\in B^S$, there exists a product of divided powers $x$, such that
\[
\phi(x1_{(3^k)})=w.
\] 
\end{lem}
\begin{proof}
Choose any $w\in B^S$. We consider $w\in B_{(\times^k,\circ^{2k})}^S$, i.e. 
a non-elliptic web with (empty) lower boundary determined by 
$(\times^k,\circ^{2k})$ and upper boundary determined by $S$. 
Express $w$ using the growth algorithm, 
in an arbitrary way. Suppose there are $m$ steps in this instance of the 
growth algorithm. The element $x$ is built up in $m+2$ steps, i.e. 
an initial step, one step for each step in the growth algorithm, and a last 
step. During the construction of $x$, we always keep track of 
the $\circ$s and $\times$s. At each step 
the strands of $w$ are numbered according to their position in $x$. 

If the $H$, $Y$ or arc-move is applied to two non-consecutive strands, 
we first have to apply some divided powers, 
as in~\eqref{eq:adjust1} and~\eqref{eq:adjust2}, 
to make them consecutive. Let $x_k\in S_q(n,n)$ be the element 
assigned to the $k$-th step and let $\mu^k$ be the weight after the 
$k$-step, i.e. $x_k=1_{\mu^{k-1}}x_k1_{\mu^k}$. 
The element $x$ we are looking for is the product of all $x_k$. 
\begin{enumerate}
\item Take $x_0=1_{\mu_S}$.
\item Suppose that the $k$-th step in the growth algorithm is applied to 
the strands $i$ and $i+r$, for some $r\in\mathbb{N}_{>0}$. This means that 
the entries of $\mu^{k-1}$ satisfy $\mu_j\in\{0,3\}$, 
for all $j=i+1,\ldots,i+r-1$. Let $x_k^{\prime}$ be the product of divided powers 
which ``swap'' the $(\mu_{i+1},\ldots,\mu_{i+r-1})$ and $\mu_{i+r}$. So, we 
first swap $\mu_{i+r-1}$ and $\mu_{i+r}$, then $\mu_{i+r-2}$ and $\mu_{i+r}$ 
etc. Now, the rule in the growth algorithm, still corresponding to the 
$k$-th step, can be applied to the strands $i$ and $i+1$. 
\item Suppose that it is an $H$-rule. If the bottom of the $H$ is a pair 
(up-arrow down-arrow), then take $x_k=x_k^{\prime}E_{+i}$. If the bottom of the 
$H$ is a pair (down-arrow up-arrow), then take $x_k=x_k^{\prime}E_{-i}$. 
\item Suppose that the rule, corresponding to the $k$-th step in 
the growth algorithm, is a $Y$-rule. If the bottom strand of $Y$ is oriented 
downward, then take $x_k=x_k^{\prime}E_{-i}$. If it is oriented upward, take 
$x_k=x_k^{\prime}E_{+i}$. Note that these two choices are not unique. They depend on 
where you put $0$ or $3$ in $\mu^k$. The choice we made corresponds to taking 
$(\mu^k_i,\mu^k_{i+1})=(2,0)$ in the first case and 
$(\mu^k_i,\mu^k_{i+1})=(1,3)$ in the second case. Other choices would be 
perfectly fine and would lead to equivalent elements in $S_q(n,n)1_{(3^k)}/
(\mu > (3^k))$. 
\item Suppose that the rule, corresponding to the $k$-th step in 
the growth algorithm, is an arc-rule. If the arc is oriented clockwise, 
take $x_k=x_k^{\prime}E_{-i}^{(2)}$. If the arc is oriented counter-clockwise, 
take $x_k=x_k^{\prime}E_{-i}$. Again, these choices are not unique. They correspond 
to taking $(\mu^k_i,\mu^k_{i+1})=(3,0)$ in both cases. 
\item After the $m$-th step in the growth algorithm, which is the last one, 
we obtain $\mu^m$, which is a sequence of $3$'s and $0$'s. Let $x_{m+1}$ be 
the product of divided powers which reorders the entries of $\mu^m$, so that 
$\mu^{m+1}=(3^k)$. 
\item Take $x=1_{\mu_S}x_1x_2\cdots x_{m+1}1_{(3^k)}\in S_q(n,n)$. Note that 
$x$ is of the form $E_{\ii}1_{(3^k)}$.      
\end{enumerate}
From the analysis of the images of the divided powers under $\phi$, it is 
clear that 
\[
\phi(x)=w.
\] 
\end{proof}

We do a simple example to illustrate Lemma~\ref{lem:phisurj}.
Let 
\[
w=
\;
\xy
(0,0)*{\includegraphics[width=40px]{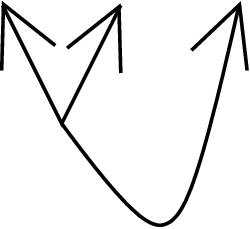}};
(-6,10)*{\scriptstyle 1};
(0,10)*{\scriptstyle 1};
(6,10)*{\scriptstyle 1};
\endxy
\] 
Then the algorithm in the proof of Lemma~\ref{lem:phisurj} gives 
\[
x=1_{(111)}E_{-1}E_{-2}E_{-1}1_{(300)},
\]
or as a picture (read from bottom to top)
\[
\xy
(0,0)*{\includegraphics[width=50px]{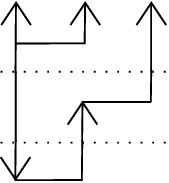}};
(-9,12)*{\scriptstyle 1};
(0.3,12)*{\scriptstyle 1};
(9,12)*{\scriptstyle 1};
(-15,5)*{\scriptstyle{E_{-1}}};
(-9,2)*{\scriptstyle 2};
(0.3,2)*{\scriptstyle 0};
(9,2)*{\scriptstyle 1};
(-15,-1.5)*{\scriptstyle{E_{-2}}};
(-9,-5.5)*{\scriptstyle 2};
(0.3,-5.5)*{\scriptstyle 1};
(9,-5.5)*{\scriptstyle 0};
(-15,-9)*{\scriptstyle{E_{-1}}};
(-9,-12)*{\scriptstyle 3};
(0.3,-12)*{\scriptstyle 0};
(9,-12)*{\scriptstyle 0};
\endxy.
\]

We are now ready to start explaining the categorified story. 

\subsection{And its categorification}\label{sec-webhoweb}

Let $\mathcal{W}^S=K^S\text{-}\mathrm{p\textbf{Mod}}_{\mathrm{gr}}$ 
be the category of all finite dimensional projective 
graded $K^S$-modules. 
In what follows, it will sometimes be useful to consider homomorphisms 
of arbitrary degree, so we define
\[
\mathrm{HOM}_{B}(M,N)=\bigoplus_{t\in\mathbb{Z}}\text{Hom}_B(M\{t\},N),
\]
for any finite dimensional associative unital graded algebra $B$ and 
any finite dimensional graded $B$-modules $M$ and $N$. Note that for 
almost all 
$t\in\mathbb{Z}$ we have $\text{Hom}_B(M\{t\},N)$, so 
$\text{HOM}_B(M,N)$ is still finite dimensional. We also define 
$$\dim_q \mathrm{HOM}_{B}(M,N)=\sum_{t\in\bZ} q^t\dim \text{Hom}_B(M\{t\},N).$$

Moreover, we need the following notions throughout the rest of the section.

Suppose that $S$ is an enhanced sign string 
such that $\mu_S\in\Lambda(n,n)_3$. For any $u\in B^S$, let 
\[
P_u=\bigoplus_{w\in B^S} {}_uK_w.
\]
Then we have   
\[
K^S=\bigoplus_{u\in B^S} P_u,
\]
and so $P_u$ is an object in $\mathcal{W}^S$, 
for any $u\in B^S$. Note that, for any $u,v\in B^S$, we have  
\[
\mathrm{HOM}(P_u,P_v)\cong {}_uK_v,
\]
where an element in ${}_uK_v$ acts on $P_u$ by composition on the 
right-hand side. 

Similarly, we can define 
\[
{}_uP=\bigoplus_{w\in B^S} {}_uK_w,
\]
which is a right graded projective $K^S$-module. 

\begin{rem}
Just one warning: The reader should not confuse 
$P_u$ with $P_{u,T}$ in Section~\ref{sec:webalgebra}.
\end{rem} 
\subsubsection{The definition of ${\mathcal W}_{(3^k)}$}
Recall that $S$ denotes an enhanced sign string. Define  
\[
K_{(3^k)}=\bigoplus_{\mu_S\in\Lambda(n,n)_3} K^S
\]
and
\[
{\mathcal W}_{(3^k)}=K_{(3^k)}\text{-}\mathrm{p\textbf{Mod}}_{\mathrm{gr}}\cong 
\bigoplus_{\mu_S\in\Lambda(n,n)_3}\mathcal{W}^S.
\]
The main goal of this section is to show that ${\mathcal W}_{(3^k)}$ is 
a strong $\mathfrak{sl}_n$-$2$-representation and 
that 
\[
{\mathcal W}_{(3^k)}\cong {\mathcal V}_{(3^k)}
\]
as strong-$\mathfrak{sl}_n$ 2-representations. 

This will imply that 
\[
K^{\oplus}_0({\mathcal W}_{(3^k)})_{\bQ(q)}\cong V_{(3^k)}.
\]
Note that 
\[
K^{\oplus}_0({\mathcal W}_{(3^k)})_{\bQ(q)}\cong \bigoplus_{\mu_S\in\Lambda(n,n)_3}K^{\oplus}_0(\mathcal{W}^S)_{\bQ(q)}.
\]
We will show that this corresponds exactly to 
the $U_q(\mathfrak{gl}_n)$-weight space decomposition of $V_{(3^k)}$. 
In particular, this will show that 
\begin{equation}
\label{eq:fundequality}
K^{\oplus}_0(\mathcal{W}^S)_{\bQ(q)}\cong W^S,
\end{equation}
for any enhanced sign sequence $S$ such that $\mu_S\in \Lambda(n,n)_3$. 
\vskip0.5cm
First, we have to recall the definitions of 
\textit{sweet} bimodules.
\subsubsection{Sweet bimodules}
Note that the following definitions and results are 
the $\mathfrak{sl}_3$ analogues of those in Section 2.7 in~\cite{kh}.

\begin{defn} Given rings $R_1$ and $R_2$, a $(R_1,R_2)$-bimodule $N$ is 
called \textit{sweet} if it is finitely generated and projective as a 
left $R_1$-module and as a right $R_2$-module. 
\end{defn}
If $N$ is a sweet $(R_1,R_2)$-bimodule, then the functor 
\[
N\otimes_{R_2}- \colon 
R_2\text{-}\mathrm{\textbf{Mod}}\to R_1\text{-}\mathrm{\textbf{Mod}}
\]
is exact and sends projective modules to projective modules. Given a 
sweet $(R_1,R_2)$-bimodule $M$ and a sweet $(R_2,R_3)$-bimodule $N$, then the 
tensor product $M\otimes_{R_2}N$ is a sweet $(R_1,R_3)$-bimodule.

Let $S$ and $S^{\prime}$ be two enhanced sign strings. Then $\widehat{B}_S^{S^{\prime}}$ 
denotes the set of all monomial webs whose boundary is divided 
into a lower part, determined by $S$, and an upper part, determined by $S^{\prime}$. 
By a {\em monomial} web we mean a web given by one diagram. 
Let $B_S^{S^{\prime}}\subset \widehat{B}_S^{S^{\prime}}$ be 
the subset of non-elliptic webs. 

For any $w\in \widehat{B}_S^{S^{\prime}}$, define a graded finite dimensional 
$(K^{S^{\prime}},K^{S})$-bimodule $\Gamma(w)$ by 
\[
\Gamma(w)=\bigoplus_{u\in B^{S^{\prime}},v\in B^{S}}{}_u\Gamma(w)_v,
\]
with 
\[
{}_u\Gamma(w)_v={\mathcal F}^c(u^*wv)\{\ell(\hat{S})\}.
\] 
The left and right actions of $K^S$ on 
$\Gamma(w)$ are defined by applying the multiplication foam in~\ref{multfoam} 
to 
\[
{}_rK_u\otimes{}_u\Gamma(w)_v\to {}_r\Gamma(w)_v\quad\text{and}\quad 
{}_u\Gamma(w)_v\otimes {}_vK_r\to {}_u\Gamma(w)_r.
\]
Let $w\in\widehat{B}_S^{S^{\prime}}$. Then $w=c_1w_1+\cdots +c_tw_t$, for 
certain $w_i\in B_S^{S^{\prime}}$ and $c_i\in \mathbb{N}[q,q^{-1}]$. Since all 
relations which are satisfied by the Kuperberg bracket have categorical 
analogues for foams, this shows that 
\[
\Gamma(w)\cong \Gamma(w_1)^{\oplus c_1}\oplus \cdots\oplus \Gamma(w_t)^{\oplus c_t}.
\]

We have the following analogue of Proposition 3 in~\cite{kh}.
\begin{prop}\label{prop-sweet}
For any $w\in\widehat{B}_S^{S^{\prime}}$, the graded $(K^{S^{\prime}},K^{S})$-bimodule 
$\Gamma(w)$ is sweet. 
\end{prop} 
\begin{proof}
As a left $K^S$-module, we have 
\[
\Gamma(w)\cong \bigoplus_{v\in B^{S}} \Gamma(w)_v,
\]
where 
\[
\Gamma(w)_v=\bigoplus_{u\in B^{S^{\prime}}}{}_u\Gamma(w)_v.
\]
So, as far as the left action is concerned, it is sufficient to show that 
$\Gamma(w)_v$ is a left projective $K^{S^{\prime}}$-module. Note that, as a left 
$K^{S^{\prime}}$-module, we have  
\[
\Gamma(w)_v\cong \bigoplus_{u\in B^{S^{\prime}}}{\mathcal F}^0(u^*wv)\{\ell(\hat{S})\}.
\] 
Then $wv=c_1u_1+\cdots+c_tu_t$, for certain 
$u_i\in B^{S^{\prime}}$ and $c_i\in\mathbb{N}[q,q^{-1}]$. By the remarks above, this 
means that 
\begin{equation}
\label{eq:projbimodule}
\Gamma(w)_v\cong P_{u_1}^{\oplus c_1}\{\ell(\hat{S})-\ell(\hat{S}')\}\oplus \cdots\oplus P_{u_t}^{\oplus c_t}
\{\ell(\hat{S})-\ell(\hat{S}')\}.
\end{equation}
This proves that $\Gamma(w)$ is projective as a left $K^{S^{\prime}}$-module. 

The proof that $\Gamma(w)$ is projective as a right $K^{S}$-module is 
similar.   
\end{proof}

It is not hard to see that (see for example~\cite{kh}), for any $w\in \widehat{B}_S^{S^{\prime}}$ and 
$w'\in \widehat{B}_{S^{\prime}}^{S^{\prime\prime}}$, we have 
\begin{equation}
\label{eq:sweettensor}
\Gamma(w'w)\cong \Gamma(w')\otimes_{K^{S^{\prime}}}\Gamma(w).
\end{equation}
\begin{lem} Let $w, w'\in \widehat{B}_{S}^{S^{\prime}}$. An isotopy between 
$w$ and $w'$ induces an isomorphism between $\Gamma(w)$ and $\Gamma(w')$. 
Two isotopies between $w$ and $w'$ induce the same isomorphism if and only 
if they induce the same bijection between the connected components of $w$ 
and $w'$.  
\end{lem}

\begin{lem}
Let $w,w'\in \widehat{B}_{S}^{S^{\prime}}$ and let $f\in \foamt^0(w,w')$ be a foam 
of degree $t$. Then $f$ induces a bimodule map 
\[
\Gamma(f)\colon \Gamma(w)\to \Gamma(w')
\] 
of degree $t$.
\end{lem}
\begin{proof}
Note that, for any $u\in B^{S^{\prime}}$ and $v\in B^{S}$, the foam $f$ induces a 
linear map 
\[
{\mathcal F}^0(1_{u^*}f1_v)\colon {\mathcal F}^0(u^*wv)\to \F^0(u^*w'v),
\]
by glueing $1_{u^*}f1_v$ on top of any element in 
${\mathcal F}^0(u^*wv)=\foamt^0(\emptyset, u^*wv).$ 
This map has degree $t$, e.g. the identity has degree $0$ 
because the multiplication in $K^S$ is degree preserving. 
By taking the direct sum over all 
$u\in B^{S^{\prime}}$ and $v\in B^{S}$, we get a linear map 
\[
\Gamma(f)\colon \Gamma(w)\to \Gamma(w').
\]
With the shift $\ell(\hat{S})$ in the definition of 
$\Gamma(w)$ and $\Gamma(w')$, we get exactly 
$\deg \Gamma(f)=t$.
\vskip0.5cm 
The fact that $\Gamma(f)$ is a left $K^S$-module map follows from 
the following observation. For any $u\in B^S$ and $v\in B^{S^{\prime}}$, 
the linear map ${\mathcal F}^0(1_{u^*}f1_v)$ corresponds to the linear map 
\[
\foamt^0(u,wv)\to \foamt^0(u,w'v)
\]
determined by horizontally composing with $f1_v$ on the right-hand side. 
This map clearly commutes with any composition on the left-hand side. 

Analogously, the linear map ${\mathcal F}^0(1_{u^*}f1_v)$ corresponds to 
the linear map  
\[
\foamt^0(w^*u,v)\to \foamt^0((w')^*u,v)
\]
determined by horizontally composing with $f^*1_u$ on the left-hand side. 
This map clearly commutes with any composition on the right-hand side. 

These two observations show that $\Gamma(f)$ is a $(K^{S^{\prime}},K^{S})$-bimodule map. 
\end{proof}
It is not hard to see that, for any $f\in \foamt^0(w,w')$ and 
$g\in\foamt^0(w',w'')$, we have 
\[
\Gamma(fg)=\Gamma(f)\Gamma(g).
\]
And similarly, for any $u_1,u_2\in \widehat{B}_{S}^{S^{\prime}}$ and 
$u_1',u_2'\in \widehat{B}_{S^{\prime}}^{S^{\prime\prime}}$ and for any 
$f\in\foamt^0(u_1,u_2)$ and $f'\in\foamt^0(u'_1,u'_2)$, we have 
a commuting square 
\[
\begin{CD}
\Gamma(u_1u_1')&@>{\Gamma(f\circ f')}>>&\Gamma(u_2u_2')\\
@V{\cong}VV&&@V{\cong}VV\\
\Gamma(u_1)\otimes_{K_{S^{\prime}}}\Gamma(u_1')&@>{\Gamma(f)\otimes \Gamma(f')}>>&
\Gamma(u_2)\otimes_{K_{}}\Gamma(u_2')
\end{CD}
\]
where the vertical isomorphisms are as in~\eqref{eq:sweettensor}.

\subsubsection{The strong $\mathfrak{sl}_n$-$2$-representation on 
${\mathcal W}_{(3^k)}$}
We are now going to use sweet bimodules to define a strong 
$\mathfrak{sl}_n$-$2$-representation on $\mathcal{W}_{(3^k)}$. 
It suffices to define a functorial action of 
$\Scat(n,n)$ on ${\mathcal W}_{(3^k)}$, which then gives the 
desired strong $\mathfrak{sl}_n$-$2$-representation by pulling back 
along the $2$-functor $\Psi_{n,n}\colon \Ucat\to \Scat(n,n)$. 
For convenience, we consider $\Scat(n,n)$ to be a monoidal category 
rather than a 2-category in this section.  
\begin{defn}
\label{defn:cataction}
\textbf{On objects:} The functorial action of any object 
$\mathcal{E}_{\ii}1_{\lambda}$ in $\Scat(n,n)$ on $\mathcal{W}_{(3^k)}$ 
is defined by tensoring with the sweet bimodule (see Proposition~\ref{prop-sweet})
\[
\Gamma\left(\phi\left(E_{\ii}1_{\lambda}\right)\right).
\]
Recall that $\phi\colon S_q(n,n)\to \mathrm{End}_{\mathbb{Q}(q)}(W_{(3^k)})$ 
was defined in Definition~\ref{defn:phi}.
\vskip0.5cm
\textbf{On morphisms:} We give a list of the foams associated to the 
generating morphisms of $\Scat(n,n)$. 
Applying $\Gamma$ to these foams determines 
the natural transformations associated to the morphisms of $\Scat(n,n)$. 

As before, we only draw the most important part of the foams, 
omitting partial identity foams. Our conventions are the following.
\begin{itemize}
\item[(1)]We read the regions of the morphisms in 
$\Scat(n,n)$ from right to left and the morphisms themselves from bottom to 
top.
\item[(2)]The corresponding foams we read from bottom to top and from front to 
back.
\item[(3)]Vertical front edges labeled $1$ are assumed to be oriented upward and 
vertical front edges labeled $2$ are assumed to 
be oriented downward.
\item[(4)]The convention for the orientation of the back edges is 
precisely the opposite.
\item[(5)]A facet is labeled 0 or 3 if and only if its boundary has edges 
labeled 0 or 3. 
\end{itemize}
In the list below, we always assume that $i<j$. 
Finally, all facets labeled 0 or 3 in the images below have to be erased, 
in order to get real foams. For any $\lambda>(3^k)$, the image of 
the elementary morphisms below is taken to be zero, by convention.

{\allowdisplaybreaks
\begin{align*}
\xy
(4,1.5)*{\includegraphics[width=9px]{section23/upsimpledot}};
(7,-4)*{{\scriptstyle i,\lambda}};
\endxy &\mapsto 
\;\xy
(0,0)*{\includegraphics[width=60px]{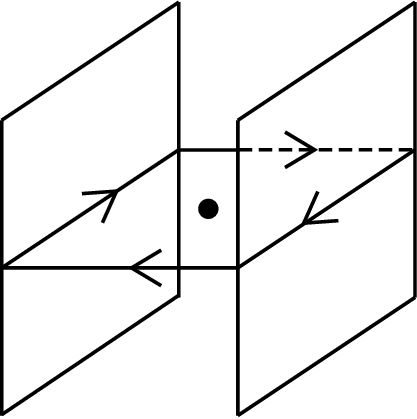}};
(-8,-12)*{\scriptstyle\lambda_i};
(6,-12)*{\scriptstyle\lambda_{i+1}};
\endxy
\\
\xy
(4,1.5)*{\includegraphics[width=9px]{section23/downsimpledot}};
(7,-4)*{{\scriptstyle i,\lambda}};
\endxy &\mapsto 
\;\xy
(0,0)*{\includegraphics[width=60px]{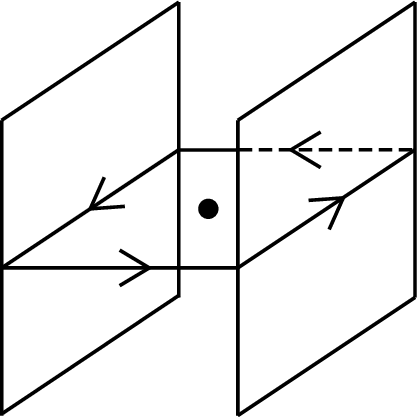}};
(-8,-12)*{\scriptstyle\lambda_i};
(6,-12)*{\scriptstyle\lambda_{i+1}};
\endxy
\\ \xy
(6,1.5)*{\includegraphics[width=20px]{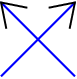}};
(12.5,-4)*{{\scriptstyle i,i,\lambda}};
\endxy  &\mapsto
\;\;-\;\;
\xy
(0,0)*{\includegraphics[width=60px]{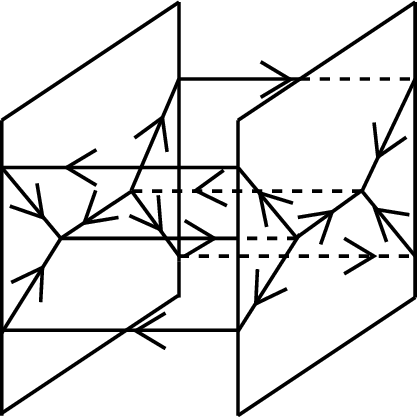}};
(-8,-12)*{\scriptstyle\lambda_i};
(6,-12)*{\scriptstyle\lambda_{i+1}};
\endxy
\\
\xy
(6,1.5)*{\includegraphics[width=20px]{section23/upcross}};
(12.5,-4)*{{\scriptstyle i,i+1,\lambda}};
\endxy &\mapsto
\;\;(-1)^{\lambda_{i+1}}\;\;
\xy
(0,0)*{\includegraphics[width=100px]{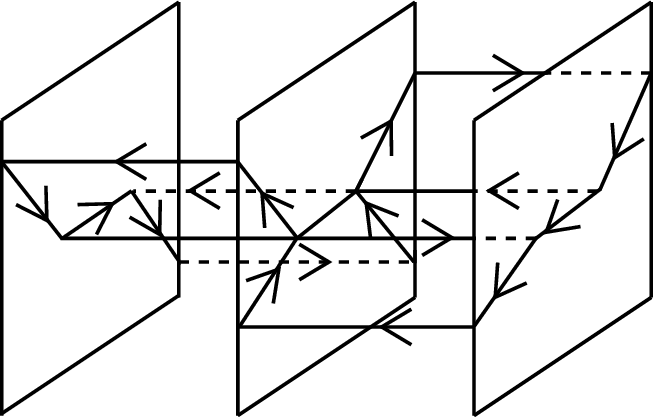}};
(-14,-12)*{\scriptstyle\lambda_i};
(0,-12)*{\scriptstyle\lambda_{i+1}};
(13,-12)*{\scriptstyle\lambda_{i+2}};
\endxy
\\
\xy
(6,1.5)*{\includegraphics[width=20px]{section23/upcross}};
(12.5,-4)*{{\scriptstyle i+1,i,\lambda}};
\endxy &\mapsto
\;
\xy
(0,0)*{\includegraphics[width=100px]{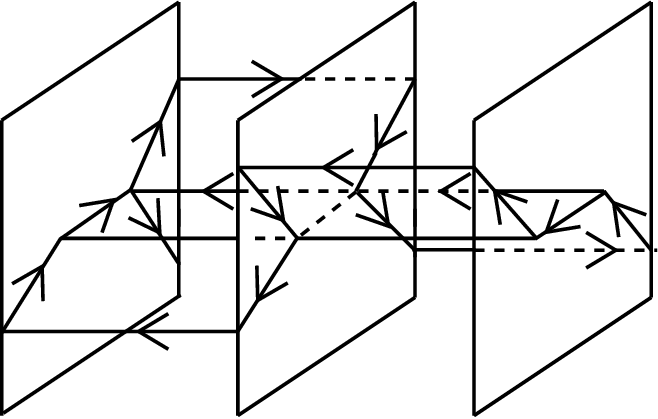}};
(-14,-12)*{\scriptstyle\lambda_i};
(0,-12)*{\scriptstyle\lambda_{i+1}};
(13,-12)*{\scriptstyle\lambda_{i+2}};
\endxy
\\
\xy
(6,1.5)*{\includegraphics[width=20px]{section23/upcross}};
(12.5,-4)*{{\scriptstyle i,j,\lambda}};
\endxy &\mapsto
\;
\xy
(0,0)*{\includegraphics[width=130px]{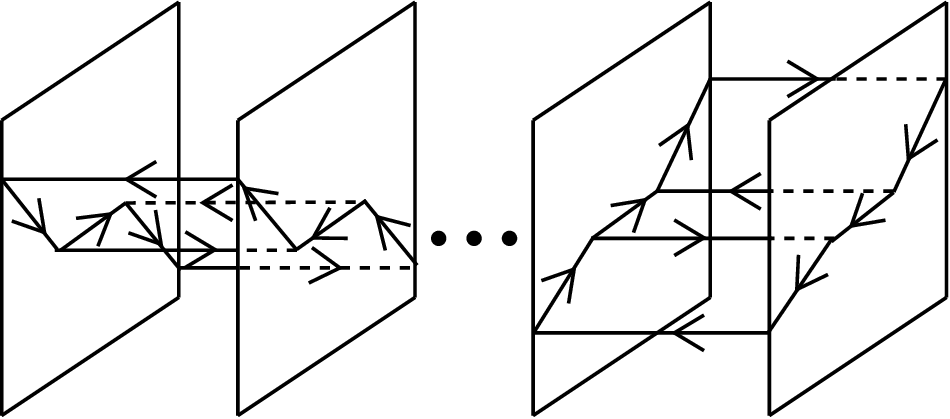}};
(-22,-12)*{\scriptstyle\lambda_i};
(-7,-12)*{\scriptstyle\lambda_{i+1}};
(4,-12)*{\scriptstyle\lambda_{j}};
(19,-12)*{\scriptstyle\lambda_{j+1}};
\endxy
\\
\xy
(6,1.5)*{\includegraphics[width=20px]{section23/upcross}};
(12.5,-4)*{{\scriptstyle j,i,\lambda}};
\endxy &\mapsto
\;
\xy
(0,0)*{\includegraphics[width=130px]{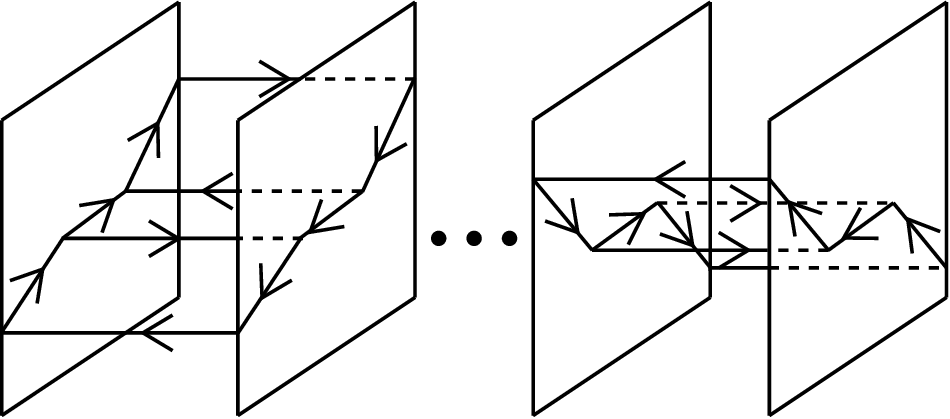}};
(-22,-12)*{\scriptstyle\lambda_i};
(-7,-12)*{\scriptstyle\lambda_{i+1}};
(4,-12)*{\scriptstyle\lambda_{j}};
(19,-12)*{\scriptstyle\lambda_{j+1}};
\endxy
\\
\xy
(8,0)*{\includegraphics[width=25px]{section23/rightcup}};
(12.5,-4)*{{\scriptstyle i,\lambda}};
\endxy &\mapsto
\;\xy
(0,0)*{\includegraphics[width=60px]{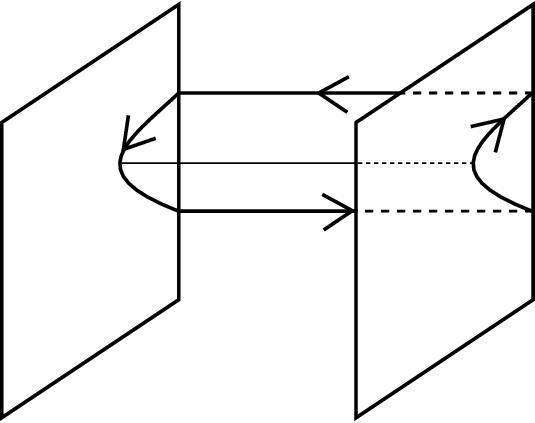}};
(-8,-10)*{\scriptstyle\lambda_i};
(8,-10)*{\scriptstyle\lambda_{i+1}};
\endxy
\\
\xy
(8,0)*{\includegraphics[width=25px]{section23/leftcup}};
(12.5,-4)*{{\scriptstyle i,\lambda}};
\endxy &\mapsto
(-1)^{\lfloor\frac{\lambda_i}{2}\rfloor+\lceil\frac{\lambda_{i+1}}{2}\rceil}\;\xy
(0,0)*{\includegraphics[width=60px]{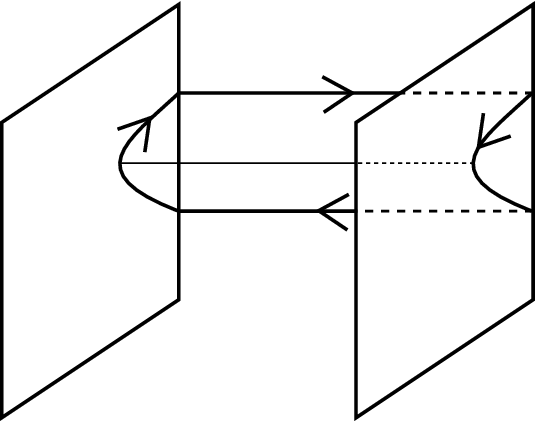}};
(-8,-10)*{\scriptstyle\lambda_i};
(8,-10)*{\scriptstyle\lambda_{i+1}};
\endxy
\end{align*}
\begin{align*}
\xy
(8,0)*{\includegraphics[width=25px]{section23/leftcap}};
(12.5,-4)*{{\scriptstyle i,\lambda}};
\endxy &\mapsto 
(-1)^{\lceil\frac{\lambda_i}{2}\rceil+\lfloor\frac{\lambda_{i+1}}{2}\rfloor}\;\xy
(0,0)*{\includegraphics[width=60px]{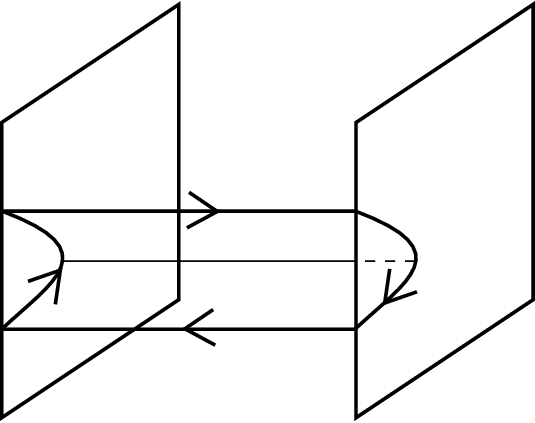}};
(-8,-10)*{\scriptstyle\lambda_i};
(8,-10)*{\scriptstyle\lambda_{i+1}};
\endxy
\\
\xy
(8,0)*{\includegraphics[width=25px]{section23/rightcap}};
(12.5,-4)*{{\scriptstyle i,\lambda}};
\endxy &\mapsto
\;\xy
(0,0)*{\includegraphics[width=60px]{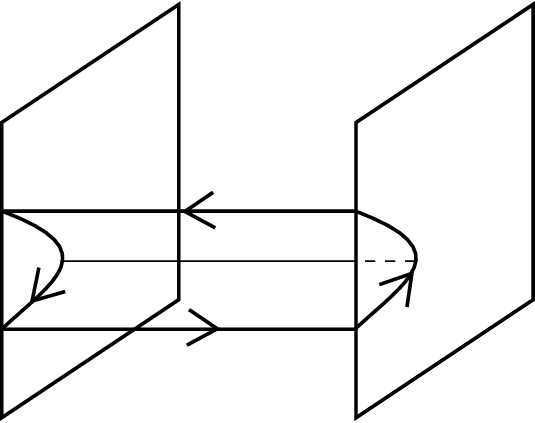}};
(-8,-10)*{\scriptstyle\lambda_i};
(8,-10)*{\scriptstyle\lambda_{i+1}};
\endxy
\end{align*} 
}
\end{defn}  

\begin{prop} 
\label{prop:cataction}
The formulas in Definition~\ref{defn:cataction} 
determine a well-defined graded functorial action of 
$\Scat(n,n)$ on ${\mathcal W}_{(3^k)}$. 
\end{prop}
\begin{proof}
A tedious but straightforward case by case check, for each 
generating morphism and each $\lambda$ which give a non-zero foam, 
shows that each of the foams in 
Definition~\ref{defn:cataction} has the 
same degree as the elementary morphism in $\Scat(n,n)$ to which it is 
associated. Note that it is important to erase the facets labeled 0 or 3, 
before computing the degree of the foams. We do just one example here. We 
have 
\begin{align*}
\xy
(8,0)*{\includegraphics[width=25px]{section23/leftcup}};
(12.5,-4)*{{\scriptstyle i,(12)}}
\endxy &\mapsto
-\;\xy
(0,0)*{\includegraphics[width=80px]{section52/lcupfoam}};
(-12,-12)*{\scriptstyle 1};
(7,-12)*{\scriptstyle 2};
(-5.5,2.5)*{\scriptstyle 0};
(12.5,2.5)*{\scriptstyle 3};
\endxy=f
&\text{and}\quad \deg(\xy
(8,0)*{\includegraphics[width=25px]{section23/leftcup}};
(12.5,-4)*{{\scriptstyle i,(12)}}
\endxy)=2.
\end{align*}
We see that $f$ has one facet labeled 0 and another 
labeled 3, so those two facets have to be erased. Therefore, $f$ has 
12 vertices, 14 edges and 3 faces, i.e.
\[
\chi(f)=12-14+3=1.
\]
The boundary of $f$ has 12 vertices and 12 edges, so 
\[
\chi(\partial f)=12-12=0.
\]
Note that the two circular edges do not belong to $\partial f$, because 
the circular facets have been removed. 
In this section we draw the foams horizontally, so 
$b$ is the number of horizontal edges at the top and the bottom of $f$, 
which go from the front to the back. Thus, for $f$ we have 
\[
b=4.
\] 
Altogether, we get 
\[
q(f)=0-2+4=2.
\]
\vskip0.5cm
In order to show that the categorical action is well-defined, one has to 
check that it preserves all the relations in Definition~\ref{def_glcat}. 
modulo 2 this was done in the proof of Theorem 4.2 in~\cite{mack}. At the 
time there was a small issue about the signs in~\cite{kl3}, which prevented 
the author to formulate and prove Theorem 4.2 in~\cite{mack} over 
$\mathbb{C}$. That issue has now been solved (see~\cite{kl4} and~\cite{msv2} 
for more information) and in this paper we use the sign conventions 
from~\cite{msv2}, which are compatible with those from~\cite{kl4}.  
We laboriously checked all these relations again, but now over $\mathbb{C}$ 
and with the signs above. The arguments are exactly the same, 
so let us not repeat them one by one here. Instead, we first explain 
how we computed the signs for the categorical 
action above and why they give the desired result over $\mathbb{C}$. 
After that, we will do an example. 
For a complete case by case check, we refer to the arguments used in the 
proof of Theorem 4.2 in~\cite{mack}. The reader should check that our signs above 
remove the sign ambiguities in that proof. 

One can compute the signs above as follows: first check the relations 
only involving strands of one color, i.e. the $\mathfrak{sl}_2$-relations. 
The first thing to notice is that 
the foams in the categorical action do not satisfy 
relation~\eqref{eq_nil_dotslide}, i.e. 
for all $\lambda$, which give a non-zero foam, the sign is wrong. 
Therefore, one is forced to multiply the foam associated to 
\[
\xy
(6,1.5)*{\includegraphics[width=20px]{section52/upcrossblue}};
(12.5,-4)*{{\scriptstyle i,i,\lambda}};
\endxy
\]
by $-1$, for all $\lambda$. 

After that, compute the foams 
associated to the degree zero bubbles (real bubbles, not fake bubbles) and 
adjust the signs of the images of the left cups and caps accordingly. 
This way, most of the signs of the images of the 
left cups and caps get determined. The remaining ones can be determined by 
imposing the zig-zag relations in~\eqref{eq_biadjoint1} 
and~\eqref{eq_biadjoint2}. 
  
Of course, one could also choose to adjust the signs of the images of 
the right cups and caps. That would determine a categorical 
action that is naturally isomorphic to the one in this paper. 
  
After these signs have been determined, one can check that all 
$\mathfrak{sl}_2$-relations are preserved by the categorical action. 

The next and final step consists in determining the signs of
\[
\xy
(6,1.5)*{\includegraphics[width=20px]{section23/upcross}};
(12.5,-4)*{{\scriptstyle i,j,\lambda}};
\endxy,
\]
for $i\ne j$. First one can check that cyclicity is already preserved. 
The relations in~\eqref{eq_cyclic_cross-gen} are preserved by the corresponding 
foams, which are all isotopic, with our sign choices for the foams 
associated to the left cups and caps. Therefore, cyclicity does not 
determine any more signs. 

The relations in~\eqref{eq_downup_ij-gen} are preserved on the nose, for 
$i=j$ and $|i-j|>1$. For $|i-j|=1$, they are only preserved up to a sign. 
Note that, since the corresponding foams are all isotopic, the signs 
actually come from the sign choice for the foams associated to the 
left cups and caps. Thus, whenever the total sign 
in the image of~\eqref{eq_downup_ij-gen} becomes negative, one has 
to change the sign of one of the two crossings (not of both of course). 
Our choice has been to change the sign of the foam associated to 
\[
\xy
(6,1.5)*{\includegraphics[width=20px]{section23/upcross}};
(12.5,-4)*{{\scriptstyle i,i+1,\lambda}};
\endxy,
\]
whenever necessary. Any other choice, consistent with all the previous 
sign choices, leads to a naturally isomorphic categorical action. 
It turns out that the sign has to be equal to $(-1)^{\lambda_{i+1}}$, after 
checking for all $\lambda$. 

After this, one can check that all relations involving two or three colors 
are preserved by the categorical action. Note that we have not specified an 
image for the fake bubbles. As stressed repeatedly in~\cite{kl3}, fake bubbles 
do not exist as separate entities. They are merely formal symbols, used as 
computational devices to keep the computations involving real bubbles 
tidy and short. As we are using $\mathfrak{sl}_3$-foams in this paper, 
most of the dotted bubbles are mapped to zero. Therefore, under the 
categorical action it is very easy to convert the fake bubbles 
in the relations in Definition~\ref{def_glcat} 
into linear combinations of real bubbles, using the 
infinite Grassmannian relation~\eqref{eq_infinite_Grass}. 
Thus, there is no need to use fake bubbles in this paper.  
\vskip0.5cm
Finally, let us do two examples, i.e. one involving only one color and another 
involving two colors. 

The left side of the equation in~\eqref{eq:EF}, for $i=1$ and 
$\lambda=(1,2)$ (the other entries are omitted for simplicity), becomes
\begin{align*}
\xy
(0,0)*{\includegraphics[width=120px]{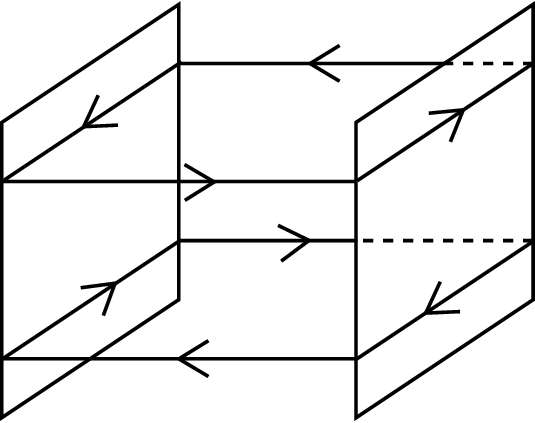}};
(-23,-18)*{\scriptstyle 1};
(5,-18)*{\scriptstyle 2};
(-23,-4)*{\scriptstyle 2};
(5,-4)*{\scriptstyle 1};
(-23,7)*{\scriptstyle 1};
(5,7)*{\scriptstyle 2};
\endxy
\;=\;
-\;\xy
(0,0)*{\includegraphics[width=120px]{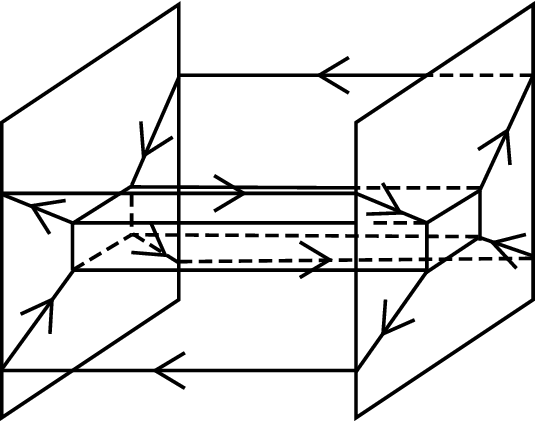}};
(-23,-18)*{\scriptstyle 1};
(5,-18)*{\scriptstyle 2};
(-19,-6)*{\scriptstyle 2};
(8,-7)*{\scriptstyle 1};
(-23,7)*{\scriptstyle 1};
(5,7)*{\scriptstyle 2};
(-14,-2.3)*{\scriptstyle 3};
(14,-2)*{\scriptstyle 0};
(-8.5,3)*{\scriptstyle 2};
(19,3)*{\scriptstyle 1};
\endxy
\;-\;
\xy
(0,0)*{\includegraphics[width=120px]{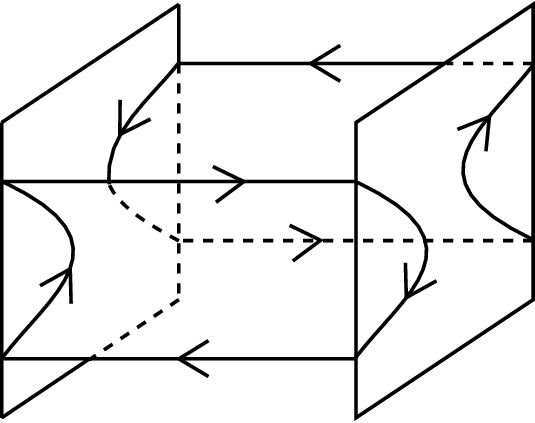}};
(-23,-18)*{\scriptstyle 1};
(5,-18)*{\scriptstyle 2};
(-19,-6)*{\scriptstyle 2};
(8,-7)*{\scriptstyle 1};
(-23,7)*{\scriptstyle 1};
(5,7)*{\scriptstyle 2};
(-8.5,4)*{\scriptstyle 2};
(19,4)*{\scriptstyle 1};
\endxy.
\end{align*}  
This foam equation is precisely the relation (SqR). 
Note that the signs match perfectly, because we have 
\[
\mathrm{sign}\left(\xy
(8,0)*{\includegraphics[width=25px]{section23/leftcap}};
(12.5,-4)*{{\scriptstyle i,(12)}}
\endxy\right)=+\quad\text{and}\quad 
\mathrm{sign}\left(\xy
(8,0)*{\includegraphics[width=25px]{section23/leftcup}};
(12.5,-4)*{{\scriptstyle i,(12)}}
\endxy\right)=-.
\]
\vskip0.5cm
The equation in~\eqref{eq_r2_ij-gen}, for $(i,j)=(1,2)$ and 
$\lambda=(121)$ (the other entries are omitted for simplicity), becomes
\begin{align*}
\xy
(0,0)*{\includegraphics[width=120px]{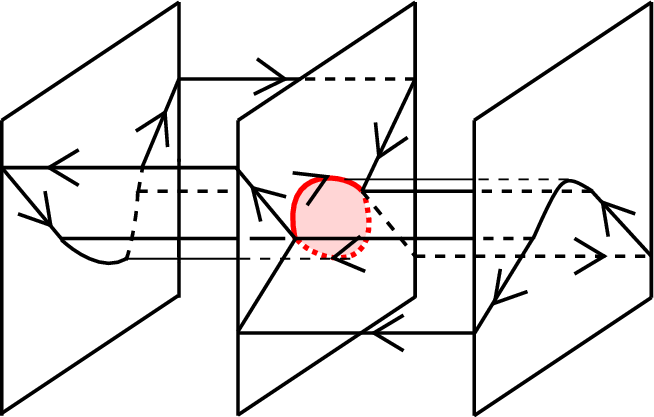}};
(-17,-13)*{\scriptstyle 1};
(-2,-13)*{\scriptstyle 2};
(13,-13)*{\scriptstyle 1};
(-13,13)*{\scriptstyle 2};
(2,13)*{\scriptstyle 2};
(17,13)*{\scriptstyle 0};
(0.5,-0.5)*{\scriptstyle 1};
(-4.5,-3.5)*{\scriptstyle 3};
(4.5,3)*{\scriptstyle 3};
\endxy
\;=\;
\xy
(0,0)*{\includegraphics[width=120px]{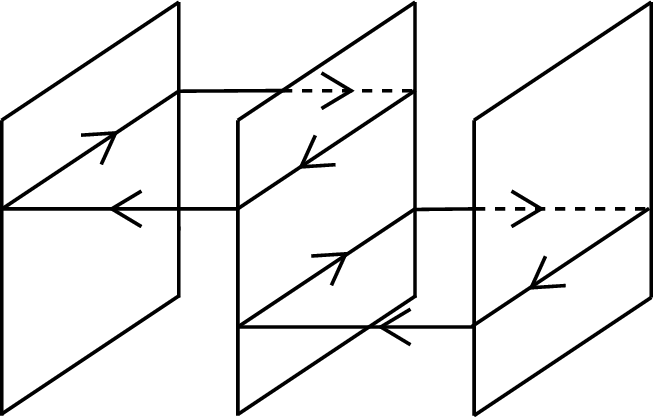}};
(-17,-13)*{\scriptstyle 1};
(-2,-13)*{\scriptstyle 2};
(13,-13)*{\scriptstyle 1};
(-13,13)*{\scriptstyle 2};
(2,13)*{\scriptstyle 2};
(17,13)*{\scriptstyle 0};
(0.5,-0.5)*{\scriptstyle 3};
(7.5,-4)*{\bullet};
\endxy
\;-\;
\xy
(0,0)*{\includegraphics[width=120px]{section52/R2ijrhsfoam}};
(-17,-13)*{\scriptstyle 1};
(-2,-13)*{\scriptstyle 2};
(13,-13)*{\scriptstyle 1};
(-13,13)*{\scriptstyle 2};
(2,13)*{\scriptstyle 2};
(17,13)*{\scriptstyle 0};
(0.5,-0.5)*{\scriptstyle 3};
(-7.5,4)*{\bullet};
\endxy.
\end{align*} 
To see that this holds, apply the (RD) relation to the foam on the l.h.s., 
in order to remove the disc bounded by the red singular circle on the middle 
sheet. 
\end{proof}

Let $W_h\cong \mathbb{C}$ be the unique indecomposable projective graded 
$K^{(\times^k,\circ^{2k})}$-module of degree zero. 
Recall that $K^{(\times^k,\circ^{2k})}$ is 
generated by the empty diagram, so $W_h$ is indeed one-dimensional. It is the 
categorification of $w_h$, the highest weight vector in $W_{(3^k)}$. 

As already remarked, we can pull back the functorial action of $\Scat(n,n)$ on 
${\mathcal W}_{(3^k)}$ 
via 
\[
\Psi_{n,n}\colon \Ucat\to \Scat(n,n).
\]
We are now able to prove one of our main results. Recall that $\mathcal V$ is any additive, idempotent complete category, which allows an integrable, graded categorical action by $\Ucat$.
\begin{thm}
\label{thm:equivalence}
There exists an equivalence of strong $\mathfrak{sl}_n$-$2$-representations   
\[
\Phi\colon {\mathcal V}_{(3^k)}\to {\mathcal W}_{(3^k)}.
\]
\end{thm}   
\begin{proof}
As we already mentioned above, we have  
\[
\mathrm{End}_{{\mathcal W}_{(3^k)}}(W_h)\cong \mathbb{C}.
\]
Let $Q$ be any indecomposable object in ${\mathcal W}_{(3^k)}$. 
There exists an enhanced sign string $S$ such that $Q$ belongs to 
$\mathcal{W}^S$. Therefore, there exists a 
basis web $w\in B^S$ and a $t\in\mathbb{Z}$, such that $Q$ is a graded 
direct summand of $P_w\{t\}$. Without loss of generality, we may assume that 
$t=-\ell(\hat{S})$.   

By Lemma~\ref{lem:phisurj} and Proposition~\ref{prop:cataction}, 
there exists an object of 
$X$ in $\Scat(n,n)$ such that $Q$ is a direct summand of $XW_h$. 
This holds, because in $\ScatD(n,n)$, the Karoubi envelope of $\Scat(n,n)$, 
the divided powers correspond to direct summands of ordinary powers. 
For more details on the categorification of the divided powers 
see~\cite{kl3} and~\cite{klms}.

Proposition~\ref{prop:rouquier} now proves the existence of $\Phi$.  
\end{proof}

An easy consequence of Theorem~\ref{thm:equivalence} is the following.
\begin{cor}
\label{cor:equivalence}
By Theorem~\ref{thm:equivalence}, the $S_q(n,n)$-module map 
\[
K^{\oplus}_0(\Phi)_{\bQ(q)}\colon K^{\oplus}_0({\mathcal V}_{(3^k)})_{\bQ(q)}\to 
K^{\oplus}_0({\mathcal W}_{(3^k)})_{\bQ(q)}
\]
is an isomorphism.\qed 
\end{cor}

The following consequence of Theorem~\ref{thm:equivalence} is very important 
and we thank Ben Webster for explaining its proof. 
\begin{prop}
\label{prop:morita} 
The graded algebras $K_{(3^k)}$ and $R_{(3^k)}$ are Morita equivalent. 
\end{prop}
\begin{proof}
For each weight $\mu_S\in \Lambda(n,n)_3$, let 
$\nu_S=(\nu_1,\ldots,\nu_n)$ be the unique element of $\mathbb{Z}_{\leq 0}[I]$ such that 
\[
\mu_S-(3^k)=\sum_{i=1}^n\nu_i\alpha_i.
\]
We are going to show that for each such $\mu_S$, 
the graded algebras $K^S$ and $R_{(3^k)}^{\nu_S}$ are Morita equivalent. This proves 
the proposition after taking direct sums over all weights. 

Let $\mu_S\in\Lambda(n,n)_3$. Define 
\[
\Theta_{\mu_S}=\bigoplus_{\ii\in\mathrm{Seq}(\nu_S)}
\mathcal{E}_{\ii}W_h \in \mathcal{W}^S.
\] 
In the proof of Theorem~\ref{thm:equivalence}, we already showed that 
every object in 
$\mathcal{W}^S$ is a direct summand 
of $XW_h$ for some object $X\in\Scat(n,n)$. By the biadjointness of 
the $\mathcal{E}_i$ and $\mathcal{E}_{-i}$ in $\Scat(n,n)$ (i.e. up to a shift) 
and the fact that 
$W_h$ is a highest weight object, it is not hard to see that 
$XW_h$ itself is a direct summand of a finite direct sum of 
degree shifted copies of 
$\Theta_{\mu_S}$. This shows that 
every object in $\mathcal{W}^S$ is a direct summand 
of a finite direct sum of degree shifted copies of $\Theta_{\mu_S}$. Since $K^S$ is a finite dimensional, complex algebra, every 
finite dimensional, graded $K^S$-module has a projective cover and is therefore 
a quotient of a finite direct sum of degree-shifted copies of $\Theta_{\mu_S}$. 
This shows that $\Theta_{\mu_S}$ is a projective generator of 
$K^S\text{-}\mathrm{\textbf{Mod}}_{\mathrm{gr}}$. 

Theorem~\ref{thm:equivalence} also shows that 
\[
\mathrm{End}_{K^S}\left(\Theta_{\mu_S}\right)\cong R_{(3^k)}^{\nu_S}
\]
holds. 

By a general result due to Morita, it follows that 
the above observations imply that 
$K^S$ and $R_{(3^k)}^{\nu_S}$ are Morita 
equivalent. For a proof see Theorem 5.55 in~\cite{rot}, for example.
\end{proof}

We can draw two interesting conclusions from Proposition~\ref{prop:morita}.

In~\cite{bk2}, Brundan and Kleshchev defined an explicit isomorphism between 
blocks of cyclotomic Hecke algebras and cyclotomic KLR-algebras. 
Theorem 3.2 in~\cite{bru2} implies that the center of the 
cyclotomic Hecke algebra, which under Brundan and Kleshchev's isomorphism 
corresponds to $R_{(3^k)}^{\nu_S}$, has the same dimension as 
$H^*(X^{(3^k)}_{\mu_S})$.

\begin{cor}
\label{cor:moritacenter}
The center of $K^S$ is isomorphic to the center of $R_{(3^k)}^{\nu_S}$. 
In particular, we have 
\[
\dim Z(K^S)=\dim Z(R_{(3^k)}^{\nu_S})=\dim 
H^*(X^{(3^k)}_{\mu_S}).
\]
\end{cor}
\begin{proof}
We only have to prove the first statement, which follows from 
the well known fact that 
Morita equivalent algebras have 
isomorphic centers. For a proof see for example Corollary 18.42 in~\cite{lam}.
\end{proof}

In Theorem~\ref{thm:center} we used 
Corollary~\ref{cor:moritacenter} to give an explicit 
isomorphism  
\[
H^*(X^{(3^k)}_{\mu_S})\to Z(K^S).
\] 

\begin{rem}
Just for completeness, we remark that the aforementioned results 
in~\cite{bru2} and~\cite{bk2} together with the results in~\cite{bru}, 
which we have not explained, imply that 
\[
H^*(X^{(3^k)}_{\mu_S})\cong Z(R_{(3^k)}^{\nu_S}),
\] 
so we have not proved anything new about 
$Z(R_{(3^k)}^{\nu_S})$.
\end{rem}

Another interesting consequence of Proposition~\ref{prop:morita} is 
the following.
\begin{cor}
\label{cor:moritacellular}
$K^S$ is a graded cellular algebra. 
\end{cor}
\begin{proof}
In Corollary 5.12 in~\cite{hm}, Hu and Mathas proved that 
$R_{(3^k)}^{\nu_S}$ is a graded cellular algebra. 

In~\cite{kx}, K\"{o}nig and Xi showed that ``being a cellular algebra'' is 
a Morita invariant property, provided that the algebra is defined over 
a field whose characteristic is not equal to two.  

These two results together with Proposition~\ref{prop:morita} 
prove that the $\mathfrak{sl}_3$-web algebra $K^S$ is indeed a graded cellular algebra. 
\end{proof}

The precise definition of a graded cellular algebra can be found 
in~\cite{hm}. We will not recall it here. In a follow-up paper, we 
intend to discuss the cellular basis of $K^S$ in detail and use it to 
derive further results on the representation theory of $K^S$.  

\begin{rem} 
Corollary~\ref{cor:moritacellular} is the $\mathfrak{sl}_3$ analogue of 
Corollary 3.3 in~\cite{bs}, which proves that Khovanov's arc algebra $H^m$ 
is a graded cellular algebra. It is easy to give a 
cellular basis of $H^m$. The proof of cellularity follows from checking a 
small number of cases by hand. 
For $K^S$, we tried to mimic that approach, but had to give up because the 
combinatorics got too complex.      
\end{rem}

\subsubsection{The Grothendieck group of $W_{(3^k)}$}
\label{sec:Grothendieck}
Recall that $W^S$ has an inner product defined 
by the normalized Kuperberg form 
(see Definition~\ref{defn:normkuperform}). 
The \textit{Euler form}
\[
\langle [P],[Q]\rangle_{\mathrm{Eul}}=\dim_q \mathrm{HOM}(P,Q)
\]
defines a $q$-sesquilinear form on $K^{\oplus}_0(\mathcal{W}^S)_{\bQ(q)}$.

\begin{lem}
\label{lem:grothendieck}
Let $S$ be an enhanced sign sequence. Take 
\[
\gamma_S\colon W^S\to K^{\oplus}_0(\mathcal{W}^S)_{\bQ(q)}
\] 
to be the $\mathbb{Q}(q)$-linear 
map defined by 
\[
\gamma_S(u)=q^{-\ell(\hat{S})}[P_u],
\]
for any $u\in B^S$. Then $\gamma_S$ is an isometric embedding.   

This implies that the $\mathbb{Q}(q)$-linear map
\[
\gamma_W=\bigoplus_{\mu(S)\colon\Lambda(n,n)_3}\gamma_S
\] 
defines an isometric embedding 
\[
\gamma_W\colon W_{(3^k)}\to K^{\oplus}_0(\mathcal{W}_{(3^k)})_{\bQ(q)}.
\] 
\end{lem}
\begin{proof}
Note that the normalized Kuperberg form, because of the 
relations~\ref{eq:circle}, ~\ref{eq:digon} and~\ref{eq:square}, and the 
Euler form are non-degenerate.   
For any pair $u,v\in B^S$, we have  
\[
\dim_q \mathrm{HOM}(P_u\{-\ell(\hat{S})\},P_v\{-\ell(\hat{S})\})=\dim_q \mathrm{HOM}(P_u,P_v)=\dim_q {}_uK_v=
q^{\ell(\hat{S})}\langle u^*v\rangle_{\mathrm{Kup}}.
\] 
The factor $q^{\ell(\hat{S})}$ is a consequence of the grading shift in the 
definition of ${}_uK_v$.

Thus, $\gamma_S$ is an isometry. Since the normalized Kuperberg form is 
non-degenerate, this implies that $\gamma_S$ is an embedding. 
\end{proof}

\begin{rem} 
\label{rem:counter2}
In Section 5.5 in~\cite{mn}, Morrison and Nieh showed that 
$P_u$ is not necessarily indecomposable (see also~\cite{rob}). This 
is closely related to 
the contents of Remark~\ref{rem:counter}, as Morrison and Nieh showed.  
Therefore, the surjectivity of $\gamma_W$ is not immediately clear and we need 
the results of the previous subsections to establish it below. 

The $\mathfrak{sl}_2$ case is much simpler. The projective modules of the 
arc algebras analogous to the $P_u$ are all indecomposable. 
See Proposition 2 in~\cite{kh} for 
the details. 
\end{rem}

\begin{thm}
\label{thm:equivalence2}
The map  
\[
\gamma_W\colon W_{(3^k)}\to K^{\oplus}_0({\mathcal W}_{(3^k)})_{\bQ(q)}
\] 
is an isomorphism of $S_q(n,n)$-modules. 

This also implies that, for each sign string $S$ with $\mu_S\in
\Lambda(n,n)_3$, the map  
\[
\gamma_S\colon W^S\to K^{\oplus}_0(\mathcal{W}^S)_{\bQ(q)}
\]
is an isomorphism. 
\end{thm}
\begin{proof}
The proof of the theorem is only a matter of assembling already known pieces.

By Proposition~\ref{prop:cataction} and the fact that 
$$\Gamma(u)=P_u\{-\ell(\hat{S})\},$$
for any $u\in B^S$, we see that   
$\gamma_W$ intertwines the $S_q(n,n)\cong 
K^{\oplus}_0(\ScatD(n,n))_{\bQ(q)}$-actions. 

We already know that $\gamma_W$ is an embedding, 
by Lemma~\ref{lem:grothendieck}. 

Note that, by Theorem~\ref{thm:bk}, Lemma~\ref{lem:isoirrep} and 
Corollary~\ref{cor:equivalence}, 
we have the following commuting square   
\[
\begin{CD}
V_{(3^k)}&@>{\gamma_V}>>&K^{\oplus}_0({\mathcal V}_{(3^k)})_{\bQ(q)}
\\
@V{\phi}VV&&@V{K^{\oplus}_0(\Phi)}VV\\
W_{(3^k)}&@>{\gamma_W}>>&K^{\oplus}_0({\mathcal W}_{(3^k)})_{\bQ(q)}.
\end{CD}
\]
We already know that $\gamma_V$, $\phi$ and $K^{\oplus}_0(\Phi)_{\bQ(q)}$ are isomorphisms. 
Therefore, $\gamma_W$ has to be an isomorphism. This shows that 
$K^S$ indeed categorifies the $\mu_S$-weight space 
of $V_{(3^k)}$. 
\end{proof}
\vskip0.5cm
A good question is how to find the graded, indecomposable, 
projective modules of $K^S$. Before answering that question, 
we need a result on the 3-colorings of webs. 

Let $w\in B^S$. Recall that there is a bijection between the flows on $w$ 
and the $3$-colorings of $w$, as already mentioned in Remark~\ref{rem:3color}. 
Call the 3-coloring corresponding to the canonical flow of $w$, the 
{\em canonical 3-coloring}, denoted $T_w$. 
\begin{lem}
\label{lem:tech}
Let $u, v\in B^S$. If there 
is a 3-coloring of $v$ which matches $T_u$ and a 3-coloring of $u$ 
which matches $T_v$ on the common boundary $S$, then $u=v$. 
\end{lem} 
\begin{proof}
This result is a direct consequence of Theorem~\ref{thm:upptriang}. Recall 
that there is a partial order on flows, and therefore on 3-colorings by 
Remark~\ref{rem:3color}. This ordering is induced by the lexicographical 
order on the state-strings on $S$, which are induced by the flows. Note that 
two matching colorings of $u$ and $v$ have the same order, by definition. 
On the other hand, Theorem~\ref{thm:upptriang} implies that 
any 3-coloring of $u$, respectively $v$, has order less than or equal to 
that of $T_u$ and $T_v$ respectively. 
Therefore, if there exists a 3-coloring of $v$ matching $T_u$, the order of 
$T_u$ must be less or equal than that of $T_v$. 

Thus, if there 
exists a 3-coloring of $v$ matching $T_u$ 
and a 3-coloring of $u$ mathing $T_v$, then $T_u$ and $T_v$ must have the 
same order. This implies that $u=v$, because canonical 3-colorings are uniquely 
determined by their order and the corresponding canonical flows determine 
the corresponding basis webs uniquely by the growth algorithm.   
\end{proof}

\begin{prop}
\label{prop:unitriang}
For each web $u\in B^S$, there exists a graded, indecomposable, projective 
$K^S$-module $Q_u$, which is unique up to degree-preserving isomorphism, 
such that 
\[
P_u\cong Q_u\oplus\bigoplus_{J_v < J_u} Q_v^{\oplus d(S,J_u,J_v)}.
\]
Here $J_u$ is the state string associated to the canonical flow on $u$, 
the coefficients $d(S,J_u,J_v)$ belong to $\mathbb{N}[q,q^{-1}]$ 
and the state strings are ordered lexicographically. 
\end{prop}
\begin{proof}
Let $u\in B^S$. Then there is a complete decomposition of $1_u$ into orthogonal 
primitive idempotents 
\[
1_u=e_1+\cdots+e_r.
\]
By Theorem~\ref{thm:sjodin} and Corollary~\ref{cor:sjodin}, we can lift this 
decomposition to $G^S$. 
We do not introduce any 
new notation for this lift, trusting that the reader will not get confused 
by this slight abuse of notation. 

Let $z_u\in Z(G^S)$ be the central idempotent corresponding to $J_u$, as 
defined in the proof of Lemma~\ref{lem:dimZG}. We claim that there is a unique 
$1\leq i\leq r$, such that 
\begin{equation}
\label{eq:orto1}
z_ue_j=\delta_{ij}z_u1_u,
\end{equation}
for any $1\leq j\leq r$. 

Let us prove this claim. Note that 
\begin{equation}
\label{eq:orto2}
z_u1_u=e_{u,T_u},
\end{equation}
where $T_u$ is the canonical coloring of $u$, i.e. 
$J_u$ only allows one compatible coloring of $u$, which is $T_u$. 
Since $e_{u,T_u}\ne 0$, courtesy of Lemma~\ref{lem:embedding}, this implies that 
\begin{equation}
\label{eq:dimone}
z_u{}_uG_u={}_uG_u z_u=z_u{}_uG_uz_u=e_{u,T_u}G^Se_{u,T_u}\cong \mathbb{C},
\end{equation}
by Theorem~\ref{thm:Gornik}. 

We also see that there has to exist at least one 
$1\leq i_0\leq r$ such that $z_ue_{i_0}\ne 0$. Then, by~\eqref{eq:dimone}, 
there exists a non-zero $\lambda_{i_0}\in \mathbb{C}$, such that 
\[
z_ue_{i_0}=\lambda_{i_0}z_u1_u=\lambda_{i_0}e_{u,T_u}.
\]  
For any 
$1\leq i,j\leq r$, we have  
\[
z_ue_iz_ue_j=z_u^2e_ie_j=z_u\delta_{ij}e_i.
\]
This implies that $i_0$ is unique and $\lambda_{i_0}=1$. 
In order to see that this is true, suppose there exist 
$1\leq i_0\ne j_0\leq r$ such that $z_ue_{i_0}\ne 0$ and $z_ue_{j_0}\ne 0$. 
By~\eqref{eq:dimone}, there exist non-zero 
$\lambda_{i_0},\lambda_{j_0}\in\mathbb{C}$ such that 
\[
z_ue_{i_0}=\lambda_{i_0}z_u1_u\quad\text{and}\quad z_ue_{j_0}=\lambda_{j_0}
z_u1_u.
\]
However, this is impossible, because we get  
\[
z_ue_{i_0}z_ue_{j_0}=\lambda_{i_0}\lambda_{j_0}z_u1_u\ne 0,
\] 
which contradicts the orthogonality of $z_ue_{i_0}$ and $z_ue_{j_0}$. 

Thus, for each $u\in B^S$, there is a unique primitive idempotent 
$e_u\in \mathrm{End}_{\mathbb{C}}(P_u)$ that is not killed by $z_u$, 
when lifted to 
$G^S$. We define $Q_u$ to be the corresponding graded indecomposable 
projective $K^S$-module by  
\[
Q_u=K^Se_u,
\]
which is clearly a direct summand of $P_u=K^S1_u$. 
\vskip0.5cm
Let us now show that, for any $u,v\in B^S$, we have  
\[
Q_u\cong Q_v\Leftrightarrow u=v.
\]
If $u=v$, we obviously have $Q_u\cong Q_v$. 
Let us prove the other implication. Suppose $Q_u\cong Q_v$. From 
the above, recall that $e_u$ and $e_v$ can be lifted to $G^S$. By a slight 
abuse of notation, call these lifted idempotents $e_u$ and $e_v$ again. 
We have 
\[
z_ue_u=e_{u,T_u}\ne 0\quad\text{and}\quad z_ve_v=e_{v,T_v}\ne 0.
\]
Since $Q_u\cong Q_v$, we then also have 
\[
z_ue_v\ne 0\quad\text{and}\quad z_ve_u\ne 0.
\]
This can only hold if $T_u$ gives a 3-coloring of $v$ and $T_v$ a 3-coloring 
of $u$. By Lemma~\ref{lem:tech}, this implies that $u=v$.
\vskip0.5cm

Since 
\[
\dim K^{\oplus}_0(\mathcal{W}^S)_{\bQ(q)} 
=\dim W^S=\#B^S,
\]
by Theorem~\ref{thm:equivalence2}, 
the above shows that 
\[
\left\{Q_u\mid u\in B^S\right\}
\]
is a basis of $K^{\oplus}_0(\mathcal{W}^S)_{\bQ(q)}$. For any $u,v\in B^S$, we 
have  
\[
z_u1_u=z_ue_u\quad\text{and}\quad z_v1_u=0,\;\text{if}\; J_v>J_u.
\]
The second claim follows from the fact that there are no admissible 
3-colorings of $u$ greater than $J_u$. The proposition now follows. 
\end{proof}

\begin{rem} 
\label{rem:mn}
Proposition~\ref{prop:unitriang} proves the conjecture 
about the decomposition of $1_u$, which Morrison and Nieh 
formulate in the text between Conjectures 5.14 and 5.15 in~\cite{mn}.  
\end{rem}

Before giving the last result of this section, we briefly recall 
some facts about the \textit{dual canonical basis} of $W^S$. 
For more details see~\cite{fkhk} and~\cite{kk}. There exists a $q$-antilinear 
involution $\tilde{\psi}$ on $V^S$ (in~\cite{fkhk} and~\cite{kk} 
this involution is denoted $\psi'$ and $\Phi$, respectively). 
For any sign string $S$ and any state 
string $J$, there exists a unique element $e^S_{\heartsuit J}\in V^S$ which 
is $\tilde{\psi}$-invariant and satisfies  
\begin{equation}
\label{eq:dualcan}
e^S_{\heartsuit J}=e^S_J+\sum_{J^{\prime}<J}c(S,J,J^{\prime})e^{S}_{J^{\prime}},
\end{equation}
with $c(S,J,J^{\prime})\in q^{-1}\mathbb{Z}[q^{-1}]$. 
(Recall again that we use $q$ instead of $v$, contrary to~\cite{fkk} and~\cite{kk} where $v$ indicates $-q^{-1}$.) 
The $e^S_J$ are the elementary tensors, which were defined 
in~\ref{thm:upptriang}. The basis $\left\{e^S_{\heartsuit J}\right\}$ 
is called the \textit{dual canonical basis} of $V^S$. Restriction to the 
dominant closed paths $(S,J)$ 
gives the dual canonical basis of $W^S$ (see Theorem 3 
in~\cite{kk} and the comments below it).

We have not given a definition of $\tilde{\psi}$, but we 
note that $\tilde{\psi}$ is completely determined by Proposition 2 
in~\cite{kk}, which we recall now. 
\begin{prop}\label{prop:basisbarinvariant}(\textbf{Khovanov-Kuperberg})
Each basis web $w\in B^S$ is invariant under $\tilde{\psi}$. 
\end{prop}

Let us show that there exists a duality 
$\circledast$ on $\mathcal{W}_{(3^k)}$ 
such that $\gamma_W\colon W_{(3^k)}\to K_0^{\oplus}(\mathcal{W}_{(3^k)})_{\bQ(q)}$
intertwines $\tilde{\psi}$ and $K_0^{\oplus}(\circledast)_{\bQ(q)}$. 

As Brundan and Kleshchev showed in Section 4.5 in~\cite{bk}, 
there is a duality $\circledast$ on $\mathcal{V}_{(3^k)}$. It is induced by 
Khovanov and Lauda's~\cite{kl3} algebra anti-automorphism 
\[
*\colon R_{(3^k)}\to R_{(3^k)},
\]
given by reflecting the diagrams in the $x$-axis and inverting 
their orientations. Let $M$ be a finite dimensional, graded $R_{(3^k)}$-module, then 
\[
M^{\circledast}=M^{\vee}
\] 
as a graded vector space and $R_{(3^k)}$ acts on $M^{\circledast}$ by 
\[
xf(y)=f(x^{*}y).
\]
In Theorem 4.18 in~\cite{bk} Brundan and Kleshchev showed that $\circledast$ 
commutes with the categorical actions of $\mathcal{E}_{+i}$ and $\mathcal{E}_{-i}$, 
for any $i=1,\ldots,n-1$, and the map 
$\gamma_V\colon V_{(3^k)}\to K_0^{\oplus}(\mathcal{V}_{(3^k)})_{\bQ(q)}$ 
intertwines the usual bar-involution on $V_{(3^k)}$, which we 
also denote by $\tilde{\psi}$, and 
$K_0(\circledast)_{\mathbb{Q}(q)}$. 

Using our equivalence 
\[
\Phi\colon \mathcal{V}_{(3^k)}\to \mathcal{W}_{(3^k)}
\] 
from Theorem~\ref{thm:equivalence}, we can define a duality 
$\circledast$ on $\mathcal{W}_{(3^k)}$ 
exactly as above. We use the anti-automorphism  
$*\colon K_{(3^k)}\to K_{(3^k)}$ defined by 
reflecting the foams in the vertical 
$yz$-plane, i.e. the plane parallel to the front and the back of the foams in 
Definition~\ref{defn:cataction}, and inverting the orientation of 
their edges. Note that $\circledast$ maps indeed 
projective modules to projective modules, 
because $K^{S}$ is a symmetric Frobenius algebra for all $S$ such that $\mu_S\in \Lambda(n,n)_3$. Note also that, by construction, the equivalence  
$$\Phi\colon \mathcal{V}_{(3^k)}\to \mathcal{W}_{(3^k)}$$
intertwines the dualities $\circledast$ on both categories.   

\begin{lem}
\label{lem:Pvbarinv}
For any $S$ such that $\mu_S\in\Lambda(n,n)_3$ and for any $v\in B^S$, we have 
\[
P_v^{\circledast}\cong P_v\{-2\ell(\hat{S})\}.
\]
\end{lem}
\begin{proof}
Let $u,v\in B^S$. The proof of Theorem~\ref{thm:frob} shows that the map determined by 
\[
f\mapsto \mathrm{tr}(f^*-)
\]
defines a degree-preserving isomorphism 
\[
{}_uK_v\to {}_uK_v^{\vee}\{2\ell(\hat{S})\}.
\]
Taking the direct sum over all $v\in B^S$, we get 
a degree-preserving isomorphism of graded vector spaces
\begin{equation}
\label{eq:isoprojstar}
P_v\cong P_v^{\vee}\{2\ell(\hat{S})\}.
\end{equation}
Since $*$ is an anti-automorphism on $K^S$, we see that 
\[
gf\mapsto \mathrm{tr}((gf)^*-)=\mathrm{tr}(f^*g^*-)=g \mathrm{tr}(f^*-).
\]
This shows that~\eqref{eq:isoprojstar} is an isomorphism of graded $K^S$-modules. 
\end{proof}

By Proposition~\ref{prop:basisbarinvariant}, we see that 
the isomorphism 
\[
\gamma_W\colon W_{(3^k)}\to K^{\oplus}_0(\mathcal{W}_{(3^k)})_{\bQ(q)}
\] 
intertwines the bar involutions $\tilde{\psi}$ and 
$K^{\oplus}_0(\circledast)_{\bQ(q)}$. 

In Theorem 5.14 in~\cite{bk}, Brundan and Kleshchev also proved that 
$\gamma_V$ maps the canonical basis elements of $V_{(3^k)}$ to 
the Grothendieck classes of certain indecomposables in 
$K^{\oplus}_0(\mathcal{V}_{(3^k)})_{\bQ(q)}$. They used multipartitions $\lambda$ to 
parametrize these indecomposables, which they denoted $Y({\lambda})\{-m\}$ 
for degree shifts $m$ depending on $\lambda$.
  
These multipartitions corresponds bijectively to the semi-standard tableaux, 
which in turn correspond bijectively to the elements of $B^S$. We therefore 
can denote these indecomposables by $Y_u\{-\ell(\hat{S})\}$, for $u\in B^S$. 
By (4.26) in~\cite{bk} (note that $\#\cong \{2\ell(\hat{S})\} \circledast$), 
we also have 
\begin{equation}
\label{eq:Ybarinv}
Y_u\{-\ell(\hat{S})\}^{\circledast}\cong Y_u\{-\ell(\hat{S})\}.
\end{equation}

Finally, in Theorem 4.18 in~\cite{bk} it was proved that $\gamma_V$ 
intertwines the $q$-Shapovalov form and the Euler form. The $q$-Shapovalov 
form $\langle -,-\rangle_{\mathrm{Shap}}$ on $V^S$ is related to 
the symmetric $\mathbb{Q}(q)$-bilinear form $(-,-)_{\mathrm{Lusz}}$ 
defined by 
\begin{equation}
\label{eq:bilinversussesquil}
(-,-)_{\mathrm{Lusz}}=\overline{\langle -,\psi(-)\rangle}_{\mathrm{Shap}}.
\end{equation}
This bilinear form is used by Lusztig in~\cite{lu}. By Proposition 19.3.3 
in that book, this implies that the Euler form on 
$\mathcal{V}_{(3^k)}$ satisfies 
\begin{equation}
\label{eq:almostortY}
\langle [Y_u\{-\ell(\hat{S})\}], [Y_v\{-\ell(\hat{S})\}]\rangle_{\mathrm{Eul}} \in \delta_{u,v}+q\mathbb{Z}[q]
\end{equation}
for any $u,v\in B^S$.
 
\begin{thm}
\label{thm:dualcan}
The basis 
\[
\left\{q^{-\ell(\hat{S})}[Q_u]\mid u\in B^S\right\}
\]
corresponds to the dual canonical basis 
of $\mathrm{Inv}_{U_q(\mathfrak{sl}_3)}(V^S)$, under the isomorphisms 
\[ 
\mathrm{Inv}_{U_q(\mathfrak{sl}_3)}(V^S)\cong W^S\cong K^{\oplus}_0(\mathcal{W}^S)_{\bQ(q)}.
\]
\end{thm}
\begin{proof}
Note that by Theorem~\ref{thm:upptriang} and Proposition~\ref{prop:unitriang}, 
we have 
\begin{equation}
\label{eq:upptriang}
q^{-\ell(\hat{S})}[Q_u]=e_{J_u}^S+\sum_{J^{\prime} < J_u} e(S,J_u,J^{\prime})e^S_{J^{\prime}}
\end{equation}
for any $u\in B^S$. The dual canonical basis elements are uniquely 
determined by the conditions that they are bar-invariant and 
satisfy~\eqref{eq:dualcan}, so all we have to show is 
\begin{eqnarray}
\label{eq:firstcondition} Q_u^{\circledast}\{-\ell(\hat{S})\}&\cong&
Q_u\{-\ell(\hat{S})\}\\
\label{eq:secondcondition} e(S,J_u,J^{\prime})&\in&q^{-1}\mathbb{Z}[q^{-1}].
\end{eqnarray} 
The equivalence 
$\Phi\colon \mathcal{V}_{(3^k)}\to \mathcal{W}_{(3^k)}$ maps indecomposables 
to indecomposables. By the observations above, this shows that for 
any $u\in B^S$ there exists a $t_u\in \mathbb{Z}$ such that 
\[
Q_u^{\circledast}\cong Q_u\{t_u\}.
\]
By Proposition~\ref{prop:unitriang} and Lemma~\ref{lem:Pvbarinv}, we get 
$t_u=-2\ell(\hat{S})$ for all $u\in B^S$. 
Hence, this proves~\eqref{eq:firstcondition}. 
In particular, this implies that 
$\Phi$ maps $Y_u$ to $Q_u$ for any $u\in B^S$. 
\vskip0.5cm
In order to prove~\eqref{eq:secondcondition}, we first recall Lusztig's 
symmetric $\mathbb{Q}(q)$-bilinear form $(-,-)_{\mathrm{Lusz}}$ form on 
$V^S$ (see Section 19.1.1 and Section 27.3 in~\cite{lu}). 
A short calculation shows that 
it is completely determined by 
\begin{equation}
\label{eq:Lusorthog}
(e^S_{J^{\prime}},e^S_{J^{\prime\prime}})_{\mathrm{Lusz}}=\delta_{J^{\prime},J^{\prime\prime}},
\end{equation}
for any elementary tensors $e^S_{J^{\prime}}$ and $e^S_{J^{\prime\prime}}$. Lusztig's symmetric 
bilinear form 
can be restricted to $W^{\mathcal{Z}}_S$ and 
condition~\eqref{eq:secondcondition} is then equivalent to the condition 
\begin{equation}
\label{eq:Lusorthog2}
(q^{-\ell(\hat{S})}[Q_u],q^{-\ell(\hat{S})}[Q_v])_{\mathrm{Lusz}}\in 
\delta_{u,v}+q^{-1}\mathbb{Z}[q^{-1}]
\end{equation}
for any $u,v\in B^S$. 

Similarly, we can define a symmetric $\mathbb{Q}(q)$-bilinear web form on 
$W^S$ by 
\[
(u,v)_{\mathrm{Kup}}=q^{-\ell(\hat{S})}\langle u^*v\rangle_{\mathrm{Kup}}.
\]

We claim that both bilinear forms are equal. Since 
$(-,-)_{\mathrm{Lusz}}$ and $(-,-)_{\mathrm{Kup}}$ 
are $\mathbb{Q}(q)$-bilinear and symmetric, 
it suffices to show that we have 
\[
(u,u)_{\mathrm{Lusz}}=(u,u)_{\mathrm{Kup}},
\]
for any $u\in B^S$. Let $u\in B^S$ be arbitrary and write  
\[
u=e^S_{J_u}+\sum_{J^{\prime}<J_u} c(S,J_u,J^{\prime})e^S_{J^{\prime}},
\]
as in Theorem~\ref{thm:upptriang}. Then, by~\eqref{eq:Lusorthog}, we get 
\begin{equation}
\label{eq:firstcheck}
(u,u)_{\mathrm{Lusz}}=1 + \sum_{J^{\prime}<J}c(S,J,J^{\prime})^2.
\end{equation}

Now let us 
compute $(u,u)_{\mathrm{Kup}}$. By definition, 
we have 
\[
(u,u)_{\mathrm{Kup}}=q^{-\ell(S)}\langle 
u^*u\rangle_{\mathrm{Kup}}.
\]
Consider the way in which the coefficients $c(S,J_u,J^{\prime})$ change 
under the symmetry $x\mapsto x^*$, for $x$ any $Y$, cup or cap with flow. 
Comparing the corresponding weights 
in~\eqref{weights} and~\eqref{weights2}, we get 
\[
\mathrm{wt}(x^*)=q^{(\ell(t(x))-\ell(b(x)))}\mathrm{wt}(x).
\]
where $t(x)$ and $b(x)$ are the top and bottom boundary of $x$. 
Recall also that the canonical flow on $u$ has weight 0 (see 
Lemma~\ref{lem:canflowzero}). 
It follows that 
\begin{eqnarray*}
\label{eq:secondcheck}
(u, u)_{\mathrm{Kup}}&=&
q^{-\ell(S)}\langle u^*u\rangle_{\mathrm{Kup}}\\
&=&q^{-\ell(S)}\left(q^{\ell(S)}+q^{\ell(S)}\sum_{J^{\prime}<J_u}c(S,J_u,J^{\prime})^2\right)\\
&=&1+\sum_{J^{\prime}<J_u}c(S,J_u,J^{\prime})^2.
\end{eqnarray*}
This finishes the proof that $(-,-)_{\mathrm{Lusz}}=(-,-)_{\mathrm{Kup}}$. 
\vskip0.5cm
Comparing the definition of the bilinear and the sesquilinear web form on 
$W^S$, 
we see that 
\[
(-,-)_{\mathrm{Kup}}=\overline{\langle -,\tilde{\psi}(-)\rangle}_{\mathrm{Kup}},
\]
just as in~\eqref{eq:bilinversussesquil} (recall that the elements 
in $B^S$ are $\tilde{\psi}$-invariant). 
Since $(-,-)_{\mathrm{Kup}}=(-,-)_{\mathrm{Lusz}}$, we see that condition 
\eqref{eq:Lusorthog2} is equivalent to the condition
\begin{equation}
\label{eq:Lusorthog3}
\langle q^{-\ell(\hat{S})}[Q_u],q^{-\ell(\hat{S})}[Q_v]\rangle_{\mathrm{Kup}}\in 
\delta_{u,v}+q\mathbb{Z}[q]
\end{equation}
for any $u,v\in B^S$. As already shown above, $\Phi$ maps 
$Y_u$ to $Q_u$ for any $u\in B^S$. Since $\Phi$ is an equivalence, 
$K_0^{\oplus}(\Phi)_{\bQ(q)}$ intertwiners the Euler forms. 
Therefore~\eqref{eq:almostortY} implies~\eqref{eq:Lusorthog3}, 
which we had already 
shown to be equivalent to~\eqref{eq:secondcondition}. 
\end{proof}

\begin{rem}
Just for completeness, let us summarize the decategorification of the 
above results. What they show is that the quantum skew Howe duality 
isomorphism $$\phi\colon V_{(3^k)}\to W_{(3^k)}$$
intertwines the respective $q$-sesquilinear forms and bar-involutions, 
and that it maps the canonical 
$\U$-basis to the dual canonical 
$\dot{\mathbf U}_q(\mathfrak{sl}_3)$-basis.
\end{rem}
\begin{rem}
It is also worth noting that Lemma~\ref{lem:grothendieck}, 
Proposition~\ref{prop:unitriang} and Theorem~\ref{thm:dualcan} 
imply that the change-of-basis matrix from 
Kuperberg's basis $B^S$ to the dual canonical basis of $W^S$ is unitriangular. 
\end{rem}

\section*{Appendix: Filtered and graded algebras and modules}
\label{sec:filt}
\setcounter{section}{5}
\setcounter{subsection}{1}
In this appendix, we have collected some basic facts about filtered algebras, 
the associated graded algebras and the idempotents in both. Our main sources 
are~\cite{sj} and~\cite{sr}. 
In this appendix, everything is defined over an arbitrary commutative 
associative unital ring $R$.  

Let $A$ be a finite dimensional, 
associative, unital $R$-algebra together with an increasing filtration of $R$-submodules
\[
\{0\}\subset A_{-p}\subset A_{-p+1}\subset\cdots\subset A_0\subset \cdots 
\subset A_{m-1}\subset A_m=A.
\]
Actually, for any $t\in \mathbb{Z}$ we have a subspace $A_t$, where 
we extend the filtration above by 
\[A_t=
\begin{cases}
\{0\},& \text{if}\quad t<-p,\\
A,&\text{if}\quad p\geq m.
\end{cases}
\]
Note that in the language of~\cite{sj}, such a filtration is \textit{discrete, 
separated, exhaustive and complete}. If $1\in A_0$ and the multiplication 
satisfies $A_iA_j\subseteq A_{i+j}$, we say that $A$ is an associative, unital, 
\textit{filtered algebra}.
\vspace*{0.25cm}

The \textit{associated graded algebra} is defined by 
\[
E(A)=\bigoplus_{i\in\mathbb{Z}} A_i/A_{i-1},
\] 
and is also associative and unital. Although $A$ and $E(A)$ are isomorphic 
$R$-modules, they are not isomorphic as algebras. 
\vspace*{0.25cm}

A finite dimensional, \textit{filtered $A$-module} is a finite dimensional,
unitary $A$-module $M$ with an increasing filtration of $R$-submodules 
\[
\{0\}\subset M_{-q}\subset M_{-q+1}\subset \cdots \subset M_t=M,
\]
such that $A_iM_j\subseteq M_{i+j}$, for all $i,j\in\mathbb{Z}$, after 
extending the finite filtration to a $\mathbb{Z}$-filtration as above. 
\vspace*{0.25cm}

We define the $t$-fold \textit{suspension} $M\{t\}$ of $M$, which 
has the same underlying $A$-module structure, but a new filtration defined by 
\[
M\{t\}_r=M_{r+t}.
\]
Given a filtered $A$-module $M$, the \textit{associated graded module} 
is defined by  
\[
E(M)=\bigoplus_{i\in\mathbb{Z}} M_i/M_{i-1}.
\]
An $A$-module map 
$f\colon M\to N$ is said to \textit{preserve the filtrations} if 
$f(M_i)\subseteq N_i$, for all $i\in\mathbb{Z}$. Any such map 
$f\colon M\to N$ induces a grading preserving $E(A)$-module map 
$E(f)\colon E(M)\to E(N)$ in the obvious way. 
\vspace*{0.25cm}

This way, we get a functor 
\[
E\colon A\text{-}\mathrm{\textbf{Mod}}_{\mathrm{fl}} \to E(A)\text{-}\mathrm{\textbf{Mod}}_{\mathrm{gr}},
\]
where $A\text{-}\mathrm{\textbf{Mod}}_{\mathrm{fl}}$ is the category of finite dimensional, filtered 
$A$-modules and filtration preserving $A$-module maps and $E(A)\text{-}\mathrm{\textbf{Mod}}_{\mathrm{gr}}$ 
is the category of finite dimensional, graded $E(A)$-modules and grading preserving $E(A)$-module maps. 
\vspace*{0.25cm}

Recall that $A\text{-}\mathrm{\textbf{Mod}}_{\mathrm{fl}}$ is not an abelian category, e.g. 
the identity map $M\to M\{1\}$ is a filtration preserving bijective $A$-module 
map, but does not have an inverse in $A\text{-}\mathrm{\textbf{Mod}}_{\mathrm{fl}}$.
\vspace*{0.25cm}

In order to 
avoid such complications, one can consider a subcategory with fewer morphisms. 
An $A$-module map $f\colon M\to N$ is called \textit{strict} if 
\[
f(M_i)=f(M)\cap N_i
\]
holds, for all $i\in\mathbb{Z}$. 
Let $A\text{-}\mathrm{\textbf{Mod}}_{\mathrm{st}}$ be the subcategory of filtered $A$-modules 
and strict $A$-module homomorphisms. 

\begin{lem}
The restriction of $E$ to $A\text{-}\mathrm{\textbf{Mod}}_{\mathrm{st}}$ is exact. 
\end{lem}

We also need to recall a simple result about bases. 
A basis $\{x_1,\ldots,x_n\}$ of a filtered algebra $A$ is called \textit{homogeneous} if, for each $1\leq j\leq n$, there exists an 
$i\in\mathbb{Z}$ such that $x_j\in A_i\backslash A_{i-1}$. In that case, 
$\{\overline{x}_1,\ldots,\overline{x}_n\}$ defines a homogeneous basis 
of $E(A)$, where $\overline{x_j}\in A_i/A_{i-1}$. In order to avoid cluttering 
our notation, we always write $\overline{x}_j$ and then specify in which 
subquotient we take the equivalence class by saying that it belongs to 
$A_i/A_{i-1}$.

Given a homogeneous basis $\{y_1,\ldots,y_n\}$ of the associated graded $E(A)$, we say that 
a homogeneous basis $\{x_1,\ldots,x_n\}$ of $A$ \textit{lifts} 
$\{y_1,\ldots,y_n\}$ if $\overline{x}_j=y_j\in A_i/A_{i-1}$ holds, for each $1\leq j\leq n$ 
and the corresponding $i\in\mathbb{Z}$. 
The result in the following lemma is well-known. However, we could not find 
a reference in the literature, so we provide a short proof here. 
\begin{lem}
\label{lem:homogbasis}
Let $A$ be a finite dimensional, filtered algebra and 
$\{y_1,\ldots,y_n\}$ be a homogeneous basis of $E(A)$. Then 
there is a homogeneous basis $\{x_1,\ldots,x_n\}$ of $A$ 
which lifts $\{y_1,\ldots,y_n\}$. 
\end{lem}
\begin{proof}
We prove the lemma by induction with respect to the filtration degree $q$. 
Suppose $A_q=0$, for all $q<-p$, and $A_q=A$, for all $q\geq m$. 
Then $E(A_{-p})=A_{-p}$. Since $\{y_1,\ldots,y_n\}$ is a homogeneous basis of $E(A)$, 
a subset of this basis forms a basis of $A_{-p}$. 

For each $-p+1\leq q\leq m$, choose elements in $A_q$ which lift the homogeneous subbasis 
of $E(A_q)$. We claim that the union of the sets of these elements, for all $-p\leq q\leq m$, 
form a homogeneous basis of $A$ which lifts $\{y_1,\ldots,y_n\}$. Call it $\{x_1,\ldots,x_n\}$. 
By definition, the $x_j$ lift the $y_j$, for all $1\leq j\leq n$. It remains to show that 
the $x_j$ are all linearly independent. This is true for $q=-p$, as shown above. 

Let $q>-p$ and suppose that the claim holds for $\{x_1,\ldots,x_{m_{q-1}}\}$, the subset of $\{x_1,\ldots,x_n\}$ 
which belongs to $A_{q-1}$. Let 
\[
\{x_1,\ldots, x_{m_q}\}=\{x_1,\ldots,x_{m_{q-1}}\}\cup \{x_{m_{q-1}+1},\ldots,x_{m_q}\}
\]
be the subset belonging to $A_q$. Suppose that  
\begin{equation}
\label{eqn:linind}
\sum_{j=1}^{m_q}\lambda_jx_j=0,
\end{equation}
with $\lambda_j\in R$. 
Then we have
\[
\sum_{j=1}^{m_q}\lambda_j\overline{x}_j=\sum_{j=m_{q-1}+1}^{m_q}\lambda_j\overline{x}_j=\sum_{j=m_{q-1}+1}^{m_q}\lambda_jy_j=0 \in A_q/A_{q-1}.
\]
By the linear independence of the $y_j$, this implies that $\lambda_j=0$, for all $m_{q-1}+1\leq j\leq m_q$. Thus, the linear 
combination in~\eqref{eqn:linind} becomes 
\[
\sum_{j=1}^{m_{q-1}}\lambda_jx_j=0.
\]
By induction, this implies that $\lambda_j=0$, for all $1\leq j\leq m_{q-1}$. 

This shows that $\lambda_j=0$, for all $1\leq j\leq n$, so the $x_j$ 
are linearly independent. 
\end{proof}
For a proof of the following proposition, see for example 
Proposition 1 in the appendix of~\cite{sr}.
\begin{prop}
\label{prop:Srid}
Let $M$ and $N$ be filtered $A$-modules and $f\colon M\to N$ a filtration 
preserving $A$-linear map. If $E(f)\colon E(M)\to E(N)$ is an isomorphism, 
then $f$ is an isomorphism (and therefore strict too). 
\end{prop} 
The most important fact about filtered, projective modules and their 
associated graded, projective modules, that 
we need in this paper, is Theorem 6 in~\cite{sj}. Note that these projective modules are the projective objects in the category $A\text{-}\mathrm{\textbf{Mod}}_{\mathrm{st}}$.
\begin{thm}
\label{thm:sjodin}(\textbf{Sj\"{o}din})
Let $P$ be a finite dimensional, graded, projective $E(A)$-module. Then 
there exists a finite dimensional, filtered, projective $A$-module $P'$, 
such that $E(P')=P$. Moreover, if $M$ is a finite dimensional, filtered, 
$A$-module, then any degree preserving $E(A)$-module map 
$P\to E(M)\{t\}$, for some grading shift $t\in\mathbb{Z}$, lifts to a 
filtration preserving $A$-module map $P'\to M\{t\}$. 
\end{thm} 

We also recall the following corollary of Sj\"{o}din (Corollary in~\cite{sj} after Lemma 20).

\begin{cor}
\label{cor:sjodin}
Let $M$ be a finite dimensional, filtered, $A$-module, then any finite or countable set of orthogonal idempotents in
\[
\mathrm{im}(\phi)\subset\mathrm{Hom}_{E(A)}(E(M),E(M))
\]
can be lifted to $\mathrm{Hom}_{A}(M,M)$, where $\phi$ is the natural transformation
\[
\phi\colon E(\mathrm{Hom}_{A}(M,M))\to \mathrm{Hom}_{E(A)}(E(M),E(M)).
\]
\end{cor}


\begin{thebibliography}{}
\bibitem{bn} D.~Bar-Natan, Khovanov's homology for tangles and cobordisms, Geom. Topol. 9 (2005), 1443-1499.
\url{https://arxiv.org/abs/math/0410495}
\bibitem{B-L-M} A.~Beilinson, G.~Lusztig and R.~MacPherson, A 
geometric setting for the quantum deformation of $\mathfrak{gl}_n$, Duke Math. J. 61--2 (1990), 655--677.
\bibitem{be} D.J.~Benson, {\em Representations and cohomology I}, Second edition. Cambridge 
Studies in Advanced Mathematics, 30. Cambridge University Press, Cambridge, 1998. xii+246 pp.
\bibitem{bru2} J.~Brundan, Centers of degenerate cyclotomic Hecke algebras and parabolic 
category $\mathcal{O}$, Represent. Theor. 12 (2008), 236--259.
\url{https://arxiv.org/abs/math/0607717}
\bibitem{bru} J.~Brundan, Symmetric functions, parabolic category $\mathcal{O}$ and the Springer fiber, Duke Math. J. 143 (2008), 41--79.
\url{https://arxiv.org/abs/math/0608235}
\bibitem{bk2} J.~Brundan and A.~Kleshchev, Blocks of cyclotomic Hecke algebras and Khovanov--Lauda algebras, Invent. Math. 178 (2009), 451--484.
\url{https://arxiv.org/abs/0808.2032}
\bibitem{bk} J.~Brundan and A.~Kleshchev, Graded decomposition numbers for cyclotomic Hecke algebras, Adv. Math. 222 (2009), 1883--1942.
\url{https://arxiv.org/abs/0901.4450}
\bibitem{bo} J.~Brundan and V.~Ostrik, Cohomology of Spaltenstein varieties, Transform. Groups 16 (2011), 619--648.
\url{https://arxiv.org/abs/1012.3426}
\bibitem{bs} J.~Brundan and C.~Stroppel, Highest weight categories 
arising from Khovanov's diagram algebra I: cellularity, Mosc. Math. J. 11--4 (2011), 685--722.
\url{https://arxiv.org/abs/0806.1532}
\bibitem{bs2} J.~Brundan and C.~Stroppel, Highest weight 
categories arising from Khovanov's diagram algebra II: Koszulity, Transform. Groups 15 (2010), 1--45.
\url{https://arxiv.org/abs/0806.3472}
\bibitem{bs3} J.~Brundan and C.~Stroppel, Highest weight 
categories arising from Khovanov's diagram algebra III: category $\mathcal{O}$, Represent. Theor. 15 (2011), 170--243.
\url{https://arxiv.org/abs/0812.1090}
\bibitem{bs4} J.~Brundan and C.~Stroppel, Highest weight 
categories arising from Khovanov's diagram algebra IV, the general linear supergroup, J. Eur. Math. Soc. 14 (2012), 373--419.
\url{https://arxiv.org/abs/0907.2543}
\bibitem{bs5} J.~Brundan and C.~Stroppel, Gradings on walled Brauer algebras and Khovanov's arc 
algebra, Adv. Math. 231 (2012), 709--773. 
\url{https://arxiv.org/abs/1107.0999}
\bibitem{caut} S.~Cautis, Clasp technology to knot homology via the affine Grassmannian, Math. Ann. 363 (2015), no. 3--4, 1053--1115.
\url{https://arxiv.org/abs/1207.2074}
\bibitem{ckm} S.~Cautis, J.~Kamnitzer and S.~Morrison, Webs and quantum skew Howe duality,  Math. Ann. 360 (2014), no. 1-2, 351--390.
\url{https://arxiv.org/abs/1210.6437}
\bibitem{cl} S.~Cautis and A.~Lauda, Implicit structures in $2$-representations of quantum groups, Selecta Math. (N.S.) 21 (2015), no. 1, 201--244.
\url{https://arxiv.org/abs/1111.1431}
\bibitem{ck} Y.~Chen and M.~Khovanov, An invariant of tangle cobordisms via subquotients of arc rings, Fund. Math. 225 (2014), no. 1, 23--44.
\url{https://arxiv.org/abs/math/0610054}
\bibitem{C-G} N.~Chris and V.~Ginzburg, {\em Representation theory and complex geometry}, Reprint of the 1997 edition. Modern Birkh\"auser Classics. Birkh\"auser Boston, Inc., Boston, MA, 2010. x+495 pp.
\bibitem{cr} J.~Chuang and R.~Rouquier, Derived equivalences for symmetric 
groups and $\mathfrak{sl}_2$-categorification, Ann. Math. 167 (2008), 245--298.
\url{https://arxiv.org/abs/math/0407205}
\bibitem{D-G} S.~Doty and A.~Giaquinto, Presenting Schur algebras, Int. Math. Res. Notices 36 (2002), 1907--1944.
\url{https://arxiv.org/abs/math/0108046}
\bibitem{fon} B.~Fontaine, Generating basis webs for $\mathfrak{sl}_n$, Adv. in Math. 229 (2012), 2792--2817.
\url{https://arxiv.org/abs/1108.4616}
\bibitem{fkk} B.~Fontaine, J.~Kamnitzer and G.~Kuperberg, Buildings, spiders, and geometric Satake, Compos. Math. 149 (2013), no. 11, 1871--1912.
\url{https://arxiv.org/abs/1103.3519}
\bibitem{fk} I.~Frenkel and M.~Khovanov, Canonical bases in tensor products 
and graphical calculus for $U_q(\mathfrak{sl}_2)$, Duke Math. J. 87 (1997), 409--480.
\bibitem{fkhk} I.~Frenkel, M.~Khovanov and A.~Kirillov Jr, Kazhdan--Lusztig polynomials and canonical basis, Transform. Groups 3-4 (1998), 321--336.
\url{https://arxiv.org/abs/q-alg/9709042} 
\bibitem{fu} W.~Fulton, {\em Young Tableaux}, London Mathematical Society Student Texts, 35. Cambridge University Press, Cambridge, 1997. x+260 pp. 
\bibitem{g} B.~Gornik, Note on Khovanov link cohomology.
\url{https://arxiv.org/abs/math/0402266}
\bibitem{ho1} R.~Howe, Remarks on classical invariant theory, Trans. Amer. Math. Soc. 313 (1989), 539--570.
\bibitem{ho2} R.~Howe, {\em Perspectives on invariant theory: Schur duality, multiplicity-free actions and beyond}, The Schur lectures, Isr. Math. Conf. Proc. 8, Tel Aviv (1992), 1--182.
\bibitem{hm} J.~Hu and A.~Mathas, Graded cellular bases for the 
cyclotomic Khovanov--Lauda--Rouquier algebras of type A, Adv. Math. 225--2 
(2010), 598--642.
\url{https://arxiv.org/abs/0907.2985}
\bibitem{hkh} R.~S.~Huerfano and M.~Khovanov, Categorification of some level two representations of 
$\mathfrak{sl}(n)$, J. Knot Theor. Ramif. 15-6 (2006), 695--713.
\url{https://arxiv.org/abs/math/0204333} 
\bibitem{kakash} S.J.~Kang and M.~Kashiwara, Categorification of highest weight modules via Khovanov--Lauda--Rouquier algebras, 
Invent. Math. 187--2 (2012), 1--44.
\url{https://arxiv.org/abs/1102.4677}
\bibitem{kh} M.~Khovanov, A functor-valued invariant of tangles, Algebr. Geom. Topol. 2 (2002), 665--741 (electronic). 
\url{https://arxiv.org/abs/math/0103190}
\bibitem{kh2} M.~Khovanov, Crossingless matchings and the cohomology of 
$(n,n)$ Springer varieties, Comm. Contemp. Math. 6-2 (2004), 561--577.
\url{https://arxiv.org/abs/math/0202110}
\bibitem{kv} M.~Khovanov, $\mathfrak{sl}_3$ link homology, Algebr. Geom. Topol. 4 (2004), 1045--1081.
\url{https://arxiv.org/abs/math/0304375}
\bibitem{kk} M.~Khovanov and G.~Kuperberg, Web bases for $\mathfrak{sl}_3$ are 
not dual canonical, Pacific J. Math. 188:1 (1999), 129--153.
\url{https://arxiv.org/abs/q-alg/9712046}
\bibitem{kl3} M.~Khovanov and A.D.~Lauda, A categorification of 
quantum $\mathfrak{sl}_n$, Quantum Topol. 2--1 (2010), 1--92.
\url{https://arxiv.org/abs/0807.3250}
\bibitem{kl4} M.~Khovanov and A.D.~Lauda, Erratum: 
``A categorification of quantum $\mathfrak{sl}_n$'', Quantum Topol. 2--1 (2011), 97--99.
\bibitem{kl1} M.~Khovanov and A.D.~Lauda, A diagrammatic approach 
to categorification of quantum groups I, Represent. Theor. 13 (2009), 309--347.
\url{https://arxiv.org/abs/0803.4121}
\bibitem{klms} M.~Khovanov, A.D.~Lauda, M.~Mackaay 
and M.~Sto\v{s}i\'c, {\em Extended Graphical Calculus 
for Categorified Quantum} $\mathfrak{sl}_2$,  Mem. Amer. Math. Soc. 219 (2012), no. 1029, vi+87 pp.
\url{https://arxiv.org/abs/1006.2866}
\bibitem{kr} M.~Khovanov and L.~Rozansky, Matrix factorizations and link homology, Fund. Math. 199--1 (2008), 1--91.
\url{https://arxiv.org/abs/math/0401268}
\bibitem{kx} S.~K\"{o}nig and C.~Xi, Cellular algebras: Inflations and 
Morita equivalences, J. Lond. Math. Soc. 60--3 (1999), 700--722.   
\bibitem{ku} G.~Kuperberg, Spiders for rank 2 Lie algebras,  Comm. Math. Phys. 180 (1996), no. 1, 109--151. 
\url{https://arxiv.org/abs/q-alg/9712003}
\bibitem{lam} T.~Y.~Lam, {\em Lectures on Modules and Rings}, Graduate Texts in Mathematics, 189. Springer-Verlag, New York, 1999. xxiv+557 pp.
\bibitem{lqr} A.D.~Lauda, H.~Queffelec and D.E.V.~Rose, Khovanov homology is a skew Howe 2-representation of categorified quantum 
$\mathfrak{sl}(m)$,  Algebr. Geom. Topol. 15 (2015), no. 5, 2517--2608.
\url{https://arxiv.org/abs/1212.6076}
\bibitem{lv} A.D.~Lauda and M.~Vazirani, Crystals from categorified quantum groups, Adv. Math. 228--2 (2011), 803--861.
\url{https://arxiv.org/abs/0909.1810} 
\bibitem{lu} G.~Lusztig, {\em Introduction to Quantum Groups}, Reprint of the 1994 edition. Modern Birkh\"auser 
Classics. Birkh\"auser/Springer, New York, 2010. xiv+346 pp. 
\bibitem{mack} M.~Mackaay, $\mathfrak{sl}(3)$-Foams and the Khovanov--Lauda 
categorification of quantum $\mathfrak{sl}(k)$. 
\url{https://arxiv.org/abs/0905.2059}
\bibitem{mack1} M.~Mackaay, The $\mathfrak{sl}(N)$-web algebras and dual canonical bases, J. Algebra 409 (2014), 54--100.
\url{https://arxiv.org/abs/1308.0566}
\bibitem{msv} M.~Mackaay, M.~Sto\v{s}i\'{c} and P.~Vaz, $\mathfrak{sl}_N$-link homology ($N\geq 4$) using foams and the 
Kapustin--Li formula. Geom. Topol. 13(2) (2009), 1075--1128.
\url{https://arxiv.org/abs/0708.2228}
\bibitem{msv2} M.~Mackaay, M.~Sto\v{s}i\'{c} and P.~Vaz, 
A diagrammatic categorification of the q-Schur algebra, Quantum Topol. 4--1 (2013), 1--75.
\url{https://arxiv.org/abs/1008.1348}
\bibitem{mv} M.~Mackaay and P.~Vaz, The universal $\mathfrak{sl}_3$-link homology, Algebr. Geom. Topol. 7 (2007), 1135--1169 (electronic).
\url{https://arxiv.org/abs/math/0603307}
\bibitem{mv2} M.~Mackaay and P.~Vaz, The foam and the matrix 
factorization $\mathfrak{sl}_3$ link homologies are equivalent, Algebr. Geom. Topol. 8 (2008), 309--342 (electronic).
\url{https://arxiv.org/abs/0710.0771}
\bibitem{my} M.~Mackaay and Y.~Yonezawa, The $\mathfrak{sl}(N)$ web categories. \url{https://arxiv.org/abs/1306.6242}
\bibitem{ma} A.~Mathas, {\em Iwahori--Hecke Algebras and Schur algebras of the Symmetric Group}, 
University Lecture Series, 15. American Mathematical Society, Providence, RI, 1999. xiv+188 pp. 
\bibitem{ms} V.~Mazorchuk and C.~Stroppel, A combinatorial approach 
to functorial quantum $\mathfrak{sl}(k)$ knot invariants, Am. J. Math. 131--6 (2009), 1679--1713. 
\url{https://arxiv.org/abs/0709.1971}
\bibitem{mn} S.~Morrison and A.~Nieh, On Khovanov's cobordism 
theory for $\mathfrak{su}_3$ knot homology, J. Knot Theor. Ramif. 17--9 (2008), 1121--1173. 
\url{https://arxiv.org/abs/math/0612754}
\bibitem{rob} L.-H.~Robert, A large family of indecomposable projective 
modules for the Khovanov--Kuperberg algebra of $\mathfrak{sl}_3$-webs, J. Knot Theory Ramifications 22 (2013), no. 11, 1350062, 25 pp. 
\url{https://arxiv.org/abs/1207.6287}
\bibitem{rob2} L.-H.~Robert, A characterisation of indecomposable web-modules over Khovanov--Kuperberg Algebras, 
Algebr. Geom. Topol. 15 (2015), no. 3, 1303--1362. 
\url{https://arxiv.org/abs/1309.2793}
\bibitem{rot} J.~Rotman, {\em An Introduction to Homological Algebra}, 2nd edition, Springer (2009).
\bibitem{rou} R.~Rouquier, 2-Kac--Moody algebras. \url{https://arxiv.org/abs/0812.5023}
\bibitem{ru} H.~Russell, An explicit bijection 
between semistandard tableaux and non-elliptic ${\mathfrak sl}_3$ webs,  J. Algebraic Combin. 38 (2013), no. 4, 851--862.
\url{https://arxiv.org/abs/1204.1037}
\bibitem{sj} G.~Sj\"{o}din, On filtered modules and their associated graded modules, Math. Scand. 33 (1973), 229--249.
\bibitem{sr} R.~Sridharan, Filtered algebras and representations of Lie algebras, Trans. Am. Math. Soc.100-3 (1961), 530--550.
\bibitem{s} C.~Stroppel, Parabolic category $\mathcal{O}$, 
perverse sheaves on Grassmannians, Springer fibres and Khovanov homology, Compos. Math. 145 (2009), 945--992. 
\url{https://arxiv.org/abs/math/0608234} 
\bibitem{sw} C.~Stroppel and B.~Webster, 2-Block Springer fibers: convolution algebras 
and coherent sheaves, Comment. Math. Helv. 87 (2012), 477--520.
\url{https://arxiv.org/abs/0802.1943}
\bibitem{ta} T.~Tanisaki, Defining ideals of 
the closures of the conjugacy classes and representations of the Weyl groups, Tohoku Math. J. 34 (1982), 575--585. 
\bibitem{ue} K.~Ueyama, Graded Frobenius algebras and 
quantum Beilinson algebras, Proc. of the 44th Symposium on Ring Theor. and Represent. Theor., Okayama University, Japan, 2012.
\bibitem{vv} M.~Varagnolo and E.~Vasserot, Canonical bases and Khovanov--Lauda algebras, 
J. Reine Angew. Math. 659 (2011), 67--100. \url{https://arxiv.org/abs/0901.3992}
\bibitem{we1} B.~Webster, Knot invariants and higher representation theory I: diagrammatic 
and geometric categorification of tensor products. \url{https://arxiv.org/abs/1001.2020}
\bibitem{we2} B.~Webster, Knot invariants and higher representation theory II: the 
categorification of quantum knot invariants. \url{https://arxiv.org/abs/1005.4559}
\bibitem{wu} H.~Wu, A colored $\mathfrak{sl}(N)$-homology for links in $S^3$, 
Dissertationes Math. (Rozprawy Mat.) 499 (2014), 217 pp. 
\url{https://arxiv.org/abs/0907.0695}
\bibitem{yo} Y.~Yonezawa, Quantum ($\mathfrak{sl}_n, \wedge V_n$) link invariant 
and matrix factorizations, Nagoya Math. J., Volume 204 (2011), 69--123. \url{https://arxiv.org/abs/0906.0220}
\vspace{0.1in}

\noindent M.M.: { \sl \small CAMGSD, Instituto Superior T\'{e}cnico, 
Lisboa, Portugal; Universidade do Algarve, Faro, Portugal} 
\newline \noindent {\small \textbf{email: mmackaay@ualg.pt}}

\vspace{0.01in}

\noindent W.P.: { \sl \small Courant Research Center ``Higher Order Structures'', University of G\"{o}ttingen, G\"{o}ttingen, Germany; 
Saint Mary's College of California, Moraga, USA} 
\newline \noindent {\small \textbf{email: wp1@stmary-ca.edu}}

\vspace{0.01in}

\noindent D.T.: { \sl \small Courant Research Center ``Higher Order Structures'', University of G\"{o}ttingen, G\"{o}ttingen, Germany; 
Mathematisches Institut, Georg-August-Universit\"{a}t G\"{o}ttingen, G\"{o}ttingen, Germany} 
\newline \noindent {\small \textbf{email: dtubben@uni-math.gwdg.de}}
\end{thebibliography}
\end{document}